\documentclass[12pt]{report}
\newcommand{\draftspaced}{\onehalfspaced\normalfont}

\usepackage{amssymb,amsfonts,amsthm,amscd}
\usepackage{latexsym}
\usepackage[all,arc]{xy}
\usepackage{enumerate}
\usepackage{mathrsfs}
\usepackage{dsfont}
\usepackage{verbatim}
\usepackage{nameref}
\usepackage[english]{babel}
\usepackage{amsmath}
\usepackage{chngcntr}
\usepackage{datetime}
\usepackage{blindtext}
\usepackage{array}
\usepackage[utf8]{inputenc}
\usepackage{setspace}
\usepackage{esvect}
\usepackage{tikz-cd}
\usepackage{amsmath,amscd}
\usepackage{mathtools}
\usepackage{stmaryrd}
\usepackage[utf8]{inputenc}
\usepackage{glossaries}
\usepackage{setspace}
\usepackage{CJKutf8}

\newcommand{\doublespaced}{\renewcommand{\baselinestretch}{2}\normalfont}
\newcommand{\singlespaced}{\renewcommand{\baselinestretch}{1}\normalfont}
\newcommand{\onehalfspaced}{\renewcommand{\baselinestretch}{1.5}\normalfont}

\newtheorem{thm}{Theorem}[chapter]
\newtheorem{cor}[thm]{Corollary}
\newtheorem{prop}[thm]{Proposition}
\newtheorem{lem}[thm]{Lemma}
\newtheorem{conj}[thm]{Conjecture}
\newtheorem{quest}[thm]{Question}

\theoremstyle{definition}
\newtheorem{defn}[thm]{Definition}

\newtheorem{exs}[thm]{Examples}

\theoremstyle{remark}
\newtheorem{rmk}[thm]{Remarks}
\numberwithin{equation}{chapter}
\counterwithout{equation}{chapter}

\newcommand{\ZZ}{\mathbb{Z}}
\newcommand{\QQ}{\mathbb{Q}}
\newcommand{\RR}{\mathbb{R}}
\newcommand{\CC}{\mathbb{C}}
\newcommand{\TT}{\mathbb{T}}
\newcommand{\SSS}{\mathbb{S}}
\newcommand{\PP}{\mathbb{P}}
\newcommand{\MM}{\mathcal{M}}
\newcommand{\BB}{\mathcal{B}}
\newcommand{\LL}{\mathcal{L}}
\newcommand{\OO}{\mathcal{O}}
\newcommand{\WW}{\mathcal{W}}
\newcommand{\eqt}{\begin{equation}}
\newcommand{\eqtn}{\end{equation}}
\newcommand{\trieq}{\triangleq}
\newcommand{\4}{\text{ , for }}

\def\theauthor{CHUNG, Shun Wai}

\begin{document}

\pagenumbering{roman}
%\doublespaced
\large\newlength{\oldparskip}\setlength\oldparskip{\parskip}\parskip=.3in
\thispagestyle{empty}
\begin{center}

\vspace*{\fill}

\Huge\textbf{SYZ Mirror Symmetry \\ for \\ Dirichlet Branes}

\vspace*{1.5cm}

\large\theauthor

\vspace*{1.5cm}

%\normalsize A DISSERTATION in Mathematics

\large
\normalsize
\singlespaced
A Thesis Submitted in Partial Fulfillment \\
of the Requirements for the Degree of \\
Master of Philosophy \\
in\\
Mathematics

\vspace{4.5cm}

\normalsize
The Chinese University of Hong Kong \\
July 2016

\vspace{\fill}

\end{center}

\begin{comment}

%%%-------------------------------Thesis Committee----------------------------------------%%%

\chapter*{Thesis Committee}

\vspace{3cm}

\begin{center}

\noindent
\normalsize\underline{\hspace{40pt} Professor LEUNG Nai Chung \hspace{40pt}}\\
\footnotesize Chair \hspace*{190pt} \\

\vspace{13pt}

\noindent
\normalsize \underline{\hspace{45pt} Professor CHAN Kwok Wai \hspace{45pt}}\\
\footnotesize Supervisor \hspace*{165pt} \\

\vspace{13pt}

\noindent
\normalsize \underline{ \hspace{50pt}Professor WU Zhongtao \hspace{50pt}}\\
\footnotesize  Committee Member \hspace*{126pt} \\

\vspace{13pt}

\noindent
\normalsize\underline{\hspace{60pt}Professor LI Wei Ping \hspace{60pt}}\\
\footnotesize  External Examiner \hspace*{130pt} \\

\end{center}

%%%-------------------------------Thesis Committee----------------------------------------%%%
\end{comment}

\chapter*{ }
\draftspaced
\begin{center}

\vspace{20pt}

\Large\underline{\textbf{Thesis/Assessment Committee}}

\vspace{10pt}\normalsize%\large

Professor LEUNG Nai Chung (Chair)\\
Professor CHAN Kwok Wai (Thesis Supervisor)\\
Professor WU Zhongtao (Committee Member)\\
Professor LI Weiping (External Examiner)

\end{center}

\thispagestyle{empty}

\newpage
\doublespaced
\singlespaced
\draftspaced

\chapter*{Abstracts}

\normalsize

\hspace{13pt} In this thesis, we study a class of special Lagrangian submanifolds of toric Calabi-Yau manifolds and construct their mirrors using some techniques developed in the SYZ programme.

It is conjectured in Aganagic-Vafa \cite{Aganagic Vafa} that such pair of Lagrangian subspace and Calabi-Yau manifold is mirror to a certain complex subvariety. It is also conjectured in Ooguri-Vafa \cite{Ooguri Vafa} that the counting of holomorphic discs is encoded in the mirror side, namely, by equating the open Gromov-Witten generating function with the Abel-Jacobi map on the mirror side up to a change of variables by the mirror map.

We present a justification on the conjecture on the mirror construction of D-branes in Aganagic-Vafa\cite{Aganagic Vafa}. We apply the techniques employed in Chan-Lau-Leung \cite{Chan Lau Leung} and Chan\cite{Chan}, which give the SYZ mirror construction for D-branes. Recently, Chan-Lau-Leung\cite{Chan Lau Leung} has given an explicit mirror construction for toric Calabi-Yau manifolds as complex algebraic varieties. Inspired by the SYZ transformation of D-branes in Leung-Yau-Zaslow\cite{Leung Yau Zaslow}, Chan\cite{Chan} has defined a generalised mirror construction for a larger class of Lagrangian subspaces, which is important to our situation.

More specifically, we investigate mirrors of D-branes of Aganagic-Vafa type. According to a recent paper, Fang-Liu \cite{Fang Liu} has given a proof to the mirror conjecture on the disc counting problem and Cho-Chan-Lau-Tseng \cite{Cho Chan Lau Tseng} has described explicitly the enumerative meaning of the inverse mirror map. These recent results enable us to derive explicit formulas for such mirror branes.

The results point towards a need for further quantum correction of the SYZ transform for Lagrangian subspaces in order to recover physicist’s mirror prediction of counting holomorphic discs in Ooguri-Vafa \cite{Ooguri Vafa}. When compare to physicists prediction of the Aganagic-Vafa mirror branes, such naive SYZ transform fails to coincide with the prediction. The main reasons lies in the fact that the SYZ transformation does not capture open Gromov-Witten invariants of the A-brane. Through the special case of Aganagic-Vafa A-branes, we are able to modify the SYZ transform that involves further quantum correction using discs counting of the A-brane. However it is not known how such a modified SYZ transform can be constructed and works in general. We believe that the explicit formulas and computations in the special case we consider here can give hints to the correct modification in general in future development of the SYZ programme.

\begin{CJK}{UTF8}{bkai}

\chapter*{摘要}
\textsf{}
%\normalsize
\large

在本論文中，我們研究環面 Calabi-Yau 流形中的某類特殊 Lagrangian 子流形，我們更使用 SYZ 最近期的發展和技術來構建它的鏡對稱模型。

跟據 Aganagic-Vafa\cite{Aganagic Vafa} 的猜想，這對 Lagrangian 子空間和 Calabi-Yau 流形的鏡子模型能表示成某個代數複簇。另外，物理學家 Ooguri-Vafa\cite{Ooguri Vafa}　也提出了一個關於鏡對稱的推測，就是說全純光盤的計算，可通過它的開放 Gromov-Witten 函數與鏡子中的 Abel-Jacobi 函數的等式，而得知其數據。

我們會提出 Aganagic-Vafa 鏡膜猜想\cite{Aganagic Vafa}的數學跟據。以 D-膜的 SYZ 鏡建設概念，我們會使用於 Chan， Lau 和 Leung \cite{Chan Lau Leung,Chan}　中的技巧。 在近期的鏡對稱發展，Chan， Lau 和 Leung\cite{Chan Lau Leung} 對於環面 Calabi-Yau 代數複簇提出了一個明確的 SYZ 鏡面構作。 透過 Leung， Yau，Zaslow 為D-膜進行 SYZ 轉型 \cite{Leung Yau Zaslow} 所啟發，Chan\cite{Chan}定義了一個較為廣義的D-膜鏡模型構作，為更大類的 Lagrangian 子空間構作鏡對稱模型，這技巧對於本文章為非常重要。

更深入地，我們研究 Aganagic-Vafa 類的鏡膜。根據最近刊登的數學文章，Fang 和 Liu\cite{Fang Liu} 對於全純光盤的計算問題提出了證明，另一方面，Cho， Chan，Lau 和 Tseng\cite{Cho Chan Lau Tseng} 又運算了鏡子函數的反函數是能以全純光盤的數量表示出來。最後，我們運用這些結果，得出了Aganagic-Vafa 鏡膜代數複簇的定義公式。

以上的結論指出我們應對 SYZ 的 D-膜構作作進一步的量子修正，其構作應與物理學家 Ooguri-Vafa 對其 D-膜的推測的結果為一樣。當與物理學家對 Aganagic-Vafa 鏡膜的預測進行比較時 ，未修正的 SYZ 鏡膜構作未能與物理學家所預測一致。主要的原因在於該 SYZ 轉型未有包含A-膜的開放 Gromov-Witten 不變量數據在其中。在 Aganagic-Vafa A-膜這特殊情況下，我們可以修改 SYZ 鏡膜構作，透過涉及到 A-膜全純光盤的計算來對 SYZ 鏡膜構作作進一步量子修正。雖然現未能得知到，用這方法去修正的 SYZ 鏡膜構作能否在更一般情況下仍然能成功運用。我們認為，在本文章中所討論的 Aganagic-Vafa D-膜計算中，應能為 SYZ 未來的發展給予一些提示及指引如何去正確地修改一般情況下 SYZ 的 D-膜構作。

%對以收回 Ooguri-Vafa 全純計數盤的物理學家的預測鏡需要的 SYZ 鏡建設 Lagrangian 子空間的進一步量子修正結果百分點引用 Ooguri-Vafa 。

\end{CJK}

\chapter*{Acknowledgments}
\normalsize
\hspace{13pt} I wish to express my greatest gratitude to my supervisor Professor Chan Kwok Wai for his support and guidance throughout my post-graduate study. I am very grateful that he gave me this thesis problem to work on which enables me to dig deep into the amazing world of mirror symmetry. Also, I want to express my appreciation for his time and effort in teaching me and assisting my study. Nevertheless, I am in debt for his invaluable advice and guidance in discussing my future development in life.

I would like to thank a number of great teachers other than my supervisor who ignited my curiosity in mathematic, includes Professor Conan Leung, Professor Thomas Au, Professor Leung Chi Wai, Professor Mok Ngai Ming, Professor Ronald Lui, Professor Lau Siu Cheong and Doctor Cheung Leung Fu. I thank these teachers for their effort in preparing teaching materials and their kind response in answering students questions about future career development. Without these teachers, I will still be thinking like a kid without knowing my potential, talents and direction.

Also, I would like to thank my friends in CUHK for making my life enjoyable and fulfilling in my post-graduation years. They are Marco Suen, Lee Man Chun, Huang Shaochuang , Rex Lo,  Dr. Hansol Hong, Dr. Matthew Young, Matthew Man, Tang Sheung Ho, Patrick Kwok, Steve Hui,  Wong Ka Ho, Ken Lee, Chen Yuen, Zhang Yi and Jiang Qing Yuan. With their support and accompany, we are able to work hard together as researchers and grow mature in life as men(women).

Finally, and most importantly, I would like express my deepest appreciation to my parents, who are my first, and most important supporters of all time. Without their endless love and support, I would not be who I am and this master thesis would not have been completed.

\singlespaced
\tableofcontents

\newpage
\draftspaced
\pagenumbering{arabic}
\include{introdept}
\include{back}
\include{finitedept}
\include{infinitedept}

\doublespaced
\singlespaced
\draftspaced
\normalsize

\chapter{Introduction}

\section{Background}
Mirror symmetry is a surprising duality between pairs of Calabi-Yau manifolds which began as a phenomenon in string theory since 1980's. This suggests that Calabi-Yau manifolds should come in pairs. According to string theory, the world we are living in is a 10-dimensional manifold of the form $\RR^{3,1}\times X^{6}_{CY}$ where $X^{6}_{CY}$ is a compact Calabi-Yau manifold of complex dimension 3. String theorists also suggest that physical phenomena in
$\RR^{3,1}\times X$ is exactly the same as that in $\RR^{3,1}\times X^\vee$ if $X$ and $X^\vee$ are mirror to each other. In other words, computations on one-side can be done on the mirror side which makes mirror symmetry a fundamental and powerful computational tool in string theory. Moreover, such mirror pairs $(X,X^\vee)$ exhibit an exchange of Hodge numbers, that is, $h^{1,1}(X)=h^{1,2}(X^\vee)$ and $h^{1,2}(X)=h^{1,1}(X^\vee)$. Using string theory as guidelines, geometers would like to understand the philosophy and the mathematics underlying such unexpected phenomenon of mirror symmetry.

There are two principal approaches towards understanding the geometry underlying mirror symmetry, namely, Kontsevich's homological mirror symmetry (HMS) conjecture \cite{Kontsevich HMS} in 1994 and the Strominger-Yau-Zaslow (SYZ) conjecture \cite{SYZ} in 1996.

We briefly describe the statement of HMS conjecture which is expressed in the language of categories. Suppose we have a pair of Calabi-Yau manifolds $(X,\check{X})$ which are mirror to each other. Let $X$ be equipped with a symplectic structure $\omega$ and a complex structure $J$, and $\check{X}$ with $(\check{\omega},\check{J})$. Mirror symmetry suggests that $\omega$ should determine $\check{J}$ and vice versa $\check{\omega}$ should determine $J$. Therefore Kontsevich proposed the HMS conjecture that the symplectic geometric data of $(X, \omega)$ and complex geometric data of $(\check{X},\check{J})$ are exchanged under an equivalence of categories as follows:

\begin{conj}[HMS]
There is an equivalence of (triangulated) categories between the derived Fukaya category $\mathcal{DF}uk(X)$ of the symplectic manifold $X$ with symplectic structure $\omega$
%$(X,\omega)$ Lagrangians submanifolds $\mathcal{F}uk(X,\omega)$
and the derived category $\mathcal{D}^b\mathcal{C}oh(\check{X})$ of the complex manifold $\check{X}$ with complex structure $\check{J}$, that is,
$$\mathcal{DF}uk(X,\omega) \simeq \mathcal{D}^b\mathcal{C}oh(\check{X},\check{J}).$$
%$(\check{X},\check{J})$. of bounded complexes of coherent sheaves $\mathcal{D}^b(\check{X},\check{J})$.
\end{conj}

On the other hand, the SYZ conjecture describes the phenomenon of mirror symmetry via more geometric construction. In 1996, inspired by new development in string theory, Strominger, Yau and Zaslow proposed that mirror symmetry is nothing but a $\TT$-duality, short for \textit{duality of torus}. Comparing to the HMS conjecture, the SYZ conjecture is a more constructive and more geometric approach towards understanding the geometry of mirror symmetry.

\begin{conj}[SYZ]
Both $X,\check{X}$ admit special Lagrangian torus fibrations over the same base $B$:
$$\begin{array}[c]{ccccc}
X&&&&\check{X}\\
&\mu\searrow&&\swarrow\check{\mu}\\
&&B
\end{array}$$
such that regular fibers $\mu^{-1}(b)\subset X$ and $\check{\mu}^{-1}(b)\subset \check{X}$ over the same base $b\in B$ are dual tori. Moreover, there exist a Fourier-type transformation which exchanges symplectic geometric and complex geometric data between $X$ and $\check{X}$.
\end{conj}

Dirichlet branes ( simply D-branes ) are important objects well-studied by string theorists. In string theory, as suggested by its name, point particles in classical physics are now viewed as strings of energy propagating through our ten dimensional world $\RR^{3,1}\times X$. As we consider open strings, its endpoints are required to lie on a D-brane. D-brane refers to the boundary condition that the endpoints have to satisfy which is also known as the Dirichlet boundary condition. Mathematically speaking, D-branes are special Lagrangian submanifolds in the A-model of topological string theory which we call A-branes. In the B-model of topological string theory,  D-branes are complex submanifolds of a Calabi–Yau together with additional data that arise physically from having charges at the endpoints of strings which we denote it as B-branes. In the language of HMS, A-branes are objects in the Fukaya category and B-branes are objects in the derived category of coherent sheaves. On the other hand, in the setting of SYZ conjecture, we expect to have a Fourier-type transformation $\mathcal{F}_{SYZ}$ which carries an A-brane on one side to a B-brane on the mirror side as described in \cite{Leung Yau Zaslow,Arinkin Polishchuk,Chan}.

We now describe the problem that motivates the work of this thesis. We all agree that to reach proofs of both conjectures are still miles and miles away. However, in some special examples, such as the situation considered by Aganagic and Vafa in \cite{Aganagic Vafa,Fang Liu}, we can ask whether these conjectures can be verified. With the SYZ transform $\mathcal{F}_{SYZ}$ explicitly defined as in \cite{Leung Yau Zaslow,Arinkin Polishchuk,Chan} and a clear understanding of the mirrors of Aganagic-Vafa A-branes $L_{AV}$ as in \cite{Aganagic Vafa,Fang Liu}, a natural question is raised:
\begin{quest}
Does the SYZ transformation $\mathcal{F}_{SYZ}$ transforms the Aganagic-Vafa A-brane $L_{AV}$ into the desired mirror B-brane that matches the physicists' prediction?
\end{quest} In this thesis, we will answer this question and give a comparison between explicit expressions of the (naive) SYZ mirror and the predicted mirror. We first begin with an overview of the main results.

\section{Main results}
In this thesis, three main results will be presented. Throughout the thesis, we apply the SYZ principles to certain A-model $X$ and A-brane $L$. By A-models and A-branes, we mean symplectic manifolds and Lagrangian submanifolds respectively. Moreover, we give special attention to Aganagic-Vafa A-branes in toric Calabi-Yau threefolds as in \cite{Aganagic Vafa, Fang Liu, Lerche Mayr Warner}. And we are interested in the mirror construction of A-branes via the SYZ proposal which we expect to recover physicists' predictions.

First of all, we survey the SYZ construction of the B-model $X^\vee$ as a variety over the complex number field $\CC$ to obtain our first result which is the computation of $X^\vee$ via the SYZ construction as in \cite{Chan Lau Leung}. The SYZ proposal suggests that mirror symmetry is nothing but torus duality. By applying SYZ to toric Calabi-Yau manifolds $X_\Delta$ equipped with the Gross fibration $\mu_G$, we obtain the semi-flat mirror as some moduli $\MM_0$ of Lagrangians fibers $L_G$ with a connection $\nabla$ together with a naturally attached complex structure, the semi-flat complex structure $J_0$. To construct the mirror $X^\vee$, we modify $(\MM_0,J_0)$ via symplectic information of the A-model and A-brane, namely the open Gromov-Witten invariants $n_\beta$'s for $\beta\in H_2(X,L;\ZZ)$. After all the procedures of the SYZ program, we obtain the desired mirror $X^\vee$ as a complex variety defined by the following equation:
$$ uv=W(z_0,...,z_{n-1}):=(1+\delta_0)+\sum_{i=1}^{n-1}(1+\delta_i)z_i+\sum_{i=n}^{m-1}(1+\delta_i)Q_{i-n+1}\prod_{j=1}^{n-1}z_j^{<v_j^*,v_i>}$$
where $\delta_0,...,\delta_{m-1}$ are open Gromov-Witten generating functions of SYZ fibers defined in \eqref{definition of delta i}, $Q_1,...,Q_{m-n}$ are the complexified K$\ddot{a}$hler parameters and the $v_0,...,v_{m-1}$ are the rays that generate the fan $\Sigma=\Sigma(\Delta)$.

\begin{rmk}
According to the physical literature \cite{Hori Iqbal Vafa}, it is predicted based on physical consideration that the mirror of the toric Calabi-Yau manifold $X_\Delta$ is same as what we have computed above except that they ignored all the open Gromov-Witten generating function $\delta_0,...,\delta_{m-1}=0$ although the physicists realized that "quantum correction" terms are needed. According to the SYZ consideration, we now understand how the "quantum correction" is expressed in terms of open Gromov-Witten invariants of the A-side which is given explicitly by the generating functions $\delta_0,...,\delta_{m-1}$.
\end{rmk}

Secondly, we apply the SYZ transformation as appeared in \cite{Chan} to certain A-branes $(L,\nabla)$ of $X^n_\Delta $ to obtain the second result that the mirror $L_{SYZ}^\vee$ can be computed explicitly which is stated in Theorem \ref{main thm B} and Definition \ref{main def B}. Suppose that we now focus on certain A-branes $(L,\nabla)$ in $X^n_\Delta$ which is constructed using charges $l^{(1)},...,l^{(k)}\in \ZZ^m$ together with some constants $exp(c^{(1)}+i\phi^{(1)}),...,exp(c^{(k)}+i\phi^{(k)})\in \CC^\times$. In Chan \cite{Chan}, the author gave a definition of the SYZ transformation of A-branes which generalized that of Leung, Yau and Zaslow in \cite{Leung Yau Zaslow}. We directly apply such SYZ transformation to our A-brane $(L,\nabla)$ and we obtain that the (naive) SYZ mirror B-brane as a subvariety in $X^\vee$ which is given by the following equations:
$$ \prod_{j=1}^{m-1} z_j^{-l^{(a)}_j} %= e^{c^{(a)}+i\phi^{(a)}}
=exp(c^{(a)}+i\phi^{(a)}) \4 a=1,...k $$
where $z_i=z_i(z_1,...,z_{n-1})=Q_{i-n+1}\prod_{j=1}^{n-1} z_j^{<v_j^*,v_i>}$ for $i=n,...,m-1$, and $l^{(a)}=(l^{(a)}_0,...,l^{(a)}_{m-1})\in\ZZ$, $\forall a=1,...k$. %$l^{(a)}=(l^{(a)}_0,...,l^{(a)}_{m-1})\in\ZZ$, $\forall a=1,...k$ and
%$z_i=z_i(z_1,...,z_{n-1})=Q_{i-n+1}\prod_{j=1}^{n-1} z_j^{<v_j^*,v_i>}$ for $i=n,...,m-1$.
%\begin{align*} \vec{l}^{(a)}&=(l^{(a)}_0,...,l^{(a)}_{m-1})\in\ZZ \4 a=1,...k \\   z_i&=z_i(z_1,...,z_{n-1})=Q_{i-n+1}\prod_{j=1}^{n-1} z_j^{<v_j^*,v_i>}  \4 i=n,...,m-1 \end{align*}

%However, this SYZ mirror subvriety does not necessarily match with physicists prediction which we will see in the special case of Aganagic-Vafa A-branes. For this reason we will call it the naive SYZ mirror brane and denote it by $L_{naive}^\vee$. The key reason for the failure is due to the fact that the SYZ construction have not yet capture the information of the counting discs of the A-brane itself.

\begin{rmk}
According to the physical literature \cite[Chapter 3.1]{Aganagic Vafa}, it is expected and predicted that the mirror brane $L^\vee$ for certain Lagrangian A-brane $L\subset X_\Delta$ is exactly defined by the above equations which we derived based on SYZ consideration. In other words, the SYZ construction has verified physicists' predictions for such A-branes. Moreover, the term $exp(i\phi^{(a)})\in U(1)$ is considered as a phase term which can be combined with $exp(c^{(a)})$ to give it an imaginary part. However, the SYZ consideration shows us that the phase term $exp(i\phi^{(a)})$ is not arbitrary and has a geometric meaning which represents a choice of a flat $U(1)$-connection $\nabla$ over $L$. Furthermore, the SYZ construction provide us more information of the mirror than we need, that is a holomorphic $U(1)$-connection $\check{\nabla}$ over $L^\vee$ which we have not used yet in this thesis.
\end{rmk}

\begin{rmk}
Although the defining equations of the SYZ mirror subvariety recovers the physicists prediction, such SYZ mirror subvriety does not necessarily give the correct mirror which we will see in the special case of Aganagic-Vafa A-branes. For this reason we will denote it by $L_{naive}^\vee$, the naive SYZ mirror brane. The key reason for the failure is due to the fact that the SYZ construction have not yet captured the information of the holomorphic discs bounded by the A-brane itself.
\end{rmk}

Thirdly, we examine the defining equations of the well-known mirror B-brane $L^\vee$ of Aganagic-Vafa A-branes $(L,\nabla)$ in toric Calabi-Yau A-model $X^3_\Delta$ of complex dimension three as described in \cite{Aganagic Vafa} and obtain our third result by re-expressing them using just the data of the given toric Calabi-Yau A-model $X_\Delta^3$ and A-brane $(L,\nabla)$ which is stated in Theorem \ref{main theorem Lv=z1,z2}. In Fang and Liu's recent paper \cite{Fang Liu}, the authors successfully proved the mirror prediction described in \cite{Ooguri Vafa,Lerche Mayr Warner,Aganagic Vafa}. On the other hand, Cho,Chan, Lau and Tseng has proved in \cite{Cho Chan Lau Tseng} that the SYZ mirror map coincide with the inverse mirror map which encodes the counting of discs bounded by SYZ fibers. According to physicists prediction in \cite{Aganagic Vafa}, the mirror B-brane $L^\vee$ is given by the 2 equations:
$$z_1=q_0 \ \ \ and \ \ \ z_2=z_2(q_0)$$
where $q_0\in \CC^\times$ is the open complex structure parameter and $z_2(-)$ is the inverse function of the mirror curve equation $W(z_1,z_2)$ such that $W(q_0,z_2(q_0))=0$ for any $q_0\in \CC^\times$. By combining the results of Fang-Liu\cite{Fang Liu} and Cho-Chan-Lau-Tseng\cite{Cho Chan Lau Tseng}, we compute in terms of A-side data that the mirror B-brane $L^\vee$ is given by the following 2 equations:
$$z_1=Q_0(1+O(Q_1,...,Q_{m-n})) \ \ \ and \ \ \ z_2=1+O(Q_1,...,Q_{m-n})$$
where $Q_0\in \CC^\times$ is the open K$\ddot{a}$hler parameter and $Q_1,...,Q_{m-n}\in \CC^\times$ is the closed K$\ddot{a}$hler parameters. We refer to Theorem \ref{main theorem Lv=z1,z2} for the explicit form of the series. We observe that the coefficients of the higher order terms involves data of the open Gromov-Witten invariants which is desired since we would like to know how the B-brane $L^\vee$ depends on the invariants of $(X_\Delta,L,\nabla)$ on the A-side.

\begin{rmk}
The mirror $L^\vee$, being a B-brane subvariety, is expressed in terms of B-side information which is the complex structure parameters $q_0,...,q_k$ and the mirror curve equation $W\in\CC[z_1^\pm,z_2^\pm]$. In recent development of mirror symmetry, we are able to express $L^\vee$ solely in terms of the A-side information. The advantage is two-fold. We now understand how the B-brane $L^\vee$ depends on deformation of the triplet $(X_\Delta,L,\nabla_L)$ on the A-side. In other words, we are able to express immediately the defining equation of the B-brane $L^\vee$ once we are given all the open Gromov-Witten invariants $n_\beta$ and (open/closed) K$\ddot{a}$hler parameters $Q_0,...,Q_{m-n}$ of certain $(X_\Delta,L,\nabla_L)$. And also, we are now able to compare $L^\vee$ with SYZ transformation of $(L,\nabla)$ since we notice that the (naive) SYZ mirror brane is expressed in terms of A-side information only. Hence we can verify or modify the SYZ transformation accordingly.
\end{rmk}

Finally, we are able to modify the naive SYZ transformation by introducing further quantum correction to the SYZ mirror brane construction using the discs counting data of the A-brane. With the defining equations of the predicted Aganagic-Vafa mirror brane and the SYZ mirror brane at hand, we are able compare them directly. Unfortunately, they do not coincide as we have expected. However, we observe that their difference is not too far apart as we compare their defining equations. More precisely, they coincide up to higher order term. Combing all the results, we are now ready to define the SYZ mirror construction for D-branes with further quantum correction which involves the discs counting data of the Aganagic-Vafa A-brane. By construction, such 'further quantum corrected' SYZ transformation works well on Aganagic-Vafa type A-branes. More generally, it would be fascinating to investigate how well such transformation works in more general cases. This issue will be addressed in more detail in the last chapter of the thesis together with some comments on future development.

%This issue will be addressed in more detail in the last chapter of the thesis together with some comments on future development.

% \section{Main results ( draft )}	

%\begin{thm}	Suppose that $X_\Delta$ is a toric Calabi-Yau manifold. Motivated by the SYZ conjecture, its mirror is given by the polynomial	\eqt uv:=(1+\delta_0)+\sum_{i=1}^{n-1}(1+\delta_i)z_i+\sum_{i=n}^{r-1}(1+\delta_i)q_i\prod_{j=1}^{n-1}z_j^{<v_j^*,v_i>}.\eqtn	\end{thm}

%\begin{thm}	Under SYZ mirror symmetry, generic A-branes $(L,\nabla)$ is mirror to some B-branes of the form	\begin{equation}	\{z^{-q}=exp(c+i\phi)\}		\end{equation}	as predicted by physicists.	\end{thm}

%\begin{thm}	For any Aganagic-Vafa A-brane, its mirror B-brane is given by	\begin{align} \begin{cases}	z_1&=q_0(Q)=(...n_{\beta}...)\\	z_2&=exp(-\Sigma_\beta n_\beta [\partial \beta] Q^\beta )	\end{cases} \end{align}	Moreover, this pair of D-branes satisfies the mirror prediction, that is,	\begin{align}	\mathcal{F}(Q)\underset{map}{\overset{Mirror}\longleftrightarrow}\mathcal{W}(q)	\end{align}	\end{thm}

\section{Organization}
The thesis is organised as follows: In the first chapter, we briefly review the background of mirror symmetry and overview the main results. In the second chapter, we survey SYZ construction for toric Calabi-Yau manifolds based on the paper \cite{Chan Lau Leung}. In the third chapter, we introduce the SYZ transformation for A-branes and apply it to certain class of A-branes which was studied in \cite{Aganagic Vafa} to recover physicist prediction. In the forth chapter, %we explicitly compute the defining equation of the Aganagic-Vafa mirror B-brane solely in terms of data from the A-side by using some important results from \cite{Cho Chan Lau Tseng} and \cite{Fang Liu}. More importanlty,
we compute and compare the predicted mirror of Aganagic-Vafa A-branes and its naive SYZ transform and therefore we naturally modify the naive SYZ transform by introducing further quantum correction. Finally, in the last chapter, we give some possible direction for future development.
%\section{Acknowledgments}
%The author is grateful to his supervisor Professor Kwokwai Chan for continuous encouragement and support. Discussion with Marco Suen, Ken Lee, Kaho Wong and other colleagues on various subjects were very inspiring and enjoyable, and the author would like to thank them all for help and support.
%%%%%%%%%%%%%%%%%%%%%%%%%%%%%%%%%%%%%%%%%%%%%%%%%%%%%%%%%%%%%%%%%%%%%%%%%%%%%%%%%%%%%%%%%%%%%%%%%%
%-------------------------------------------------------------------------------------------------

\chapter{SYZ mirror symmetry for toric Calabi-Yau manifolds}
Let $X$ be a toric Calabi-Yau manifold. The goal of this chapter is to give a review on the construction of its mirror $X^\vee$ in Chan-Lau-Leung \cite{Chan Lau Leung} as a variety over $\CC$ which is motivated by the SYZ conjecture. According to the famous paper "Mirror Symmetry is T-duality" \cite{SYZ} by Strominger, Yau and Zaslow in 1994, the SYZ conjecture suggests that mirror symmetry is nothing but a duality of torus. First of all, we discuss the basic idea behind the SYZ program which was well surveyed and clearly illustrated by Auroux in \cite{Auroux} . Then we construct and study the properties of the Gross fibration $\mu_G$ on the toric Calabi-Yau $X$. And finally, we give the geometric interpretation of the mirror moduli constructed under the SYZ program which gives us the desired mirror manifold $X^\vee$ as a variety over $\CC$.
%-------------------------------------------------------------------------------------------------
\section{T-duality}
In this section, we introduce the basic philosophy of the SYZ program, that is "Mirror symmetry is T-duality". Naively, the SYZ proposal suggests that the mirror of a compact symplectic manifold $X$ is the total space of the fiberwise torus-dual of a special lagrangian fibration
$$\mu:X\longrightarrow B,$$which can be represented as the quotient
$$\frac{TB}{\Lambda}$$%\text{for some $B_0\subset B$ and lattice bundle $\Lambda_0\subset TB_0$.}
equipped with a canonical complex structure where $\Lambda\subset TB$ is a lattice sub-bundle. Geometrically, the mirror can be viewed as a moduli space of pairs $(L,\nabla)$ where $L$ is a Lagrangian fiber of $\mu$ and $\nabla$ is a flat-$U(1)$-connection over the Lagrangian $L$.
\par First of all, we define the dual of torus and its correspondence with the moduli of flat-$U(1)$-connections.

\begin{defn}
Let $\Lambda$ be a lattice isomorphic to $\mathds{Z}^n$, $V:=\Lambda \otimes_\ZZ \RR$, and $T=\frac{V}{\Lambda}$.
We define the dual to $T$ to be
$$T^\vee := \frac{V^*}{\Lambda^*}$$
where $\Lambda^*$ is embedded in $V^*$ by
$$\Lambda^*\cong\{v^*\in V^*|<v^*,\lambda>\in\ZZ,\forall\lambda\in\Lambda\}\subset V^*.$$
\end{defn}

On the other hand, suppose that $G\curvearrowright P\longrightarrow M$ be a flat principle $G$-bundle over a smooth manifold $M$.
We have a well-known correspondence between flat connections and monodromy representation as follows:

\begin{prop}[Proposition 1.18 of \cite{Abad}]
\label{G bundle is Hom}
Let $G\curvearrowright P\longrightarrow M$ be a flat principle $G$-bundle over a smooth manifold $M$. Then
$$Hom(\pi_1(M),G)^G \longleftrightarrow \{ Isomorphism\ class\ of\ flat\ G\ bundle\ over\ M \}$$ is a one-to-one correspondence.
\end{prop}

When we restrict ourself to the trivial bundle case $U(1)\curvearrowright T\times U(1)\longrightarrow T$, we get the following correspondence between the dual torus and some moduli of flat connections.

\begin{prop}
\label{T-dual to connection}
Let $T=\frac{V}{\Lambda}$ be a torus as above. Then
\eqt  T^\vee \longleftrightarrow \{Isomorphism\ class\ of\ flat\ U(1)\ connection\ \nabla\  over\ T\}\eqtn is a one-to-one correspondence.
\end{prop}

\begin{proof}
According to Proposition \ref{G bundle is Hom}, we have the one-to-one correspondence $$\{Isomorphism\ class\ of\ flat\ U(1)\ connection\ \nabla\  over\ T\} \longleftrightarrow Hom(\pi_1(T),U(1))^{U(1)} .$$
Notice that $U(1)$ is an abelian group and $\pi_1(T)$ is exactly $\Gamma$. Hence we have
$$Hom(\pi_1(T),U(1))^{U(1)} =Hom(\Gamma,U(1))  .$$
It is suffice if we obtain the group isomorphism $$\frac{V^*}{\Gamma^*} \simeq Hom(\Gamma,U(1)),$$ since $\frac{V^*}{\Gamma^*}$ is exactly the definition of the dual torus $T^\vee$.
Consider the group homomorphism $$f:V^*\longrightarrow Hom(\Gamma,U(1))$$ given by
$$f(v^*)(\gamma)=exp(v^*(\gamma))$$ for any $v^*\in V^*$ and $\gamma\in\Gamma$.
We then have its kernel computed as
\eqt \begin{split}
Ker(f)
&=\{v^*\in V^*|f(v^*)(\gamma)=1\in U(1),\forall \gamma\in\Gamma\} \\
&= \{v^*\in V^*|exp(v^*(\gamma)=1\in U(1),\forall \gamma\in\Gamma\} \\
&= \{v^*\in V^*|v^*(\gamma)\in 2\pi\ZZ,\forall \gamma\in\Gamma\} \\
&= \Gamma^* .\\
\end{split} \eqtn
To see that the map $f$ is surjective, for any $h\in Hom(\Gamma,U(1))$, and we pick $v^*\in V^*$ to be defined by
$$v^*(v)=\sum_i v_i Arg_0(h(e_i)) $$
for any $v=\sum_i v_i e_i$ where the $e_i's$ generates the lattice $\Gamma\in V$ and $Arg_0(u)\in [0,2\pi),\forall u\in U(1)$.
Then we have
\eqt \begin{split}
f(v^*)(\gamma)
&= exp(v^*(\gamma)) \\
&= exp(v^*( \sum_i n_i e_i)) \\
&= exp(\sum_i n_i Arg_0 h( e_i)) \\
&= \prod_i exp( Arg_0 (h( e_i)^n_i)    ) \\
&= \prod_i h( e_i)^n_i) \\
&=  h( \sum_i n_i e_i ) \\
&=  h( \gamma ) \\
\end{split} \eqtn
for any $\gamma = \sum_i n_i e_i \in \Gamma,n_i\in\ZZ$.
By the first isomorphism theorem, we obtain that $\frac{V^*}{\Gamma^*} \simeq Hom(\Gamma,U(1))$. Thus the proof is completed.
\end{proof}

Let $\mu:X\longrightarrow B$ be a special Lagrangian fibration of a compact symplectic manifold. We define
$$B_0:=\{b\in B|\mu^{-1}(b)\cong\TT^n\}\subset B$$
and
$$X_0:=\mu^{-1}(B_0).$$
By a theorem of Duistermaat \cite{Duistermaat}, we have the isomorphism
$$X_0\cong \frac{T^*B_0}{\Lambda^*}$$
where $\Lambda^*$ is a $\ZZ$-lattice bundle of the vector bundle $T^*B_0$ over $B_0$.

\par According to the SYZ proposal, we are going to define the mirror of $X_0$ and its underlying geometric meaning as a moduli connections.

\begin{defn}
\label{SYZ definition of M_0}
Let $\mu:X\longrightarrow B$ be a special Lagrangian fibration of a compact symplectic manifold. We define $\check{X_0}:=\frac{TB_0}{\Lambda}$ to be the mirror of $X_0\cong\frac{T^*B_0}{\Lambda^*}$, that is
$$ X_0\cong\frac{T^*B_0}{\Lambda^*}\underset{mirror}{\overset{SYZ}{\longleftrightarrow}}\frac{TB_0}{\Lambda}=\check{X_0}\ .$$
\end{defn}
According to Proposition \ref{T-dual to connection}, the mirror geometrically corresponds to a moduli of connections as follows:
\begin{prop}
Let $\mu:X\longrightarrow B$ be a special Lagrangian fibration of a compact symplectic manifold. Then the SYZ mirror of $X_0$ is
\eqt X^\vee_0 := \mathcal{M}_0 \eqtn
%is the moduli space of the pair $(L,\nabla)$ with $L$ being a regular fiber of $\mu$ and
%\{(L,\nabla)|\text{L is a regular fiber of $\mu$, $\nabla$ represents an isomorphism class of flat $U(1)$ connection over L}  \}
where $\mathcal{M}_0$ is the moduli space of the pair $(L,\nabla)$ with $L$ being a regular fiber of $\mu$ and $\nabla$ represents an isomorphism class of flat $U(1)$ connection over the fiber $L$.
\end{prop}

\begin{proof}
According to Definition \ref{SYZ definition of M_0} and using the result in Proposition \ref{T-dual to connection}, we obtain that
\eqt \begin{split}
X_0^\vee
&= \frac{TB_0}{\Lambda_0} = \bigcup_{b\in B_0} \frac{T_bB_0}{\Lambda_{0,b}} = \bigcup_{b\in B_0} (\frac{T_b^*B_0}{\Lambda_{0,b}^*})^\vee \\
&= \bigcup_{b\in B_0} \{Isomorphism\ class\ of\ flat\ U(1)\ connection\ \nabla\  over\ \frac{T_b^*B_0}{\Lambda_{0,b}^*}\}
%\\&=: \MM_0.
\end{split} \eqtn
Thus $X_0^\vee = \MM_0$ as desired and the proof is competed.
\end{proof}

At this moment, we have only constructed the mirror of $X_0$ to be $X^\vee_0=\mathcal{M}_0$. In order to construct the SYZ mirror of $X$, we will need to compactify $X^\vee_0$. We will need some symplectic information about the holomorphic disc counting of $(X,\omega)$ that allow us to define the partially compactified mirror $X^\vee$ together with a integrable complex structure attached to it. This procedure will be investigated in later chapters.

%-------------------------------------------------------------------------------------------------
\section{Toric Calabi-Yau manifolds}
We construct Calabi-Yau manifolds, $X$, by toric data with Calabi-Yau condition, $\Delta$. The manifold $X_\Delta$ has a canonical symplectic structure $(X,\omega)$, by Delzant construction and a canonical complex structure as a scheme over $\mathds{C}$. The symplectic and complex structure on $X$ are compatible, therefore $X$ carries a canonical K$\ddot{a}$hler structure, $(X,\omega,J)$. Calabi-Yau condition on $X$ is equivalent to the Calabi-Yau condition on the fan. According to some well known results in Toric Geometry \cite{Guillemin} , the two descriptions are equivalent, that is,
$$(X=X_\Delta \text{ as toric variety}) \longleftrightarrow (X=\CC^m\sslash_t G \text{ as symplectic reduction})$$

Let $\Sigma$ be a fan with $m$ rays $v_0,...,v_{m-1}\in N:=\ZZ^n$. We obtain a pair of dual short exact sequences
\begin{equation} \label{toric exact sequence}
0\longrightarrow G\mathop{\longrightarrow}^\iota \mathds{Z}^m \mathop{\longrightarrow}^ \beta N\longrightarrow 0;
\end{equation}
\begin{equation}
0\longrightarrow M\mathop{\longrightarrow}^{\beta^\vee} {\mathds{Z}^m}^\vee  \mathop{\longrightarrow}^{\iota^\vee}G^\vee\longrightarrow 0.
\end{equation}

\begin{defn}
Let $\Delta=\{\nu\in M:(\nu,v_i)\geq -\lambda_i \forall i\}$ be a polytope.
We define $(X_\Delta,\omega)$ as the symplectic reduction of $\mathds{T}^m \curvearrowright (\mathds{C}^m,\omega_{std})$ by the sub-torus action $G \otimes_\mathds{Z} \mathds{R} \subset \mathds{T}^m$ at the point $\iota^\vee(\lambda)\in G^\vee\otimes_\mathds{Z} \mathds{R}$.
\end{defn}

\begin{defn}
Let $\Sigma$ be a fan.
We define $X_\Sigma$ as a toric variety covered by $V_\sigma:=Spec \mathds{C}[\chi^\nu:\nu\in \sigma^\vee]$ running over $\sigma \in \Sigma$.
\end{defn}

\begin{lem}
%Let $\Delta$ be a polytope with corresponding fan $\Sigma$.
Let $\Sigma$ be a fan with $m$ rays $v_0,...,v_{m-1}\in N:=\ZZ^m$. We obtain a pair of dual short exact sequences over $\RR$
\begin{equation}
0\longrightarrow G_\RR\mathop{\longrightarrow}^\iota \mathds{R}^m \mathop{\longrightarrow}^ \beta N_\RR\longrightarrow 0;
\end{equation}
\begin{equation}
0\longrightarrow M_\RR\mathop{\longrightarrow}^{\beta^\vee} {\mathds{R}^m}^\vee  \mathop{\longrightarrow}^{\iota^\vee}G_\RR^\vee\longrightarrow 0.
\end{equation}
Suppose the $\lambda\in{\mathds{R}^r}^\vee $ and $t\in G_\RR^\vee$ such that $\iota^\vee(\lambda)=t$ and $\Delta=\{\nu\in M:(\nu,v_i)\geq -\lambda_i \forall i\}$.
Then the symplectic toric manifold $(X_\Delta,\omega_\Delta)$ is symplectomorphic the symplectic manifold $(\CC^m\sslash_t G_\RR,\omega_{Red})$ which is the symplectic reduction of $\CC^m$ by the sub-torus action of $\frac{G_\RR}{G} \curvearrowright \CC^m$ at the point $t\in G_\RR^\vee$.
\end{lem}

\begin{comment}
\begin{lem}
Let $\Delta$ be a polytope with corresponding fan $\Sigma$.
Then $X_\Sigma$ and $(X_\Delta,\omega)$ is isomorphic as a manifold with compatible symplectic and complex structure. Thus $X_\Delta$ has a canonical K$\ddot{a}$hler structure.
\end{lem}

\begin{proof}
The proof can be found in the Appendix of the Toric Geometry book by M.Vullxxxx.
The key is to construct the map between $X_\Sigma$ and $X_\Delta$ be
$$ \frac{\mu^{-1}(t)}{G_\mathds{R}} \mathop{\longrightarrow}^\simeq \frac{\mathds{C}^r_\Sigma}{G_\mathds{C}}
\mathop{\longrightarrow}^\simeq X_\Sigma$$
by
$$\mu^{-1}(t) \subset \mathds{C}^r_\Sigma$$
and
$$ \chi^{\beta^\vee} : \mathds{C}[\chi^m:m\in \sigma^\vee\subset M] \longrightarrow \mathds{C}[\chi^x:x\in \beta^\vee(\sigma^\vee)\subset {\mathds{R}^r}^\vee]$$
\end{proof}
\end{comment}

\begin{rmk}
To be more specific, we have
\begin{equation}
\begin{split}
% X_\Delta :=
\CC^m\sslash_t G_\RR &\trieq
\frac{\mu^{-1}(t)}{G_\mathds{R}} \\
&=\{(z_0,...,z_{m-1})\in \CC^m| \iota^\vee(r_0^2,...,r_{m-1}^2)=t\in \mathds{R}^k \cong G_\mathds{R})^\vee \ \\
&\ \ \ \ \ \ \ \ \ \ \ \ \ \ \ and\  (\theta_0^2,...,\theta_{m-1}^2)\equiv 0 \in \mathds{R}^k (mod\ G_\mathds{R}) \}.
\end{split}
\end{equation}
And the map between $(X_\Delta,\omega_\Delta)$ and $(\CC^m\sslash_t G_\RR,\omega_{Red})$ is given by
%$$ \frac{\mu^{-1}(t)}{G_\mathds{R}} \mathop{\longrightarrow}^\simeq \frac{\mathds{C}^r_\Sigma}{G_\mathds{C}} \mathop{\longrightarrow}^\simeq X_\Sigma$$
$$ {\beta^\vee}^* : \frac{\mu^{-1}(t)}{G_\mathds{R}} \mathop{\longrightarrow}^\simeq X_\Sigma,$$
$${\beta^\vee}^*:\CC^m_\sigma=Spec\CC[\beta^\vee(\sigma)]\longrightarrow
Spec\CC[\sigma]=V_\sigma$$
for any $\sigma\in\Sigma$, where
$$X_\Delta \subset \bigcup_\sigma V_\sigma\ \ \, \ \ \   \mu^{-1}(t) \subset \bigcup_\sigma \CC^m_\sigma\ \ \ and$$
$$\CC^m_\sigma\trieq\{(z_0,...,z_{m-1})\in\CC^m|z_i\neq 0 \text{ if } v_i\notin \sigma(1) \}$$

%$$ \chi^{\beta^\vee} : \mathds{C}[\chi^m:m\in \sigma^\vee\subset M] \longrightarrow \mathds{C}[\chi^x:x\in \beta^\vee(\sigma^\vee)\subset {\mathds{R}^r}^\vee]$$

This identification preserves the canonical symplectic structure on them and thus gives us the desired symplectomorphism, that is,
$$
(X_\Delta,\omega_\Delta) \overset{Symplectomorphic}{\simeq} (\CC^m\sslash_t G_\RR,\omega_{Red}.)
$$
\end{rmk}

\begin{defn}
A Calabi-Yau manifold is defined to a K$\ddot{a}$hler manifold with trivial canonical line bundle.
\end{defn}

\begin{lem} \label{CY condition on polytope}
Let $\Sigma$ be a fan with $r$ rays $v_0,...,v_{m-1}$.
If there exist $u\in M$ such that$(u,v_i)=1$ for all $i$.
Then $X_\Delta$ is Calabi-Yau.
\end{lem}

\begin{proof}
Notice that $\mathcal{K}_X=\mathcal{O}(\sum_i \mathcal{D}_i )$, and $$Div(\chi^u:X\rightarrow \mathds{C})=\sum_{i=0}^{m-1} <u,v_i>\mathcal{D}_i=\sum_{i=0}^{m-1} \mathcal{D}_i.$$
We have that $\chi^u$ can be viewed as a nowhere vanishing chapter on the canonical line bundle $\mathcal{K}_X$. Thus $X_\Delta$ is Calabi-Yau.
\end{proof}
%-------------------------------------------------------------------------------------------------

\section{Gross fibration}
In order to perform the SYZ program to the toric Calabi-Yau manifold $X_\Delta$ we will need special lagrangian fibration on $X_\Delta$. Due to a result of Gross and Goldstein in \cite{Gross,Goldstein}, we construct the Gross fibration
$\mu_G:X_\Delta\longrightarrow B$
which is the desired special lagrangian fibration. %by perturbing the moment map $\mu_\Delta:X_\Delta\longrightarrow\Delta\subset M.$
In order to perform the SYZ transformation, we first remove the singular Gross fibers of $X$ before constructing its corresponding mirror.
%$$\mu_G:X_0 \longrightarrow B_0.$$
Moreover, we compute some homology groups of the regular Gross fibers at the end of this section.

%In this section, we will choose a trivialization of the form$$\mu_G^{-1}(U)=U \times \TT^n$$ for some contractible subset $U\subset B$. Therefore we can identify every regular Gross fiber to an n-torus which are Lagrangian isotopic to each other. And hence for any regular Gross fiber $L$ and $L'$, we can identify their homology groups such that$$H_p(L;\ZZ)\simeq H_p(L';\ZZ),$$ $$H_p(X,L;\ZZ)\simeq H_p(X,L';\ZZ).$$Furthermore, we can compute these homology groups and obtain that$$H_p(L;\ZZ)\simeq N,$$  $$H_p(X,L;\ZZ)\simeq \ZZ^r.$$

\par Firstly, for any toric Calabi-Yau manifold $X_\Delta$, we define the Gross fibration as follows:

\begin{defn}[Gross\cite{Gross},Goldstein\cite{Goldstein}]
Let $\epsilon>0$ and $B:=\frac{M}{\mathds{R}u} \times \mathds{R}_{\geq 0}$.
We define the Gross fibration, $$\mu_{G}:X_\Delta \longrightarrow B,$$by
\begin{equation}
\mu_{G}(x):= ([\mu_\Delta(x)],|\chi^u(x)-\epsilon|^2).
\end{equation}
Moreover, we denote the first projection of $\mu_{G}$ by $\mu_{G}^{(1)}:X_\Delta \longrightarrow \frac{M}{\mathds{R}u}$ and the second projection of $\mu_{G}$ by $\mu_{G}^{(2)}:X_\Delta \longrightarrow \mathds{R}_{\geq 0}$.
\end{defn}

With the definition of Gross fibration above, we are going to compute it by computing the holomorphic function $\chi^u:X_\Delta \longrightarrow \mathds{C}$ explicitly.

\begin{prop} \label{Proposition compute CHI is product of zs }
Let $0\longrightarrow M \overset{\beta^\vee}{\longrightarrow} \mathds{Z}^m \overset{\iota^\vee}{\longrightarrow} G^\vee\longrightarrow 0$ be the dual exact sequence. Suppose $t\in G^\vee$ and $u=(0,...,0,1)\in \mathds{Z}^n \cong M$ and $<u,v_i>=1$ for all $i=0,...,{m-1}$. Then the holomorphic function %\chi^u:\frac{\mu^{-1}(t)}{G_\mathds{R}}
$\chi^u:\CC^m\sslash_t G \cong X_\Delta \longrightarrow \mathds{C}$ is given by
\begin{equation}
\chi^u([z_0,...,z_{m-1}])=\prod_{i=0}^{m-1} z_i.
\end{equation}
\end{prop}

\begin{proof}
Recall that the toric variety $X_\Delta$ can be represented by the symplectic reduction of
$$\CC^m \overset{\mu_{\CC^m}}{\longrightarrow} (\RR^m)^\vee \overset{\iota^\vee}{\longrightarrow} G^\vee :z_i\mapsto |z_i|^2 \mapsto \iota^\vee (|z_i|^2) $$ at $t\in G^\vee$ which is denoted by
$\CC^m\sslash_t G := \mu^{-1}(t)/G_\RR.$
Consider the dense open complex torus $(\CC^\times)^n \subset X_\Delta$, we have the identification
$$ \frac{\mu^{-1}(t) \cap (\CC^\times)^m }{G_\CC} \overset{\sim}{\longrightarrow} V_0:=Spec\CC[M]\subset X_\Delta$$ by
the map $$(\beta^\vee)^*:(\CC^\times)^, = Spec\CC[(\ZZ^m)^\vee] %\overset{(\beta^\vee)^*}
{\longrightarrow} Spec\CC[M] = V_0.$$
According to the definition of $\chi^u:X_\Delta\longrightarrow\CC^*$, the map $\chi^u:V_0\longrightarrow\CC^*$ is given by canonical map
$Spec\CC[M] \mapsto Spec\CC[\ZZ u]$ since $u\in M$.% Overall, we have obtain the map $\chi^u:\CC^r\sslash_t G \cong X_\Delta \longrightarrow \mathds{C}$ which is explicitly given by$$Spec\CC[(\ZZ^r)^\vee]\overset{(\beta^\vee)^*}{\longrightarrow} Spec\CC[\ZZ u].$$
Notice that $\beta^\vee(u)=\sum_{i=0}^{m-1}e_i^*\in (\ZZ^m)^\vee $ since we assume the Calabi-Yau condition that $<u,v_i>=1,\forall i=0,...,m-1$, we then have
$$ \beta^\vee (\chi^u) = \chi^{\beta^\vee(u)} = \chi^{   \sum_{i=0}^{m-1}  e_i^*  }
=\prod_{i=0}^{m-1} \chi^{e_i^*}  .$$
Hence the map $(\beta^\vee)^*$
 %$$(\CC^\times)^r = Spec\CC[(\ZZ^r)^\vee] \overset{(\beta^\vee)^*}{\longrightarrow} Spec\CC[M] = V_0$$
 is explicitly given by
% $$ (z_0,...,z_{r-1}) \overset{(\beta^\vee)^*}{\longmapsto} \prod_{i=0}^{r-1} z_i .$$
$$
(\beta^\vee)^*[z_0,...,z_{m-1}]=\prod_{i=0}^{m-1} z_i.
$$
We recall that the subset $V_0\subset X_\Delta$ is a dense and the map $\chi^u:X_0\longrightarrow\CC$ is holomorphic, and hence the map $\chi^u$
%$$\chi^u([z_1,...,z_r])=\prod_{i=1}^r z_i$$
extends naturally from $V_0$ to $X_\Delta$ using the same explicit form. Thus the proof is completed.

\end{proof}

Before we examine the discriminant locus of the Gross fibration and construct the contractible subset $U\subset B_0$, we first give a notation for each face of the polytope $\Delta$.

\begin{defn} \label{definition of F I polytope faces}
Let $\Delta=\{\nu\in M|<\nu,v_i>\geq -\lambda_i,\forall i=0,...,m-1\}$ be a polytope in $M$. For each non-trivial index set $I\subset \{0,...,m-1 \}$, we define
$$F_I\triangleq \{\nu\in \Delta|<\nu,v_i>=\lambda_i,\forall i\in I \}$$ which is a face in $\partial \Delta$ of codimension $(|I|-1)$.
\end{defn}

We are now ready to describe the discriminant locus of $\mu_G$ as follows:
\begin{prop}
Let $\mu_G:X_\Delta\longrightarrow B$ as defined above.Then the discriminant locus of the fibration is
$$\Gamma=\partial B \cup ( (\bigcup_{|I|=2}\frac{F_I}{\RR u})\times\{|\epsilon|^2\}).$$
\end{prop}

\begin{proof}
The critical points of $\mu_G$ is exactly where its differential is not surjective. Firstly, the critical points of $\mu_G^{(1)}$ is exactly the codimension-2 toric strata of $X_\Delta$. Secondly,
we consider the map $\mu_G^{(2)}:X_\Delta \longrightarrow \RR_{\geq 0}$ given by $\mu_G^{(2)}:x\mapsto |\chi^u(x)-\epsilon|^2.$
We then obtain the set of critical points of $\mu_G^{(2)}$, that is $Crit(\mu_G^{(2)})=\{x\in X_\Delta|\chi^u(x)-\epsilon \}.$
Under the image of $\mu_G$, we then have
\begin{equation} \label{eq1}
\begin{split}
\text{Discriminant locus of $\mu_G$}
& = \mu_G ( Crit( \mu_G^{(1)}  )  \cup \mu_G ( Crit( \mu_G^{(2)} ) )  \\
& = ( (\bigcup_{|I|=2}\frac{F_I}{\RR u})\times\{|\epsilon|^2\})  \cup  (  \frac{M}{\mathds{R}u}  \times  \{0\}   ) \\
& = ( (\bigcup_{|I|=2}\frac{F_I}{\RR u})\times\{|\epsilon|^2\})  \cup  \partial B.
\end{split}
\end{equation}
Thus the proof is completed.
\end{proof}

In the SYZ proposal, we discard singular Lagrangian fibers. We denote
$$B_0\trieq B-\Gamma$$
which can be viewed as the moduli of regular fibers of $\mu_G$. Notice that $B_0$ is not a necessarily a contractible. In order to trivialize the fibration, we define some open contractible subsets in $B_0$.

\begin{defn}
Let $B_0=B-\Gamma$ as defined above. We define
$$U_i\trieq B_0-\bigcup_{j\neq i,j=0}^{m-1}(\frac{F_{\{j\}}}{<u>}\times\{|\epsilon|^2\})$$
for $i=0,...,m-1$.
\end{defn}

\begin{rmk}
These subsets $U_0,...,U_{m-1}\subset B_0$ have different descriptions such as
$$U_i=\{(b_1,b_2)\in B_0|b_1\in\frac{F_{\{i\}}}{<u>}\ or\ b_2\neq |\epsilon|^2\}$$
$$U_i=B_+ \cup (\frac{F_{\{i\}}}{<u>}\times\{|\epsilon|^2\}) \cup B_-$$
where $B_+:=\{(b_1,b_2)\in B_0| b_2\geq |\epsilon|^2\}$ and $B_-:=\{(b_1,b_2)\in B_0| 0<b_2\leq |\epsilon|^2\}$. We observe that $U_0,...,U_{m-1}$ are contractible open subsets of $B_0$, using which we obtain trivialization of the Lagrangian fibers.
\end{rmk}

Without loss of generality, we assume that $$U:=U_0\subset B.$$ We have the following trivialization of $\mu_G^{-1}(U)$.

\begin{defn}
Let $\mu_G:X_\Delta \longrightarrow B$ be the Gross fibration and $U\subset B$.
Then we have the trivialization
$$\mu_G^{-1}(U) \xrightarrow{\sim} U\times %\frac{\RR<v'_1,...,v'_{r-1}>}{\ZZ<v'_1,...,v'_{r-1}>} \times \frac{\RR}{\ZZ}
\TT^n$$ by the map
$$x\mapsto (  [\mu_\Delta(x)]  ,
(  \frac{1}{2\pi}  arg(\chi^{{v'_1}^*}(x))
,...,
\frac{1}{2\pi}arg(\chi^{{v'_{n-1}}^*}(x)),
\frac{1}{2\pi}arg(\chi^u(x)-\epsilon)   )  ),$$
where $v'_i:=v_i-v_0$ and $\{v_0,v'_1,...,v'_{n-1}\}$ forms a basis for $N$ with the dual basis $\{v_0^*,{v'_1}^*,...,{v'_{n-1}}^*\}\subset M$.
\end{defn}

Using this trivialization, we obtain the result that almost every regular Gross fibers are isotopic to each other. Moreover these fibers are isotopic to some moment map fibers.

\begin{prop}
Let $b_0\in B_+ \subset U \subset B_0 $ be fixed and $L_{G,0}:=\mu_G^{-1}(b_0)$. Then for any $u\in U\subset B_0$ and $L_G=\mu_G^{-1}(u)$, we have $L_G\subset X$ is Lagrangian isotopic to $L_\Delta\subset X$ for some Lagrangian moment map fiber $L_\Delta$.
\end{prop}

\begin{proof}
Let $b_0=(b_1,b_2)$ and consider the Lagrangian isotopy
$$L_t:=\{ x\in X_\Delta | [\mu_\Delta(x)]=b_1 and |\chi^u(x)-t\epsilon|^2=b_2   \}$$
for $t\in[0,1]$. For the case $t=0$, we have
$$L_0:=\{ x\in X_\Delta | [\mu_\Delta(x)]=b_1 and |\chi^u(x)|^2=b_2   \}$$ which is the desired moment map fiber denoted by $L_\Delta$. And for the case $t=1$, we have
$$L_1:=\{ x\in X_\Delta | [\mu_\Delta(x)]=b_1 and |\chi^u(x)-\epsilon|^2=b_2   \}$$ which is a regular Gross fiber $L_{G,0}$. Hence we obtained that
$$L_{G,0}\sim  L_\Delta$$ is Lagrangian isotopic to each other through $L_t$.
On the other hand, the trivialization of $\mu_G^{-1}(U)$ implies that every fibers $L_G=\mu_G^{-1}(u)$ for $u\in U$ is isotopic to $L_{G,0}=\mu_G^{-1}(b_0)$, that is
$$L_{G,0}\sim  L_G.$$
Hence we have the Lagrangian isotopy as desired, that is
$$L_G \sim L_{G,0}\sim  L_\Delta.$$
\end{proof}

In the perspective of algebraic topology, we now relate the homology groups between the isotopic submanifolds, $L_G$ and $L_\Delta$, as follows:

\begin{lem}
If $L_0,L_1\subset X$ be isotopic submanifolds, then we have an isomorphism between the two long exact sequences, that is
\eqt \begin{CD}
... @>>> H_{p+1}(X) @>>> H_{p+1}(X,L_0) @>>> H_p(L_0) @>>> H_p(X)  @>>> ...  \\
 @.           @|             @VV\simeq V   @VV\simeq V     @|             \\
 ... @>>> H_{p+1}(X) @>>> H_{p+1}(X,L_1) @>>> H_p(L_1) @>>> H_p(X)  @>>> ...  \\
\end{CD} \eqtn
for any $p\geq 0$.
\end{lem}

\begin{proof}
For each of the submanifolds $L \subset X$, we have the exact sequence of relative simplicial chains,
$$0\longrightarrow C_p(L) \longrightarrow C_p(X) \longrightarrow C_p(X,L) \longrightarrow 0$$
for $p\geq 0$. By the zig-zag lemma, we have the long exact sequence
$$... \longrightarrow H_{p+1}(X) \longrightarrow H_{p+1}(X,L)
\longrightarrow H_p(L) \longrightarrow H_p(X) \longrightarrow ...$$
for $p\geq 0$. By assumption, $L_0$ is isotopic to $L_1$. In particular, $L_0$ is isomorphic to $L_1$. Hence we have $H_p(L_0)$ is isomorphic to $H_p(L_1)$ for any $p\geq 0$.
Hence we have
\eqt \begin{CD}
... H_{p+1}(L_0) @>>> H_{p+1}(X) @>>> H_{p+1}(X,L_0) @>>> H_p(L_0) @>>> H_p(X)  ...  \\
     @VV\simeq V         @|                @.             @VV\simeq V     @|             \\
 ...H_{p+1}(L_1) @>>> H_{p+1}(X) @>>> H_{p+1}(X,L_1) @>>> H_p(L_1) @>>> H_p(X)  ...  \\
\end{CD} \eqtn
for any $p\geq 0$. By the "theorem of five", we have an isomorphism between $H_{p+1}(X,L_0)$ and $H_{p+1}(X,L_1)$ for $p\geq 1$. Thus the proof is completed.
\end{proof}

In particular, we obtain
$$H_2(X,L_G;\ZZ)\simeq \ZZ^m.$$
for any $u\in U\subset B_0$ such that $L_G=\mu_G^{-1}(u)$.
We explicitly denote the basis of $H_2(X,L_G;\ZZ)$ to be
$\{\beta_0(L_G),...,\beta_{m-1}(L_G)\}$ which corresponds to the basis $\{e_0,...,e_{m-1}\}$ of $\ZZ^m$. under the above isomorphism. Hence we have
\begin{align} \label{define beta on gross fibers} H_2(X,L_G;\ZZ) = <\beta_0(L_G),...,\beta_{m-1}(L_G)>.\end{align}
These relative 2-cycles will be crucial in describing the complex structure of the mirror which we will discuss in later section.

%-------------------------------------------------------------------------------------------------

\section{The Mirror Moduli}
Mirror symmetry asserts that if $(X,X^\vee)$ is a mirror pair, then there is a duality between the symplectic data of $(X,\omega)$ and the algebraic complex data of $(X^\vee,J)$. In previous chapters, we have constructed $X^\vee_0$ as a moduli space of connections $\mathcal{M}_0$ from the fibration of $X$,
$\mu:X \longrightarrow B.$
%In this section, we restrict ourself to toric Calabi-Yau manifolds. By considering the Gross fibration on $X_\Delta^{CY}$
we can define a complex structure on the mirror moduli. More specifically, we are going construct the semi-flat complex structure on $\MM_0$ which comes with it naturally. Furthermore, we are going to define some complex-valued functions
$$\tilde{z}_i:\MM_0\longrightarrow\CC,\ \forall i=0,...,m-1$$ which involves the holomorphic disc counting invariants. We will also compute these functions explicitly in terms of the semi-flat coordinates and describe some of their properties. These functions will be important for us to construct the mirror with "quantum correction" in later sections.

%More specifically, we are going to define complex valued function on $\mathcal{M}_0$. By specifying the holomorphic functions on $\mathcal{M}_0$, we have defined a complex structure on it.

\par To begin with, let $\Sigma$ be a fan with $r$ rays $v_0,...v_{m-1}\in N$, and $\Delta$ be the polytope associated to $\Sigma$, which is given by
$$\Delta = \{\nu\in M | <\nu,v_i>\geq -\lambda_i, \forall i\in 0,...,m-1 \}.$$
Let $$\mu_G:X_\Delta \longrightarrow B = \frac{M}{\RR u}\times\RR_{\geq0}$$ be the Gross fibration of the toric Calabi-Yau manifold. According to the SYZ mirror construction, we have the mirror of $X_0$ to be $X^\vee_0=\mathcal{M}_0$, where $\mathcal{M}_0$ is the moduli space of the pairs $(L,\nabla)$ with $L$ being a regular Gross fiber and $\nabla$ represents an isomorphism class of flat $U(1)$ connection on $L$.
%$$\mathcal{M}_0=\{(L,\nabla)|\text{$L$ is a regular Gross fiber\\and $\nabla$ represents an isomorphism class of flat $U(1)$ connection on $L$}\}$$
\par We first define the toric modification of $X_\Delta$.

\begin{defn}
Let $K\in\RR$ be large enough and $\Delta = \{\nu\in M | <\nu,v_i>\geq -\lambda, \forall i\in 0,...,m-1 \}$ be a potytope in $M$.
We define
$$\Delta^{(K)} = \{\nu\in M | <\nu,v_i>\geq -\lambda_i,<\nu,v'_i>\geq -K ,\forall i\in 0,...,m-1 \}$$ where $v'_i:=v_i-v_0$ for $i=1,...,m-1$.
We call $(X_{\Delta^{(K)}},\omega^{(K)})$ the toric modification of $(X_\Delta,\omega)$.
\end{defn}

\begin{exs}
We have few simple examples such as the toric modification of $\mathcal{K}_{\mathbb{P}^1}$,$\mathcal{K}_{\mathbb{P}^2}$,$\OO_{\PP^1}(-1)\oplus\OO_{\PP^1}(-1)$,...,etc.
\end{exs}

In previous chapters, we have defined the Gross fibration of $X_\Delta$, now we define the toric modified Gross fibration for $X^{(K)}:=X_{\Delta^{(K)}}$.

\begin{defn}
Let $\epsilon\in\CC$, $K\in \RR$ be large enough and $\Sigma'(1)=\{v_0,v_1,...,v_{m-1},v'_1,...,v'_{m-1}\}$.
We define
$$\mu^{(K)}_G:X^{(K)}\longrightarrow B^{(K)}:=\frac{M^{(K)}}{\RR u}\times\RR_{\geq0}$$
by
$$\mu^{(K)}_G(x):=([\mu_\Delta(x)]_{\RR u} , |\chi^u(x)-\epsilon|^2 )$$
where
$$M^{(K)}:=\{\nu\in M| <\nu,v'_i>\geq -K ,\forall i\in 1,...,m-1 \}.$$
\end{defn}

\begin{rmk}
Recall that the discriminant locus of $\mu_G:X_\Delta\longrightarrow B$ is $\Gamma\subset B$. For the modified Gross fibration $\mu^{(K)}_G:X^{(K)}\longrightarrow B^{(K)}$, its discriminant locus is given by
$$\Gamma^{(K)}:=(\Gamma\cap B^{(K)})\cup \partial B^{(K)} .  $$
\end{rmk}

To simplify our notations, we suppress the dependence on $K$ and $\epsilon$ for
$$\mu^{(K)}_G:X^{(K)}\longrightarrow B^{(K)},$$ we simply denote the modified Gross fibration to be
$$\mu'_G:X'\longrightarrow B'.$$

\begin{defn}
Let $L\in X_\Delta$ be a regular Lagrangian torus fiber of the Gross fibration $\mu_G$ and $\beta\in \pi_2(X_\Delta,L)$ be a relative homotopy class with Maslov index 2. Suppose $\bar{\MM}_k(L;\beta)$ denotes the moduli space of stable maps from genus 0 bordered Riemann surfaces with connected boundary and $k$ boundary mark points respecting the cyclic order which represent the class $\beta$. Then we define the genus zero open Gromov-Witten invariant $n_\beta$ by
$$ n_\beta \overset{def}= ev_*([\bar{\MM}_1(L;\beta)]^{vir}) \in H_n(L;\QQ) \simeq \QQ $$
where $ev:\bar{\MM}_1(L;\beta)\longrightarrow L$ is the evaluation at the single boundary marked point.
\end{defn}

\begin{rmk}
The above definition is well-defined. It was shown in Fukaya-Oh-Ohta-Ono \cite{FOOO} that the moduli space $\bar{\MM}_k(L;\beta)$ admits a Kuranishi structure with virtual dimension $n+\mu(\beta)+k-3$ where $\mu(\beta)$ denotes the Maslow index of the class $\beta$. Moreover, a virtual fundamental chain $[\bar{\MM}_k(L;\beta)]^{vir}$ can be constructed via perturbations. The invariant $n_\beta$ is well-defined due to the fact that $[\bar{\MM}_1(L;\beta)]^{vir}$ can be shown to be a cycle in $H_n(L;\QQ)$ when the class $\beta$ is of Maslov index 2.
\end{rmk}

%%\begin{rmk}	Geometrically, $n_\beta$ counts number of holomorphic discs in $X$ of relative homology class $\beta \in H_2^{eff}(X,L;\ZZ)$. However, $n_\beta$ may not always be an integer but an rational number instead which accounts for multiple covering of the holomorphic discs.	\end{rmk}

In order to correct the mirror $\mu_G^{-1}(B_0)$, the open Gromov-Witten invariants with respect to the Gross fibers $L_G$ are the most crucial data. According to a result in Chan-Lau-Leung \cite{Chan Lau Leung}, the invariants are computed as follows:

\begin{prop}
%{[CLL:Prop4.35,4.36]}
\label{n beta}
Let $L={\mu'}_G^{-1}(b_1,b_2)$ be a Gross fiber of $X'$ and $\beta\in H_2^{eff}(X',L;\ZZ)$.Then
\begin{enumerate}
\item
for $b_2 > |\epsilon|^2$, $n_\beta\neq 0$ only when $$\beta=\beta'_i,\forall i=1,...,m-1$$ or $$ \beta=\beta_i+\alpha,\forall i=0,...,m-1 $$
where $\alpha\in H_2^{eff}(X',\ZZ)$ is represented by rational curves of Chern number zero. Moreover, $n_\beta=1$ when $\beta=\beta_0,\beta_1,...,\beta_{m-1},\beta'_1,...,\beta'_{m-1}.$
\item for $b_2<|\epsilon|^2$,
$$n_\beta=1,for\ \beta=\beta_0,\beta'_1,...,\beta'_{m-1}.$$
and
$$n_\beta=0,for\ \beta\neq \beta_0,\beta'_1,...,\beta'_{m-1}.$$
\end{enumerate}
\end{prop}

In the modified toric Calabi-Yau manifold $X'$, there are some well-studied toric divisors $\mathcal{D}_0,\mathcal{D}_1,...,\mathcal{D}_{m-1},\mathcal{D}'_1,...\mathcal{D}'_{m-1}$ which correspond to the rays $v_0,v_1,...,v_{m-1},v'_1,...,v'_{m-1}\in\Sigma'(1)$. In order to define holomorphic function on the mirror moduli, we will need another set of divisors that is introduced as follows:

\begin{defn}
Let $\mu'_G:X'\longrightarrow B'$ be the toric modified Gross fibration.
From the definition of $B'$,we observe that $\partial B'=(\partial B')_0\cup(\partial B')_1\cup...\cup(\partial B')_{m-1}$
where
$$(\partial B')_0=\{(b_1,b_2)\in B'| b_2=0 \}$$
$$(\partial B')_i=\{(b_1,b_2)\in B'| <b_1,v'_i>=-K \}$$
for $i=1,...,m-1$ and the pairing
$\frac{M}{\RR u}\times\{v'_1,...,v'_{m-1}\}\longrightarrow \RR$ is well-defined.
\par We define
$$D_i={(\mu')}^{-1}(\partial B')_i,for\ i=0,1,...,m-1$$
\end{defn}

\begin{rmk}
The $D_0,D_1,...,D_{m-1}\subset X'$ are divisors of $X'$. Notice that these divisors $D_i$ are NOT toric divisors.
\end{rmk}

We have defined the divisors $D_0,D_1,...,D_{m-1}\subset X'$ and the relative homology classes $\beta_0(b),\beta_1(b),...,\beta_{m-1}(b),\beta'_1(b),...,\beta'_{m-1}(b)\in H_2(X',(\mu')^{-1}(b);\ZZ)$ for any $b\in B'$. We now compute the intersection number of these divisors and relative cycles.

\begin{prop}
\label{intersention}
Let $$D_0,D_1,...,D_{m-1}\subset X',$$ $$ \beta_0(b),\beta_1(b),...,\beta_{m-1}(b),\beta'_1(b),...,\beta'_{m-1}(b)\in H_2(X',(\mu')^{-1}(b);\ZZ)$$
as defined above. Then the intersection number of the divisors and the relative cycles are computed and listed as follows:

\begin{comment}
\begin{center}
  \begin{tabular}{ | r || c | c | c | }
    \hline
	$\beta\cap D$ & $\beta_0$ & $\beta_1$,...,$\beta_{m-1}$ & $\beta'_1$,...$\beta'_{m-1}$ \\ \hline\hline
    $D_0$ & 1 & 1,...,1 & 0,...,0 \\ \hline

 \hline
  \end{tabular}
\end{center}
\end{comment}

\begin{center}
  \begin{tabular}{ | r || c | c | c | }
    \hline
	$\beta\cap D$ & $\beta_0$ & $\beta_i$ & $\beta'_i$ \\ \hline\hline
    $D_0$ & 1 & 1 & 0 \\ \hline
    $D_j$ & 0 & 0 & $\delta_{ij}$ \\ \hline
  \end{tabular}
\end{center}
for $i,j=1,...,m-1$.
\end{prop}

With the results of the intersection numbers, we are going to define $r$ complex valued functions on the mirror moduli $\MM'_0$, which is the moduli space of the pair $(L,\nabla)$ with $L\subset X'$ being a regular fiber of $\mu'$ and $\nabla$ represents an isomorphism class of flat $U(1)$ connection over the fiber $L$.

\begin{defn}
Let $K\in \RR$ be fixed and $\mu'_G:X'\longrightarrow B'$ be the modified Gross fibration and $\MM'_0$ be the mirror moduli of $X'_0$. We define
\begin{align}
Z_\beta &:\MM'_0\longrightarrow \CC \\ &: (L,\nabla)\longmapsto  Z_\beta(L,\nabla) := exp(-\int_\beta \omega_{X'})Hol_\Delta(\partial\beta)
\end{align}
which is called the semi-flat complex coordinates for each $\beta\in H_2(X',L;\ZZ)$. For $i=0,1,...,m-1$, we define
\begin{align} \label{z tilda}
\tilde{z}'_i &:\MM'_0\longrightarrow \CC \\ &:(L,\nabla)\longmapsto \tilde{z}_i(L,\nabla) := \sum_\beta (\beta\cap D_i)n_\beta Z_\beta(L,\nabla)
%where\ \tilde{z}'_i(L,\nabla):= \sum_\beta (\beta\cap D_i)n_\beta Z_\beta(L,\nabla)
%\text{and the summation runs over $\beta\in H_2^{eff}(X',L;\ZZ)$}
\end{align}
where the summation runs over $\beta\in H_2^{eff}(X',L;\ZZ)$.
%which is the defined to be holomorphic functions on $\MM'_0$ and these functions $\tilde{z}_0,...\tilde{z}_{m-1}$ generates the ring $\mathcal{O}(\MM'_0)$, the holomorphic structural function ring of $\MM'_0$.
\end{defn}

\begin{rmk}
\label{rmk on holo of z-tilde}
These complex valued functions $\hat{z}'_i$ are not global holomorphic functions on $\MM'_0$ with respect to its semi-flat complex structure. However the functions $\hat{z}'_i$ are holomorphic when restricted in $\MM'_+,\MM'_-\subset \MM'_0$ where
$$\MM'_+\simeq (\mu'_G)^{-1}(B_+)$$
$$\MM'_-\simeq (\mu'_G)^{-1}(B_-).$$
\end{rmk}

Before expressing $\hat{z}'_i$ in terms of the semi-flat complex coordinates $Z_\beta$, we are not able to determine whether they are holomorphic on $\MM_0$ yet.
%the definition of holomorphic functions on $\MM'_0$ in \eqref{z tilda} and
Using the result of the intersection number in Proposition \ref{intersention} and result on computation of the open Gromov-Witten invariants in Proposition \ref{n beta}, we are ready to compute the functions $\tilde{z}'_0(L,\nabla),...,\tilde{z}'_{m-1}(L,\nabla)$ explicitly in terms of the semi-flat complex coordinates $Z_\beta$ as follows:

\begin{thm} \label{thm z tilde with defn of delta}
Let $\tilde{z}'_i:\MM'_0\longrightarrow \CC$ for $i=0,...,m-1$ as defined above and $L=(\mu'_G)^{-1}(b_1,b_2)\subset X'$.
Then we have
\begin{enumerate}
\item
for $b_2<|\epsilon|^2$ and $i=1,...,m-1$,
\eqt \tilde{z}'_i(L,\nabla)=Z_{\beta_i}(L,\nabla), \eqtn
\eqt \tilde{z}'_0(L,\nabla)=Z_{\beta_0}(L,\nabla);\eqtn
\item
for $b_2>|\epsilon|^2$ and $i=1,...,m-1$,
\eqt \tilde{z}'_i(L,\nabla)=Z_{\beta_i}(L,\nabla), \eqtn
\eqt \tilde{z}'_0(L,\nabla)=\sum_{i=0}^{m-1}(1-\delta_i)Z_{\beta_i}(L,\nabla).\eqtn
where \eqt \label{definition of delta i}\delta_i:=\sum_{\alpha>0} n_{\beta_i+\alpha}exp(-\int_\alpha\omega) \eqtn
which sums over $\alpha\in H_2(X';\ZZ)$ that can be represented by non-trivial rational curves of Chern number zero.

\end{enumerate}
\end{thm}

\begin{proof}
According to the computation of the open Gromov-Witten in Proposition \ref{n beta} and the intersection number in Proposition \ref{intersention}. We compute $\tilde{z}'_i(L,\nabla)$ for $i=0,1,...,m-1$ as defined in \eqref{z tilda}.
\par To prove part (1), we compute that
\begin{equation}
\begin{split}
\tilde{z}'_i(L,\nabla) 	& = \sum_\beta (\beta\cap D_i)n_\beta Z_\beta(L,\nabla) \\
 						& = (\beta_0\cap D_i)n_{\beta_0} Z_{\beta_0}(L,\nabla) + \sum_{j=1}^{m-1} (\beta'_j\cap D_i)n_{\beta'_j} Z_{\beta'_j}(L,\nabla) \\
                       % & = \sum_{j=1}^{m-1} \delta_{ij} \cdot 1 \cdot Z_{\beta'_j}(L,\nabla) \\
                        & = Z_{\beta'_i}(L,\nabla) \\
\end{split}
\end{equation}
for $i=1,...,m-1$. And we have
\begin{equation}
\begin{split}
\tilde{z}'_0(L,\nabla) 	& = \sum_\beta (\beta\cap D_0)n_\beta Z_\beta(L,\nabla) \\
 						& = (\beta_0\cap D_0)n_{\beta_0} Z_{\beta_0}(L,\nabla) + \sum_{j=1}^{m-1} (\beta'_j\cap D_0)n_{\beta'_j} Z_{\beta'_j}(L,\nabla) \\
                        & = Z_{\beta_0}(L,\nabla) \\
\end{split}
\end{equation}
This completes the proof of part (1).
\par To prove part (2), we compute that
\begin{equation}
\begin{split}
\tilde{z}'_i(L,\nabla) 	& = \sum_\beta (\beta\cap D_i)n_\beta Z_\beta(L,\nabla) \\
 						& = (\beta_0\cap D_i)n_{\beta_0} Z_{\beta_0}(L,\nabla) + \sum_{j=1}^{m-1} (\beta_j\cap D_i)n_{\beta_j} Z_{\beta_j}(L,\nabla) \\
                        & \ \ \ \ + \sum_{j=1}^{m-1} (\beta'_j\cap D_i)n_{\beta'_j} Z_{\beta'_j}(L,\nabla) +
                         \sum_{j=0}^{m-1} \sum_\alpha((\beta_j+\alpha)\cap D_i)n_{\beta_j+\alpha} Z_{\beta_j+\alpha}(L,\nabla)  \\
                       % & = \sum_{j=1}^{m-1} \delta_{ij} \cdot 1 \cdot Z_{\beta'_j}(L,\nabla) \\
                        & = Z_{\beta'_i}(L,\nabla) \\
\end{split}
\end{equation}
for $i=1,...,m-1$. And we have
\begin{equation}
\begin{split}
\tilde{z}'_0(L,\nabla) 	& = \sum_\beta (\beta\cap D_0)n_\beta Z_\beta(L,\nabla) \\
 						& = (\beta_0\cap D_0)n_{\beta_0} Z_{\beta_0}(L,\nabla)
                        + \sum_{j=1}^{m-1} (\beta_j\cap D_0)n_{\beta_j} Z_{\beta_j}(L,\nabla) \\
                        & \ \ \ \ + \sum_{j=1}^{m-1} (\beta'_j\cap D_0)n_{\beta'_j} Z_{\beta'_j}(L,\nabla) +
                         \sum_{j=0}^{m-1} \sum_\alpha((\beta_j+\alpha)\cap D_0)n_{\beta_j+\alpha} Z_{\beta_j+\alpha}(L,\nabla)  \\
                        %& =  Z_{\beta_0}(L,\nabla)+\sum_{j=1}^{m-1} Z_{\beta_j}(L,\nabla)
                       % +\sum_\alpha n_{\beta_0+\alpha}exp(-\int_\alpha\omega)Z_{\beta_0}(L,\nabla) \\
                       % & \ \ \ \ +\sum_{j=1}^{m-1}\sum_\alpha n_{\beta_j+\alpha}exp(-\int_\alpha\omega)Z_{\beta_j}(L,\nabla) \\
                        & = (1+\delta_0)Z_{\beta_0}(L,\nabla) + \sum_{j=1}^{m-1}(1+\delta_i)Z_{\beta_i}(L,\nabla)\\
                        & = \sum_{i=0}^{m-1}(1+\delta_i)Z_{\beta_i}(L,\nabla) \\
\end{split}
\end{equation}
Thus the proof is completed.
\end{proof}

The above computation verifies the remark made in Remarks \ref{rmk on holo of z-tilde}. When we are restricted to the disjoint union of subspaces $\MM_+ \sqcup \MM_-\subset \MM_0$, we obtain immediately that
$$\tilde{z}^{(K)}_0,...,\tilde{z}^{(K)}_{m-1} \in \OO(\MM_+ \sqcup \MM_-).$$
The construction of mirror is not done yet. In order to construct the mirror, our aim is get rid of the dependence of $K$ on
$$\tilde{z}^{(K)}_0,...,\tilde{z}^{(K)}_{m-1}:\MM_0^{(K)}\longrightarrow\CC$$
to obtain
$$\tilde{z}_0,...,\tilde{z}_{m-1}:\MM_0\longrightarrow\CC$$
which will be studied in the next section.

%We have defined $r$ holomorphic functions $\tilde{z}^{(K)}_0,...,\tilde{z}^{(K)}_{m-1}$ on the toric modified mirror moduli $\MM_0^{(K)}$ for some fixed $K\in \RR$ and have computed them explicitly on the above computation. The construction is not done yet. Our aim is to extend the functions $$\tilde{z}^{(K)}_0,...,\tilde{z}^{(K)}_{m-1}:\MM_0^{(K)}\longrightarrow\CC$$ to $$\tilde{z}_0,...,\tilde{z}_{m-1}:\MM_0\longrightarrow\CC$$ which will be investigated in the next chapter.

%-------------------------------------------------------------------------------------------------

\section{Geometric realization}
In this chapter, our aim is to associate the semi-flat mirror moduli to an algebraic variety over $\CC$.

The main difficulty for SYZ proposal to work is that the semi-flat mirror $$\MM_0\simeq \frac{TB_0}{\Lambda_0}$$ does not behave as good as we desire. Since the natural semi-flat complex structure on $\MM_0$ cannot be extended further to any partial compactification as the monodromy of the complex structure around the discriminant locus $\Gamma$ is non-trivial. To obtain the correct and partially compactified mirror $X^\vee$, we modify the complex structure on $\MM_0$ by quantum correction involving instantons.

In the language of algebraic geometry, instead of studying the complex geometry of the mirror moduli $\MM_0$, we study the coordinate ring $\mathcal{O}(\MM_0)$. The main argument is to construct a subring of holomorphic functions on $\MM_+\sqcup \MM_-$, that is,
$$R \overset{\varphi}{\subset} \OO(\MM_+\sqcup \MM_-).$$
This immediately give us the geometric realization of the mirror moduli $\MM_+\sqcup \MM_-$ to be the mapped into the variety $X^\vee$, a partial compactification of $\text{Spec}R$, via the pullback of the ring homomorphism $\varphi$, that is,
$$\MM_+\sqcup \MM_- \overset{\varphi^*}{\longrightarrow} \text{Spec}R \subset X^\vee .$$
%After partial compactification of $X^\vee_0$, we obtain $X^\vee$.
Then for any toric Calabi-Yau manifolds $X$, the corresponding mirror variety $X^\vee$ constructed according to the above construction is called the SYZ mirror with quantum correction which is the mirror we desired that matches the prediction given in the physics literature.

To begin with, we modify the functions $\tilde{z}_0^{(K)},...,\tilde{z}_{m-1}^{(K)}$ by a change of coordinates. In order to construct the mirror variety, the functions $\tilde{z}_0^{(K)},...,\tilde{z}_{m-1}^{(K)}$ are as not suitable as we desire. We define some functions $\hat{z}_0^{(K)},...,\hat{z}_{m-1}^{(K)}$ and relate them to $\tilde{z}_0^{(K)},...,\tilde{z}_{m-1}^{(K)}$ as follows:

\begin{prop} \label{z hat definition}
Let $K\in\RR$ be fixed we define $\hat{z}_0^{(K)},...,\hat{z}_{m-1}^{(K)}:\MM_0^{(K)}\longrightarrow\CC$ by
\begin{align}
\hat{z}_0^{(K)}(L,\nabla) & = \tilde{z}_0^{(K)}(L,\nabla) \\
\hat{z}_i^{(K)}(L,\nabla) & = \frac{Z_{\beta_i}(L,\nabla)}{Z_{\beta_0}(L,\nabla)}
\end{align}
for any $i=1,...,m-1$ and $(L,\nabla)\in\MM_0$.
Then we have $$\hat{z}_i^{(K)}(L,\nabla)=C_i\tilde{z}_i^{(K)}(L,\nabla)\ \ , \forall i=0,...,m-1$$
for any $i=1,...,m-1$ and $(L,\nabla)\in\MM_0$ where the $C_i$'s are constants such that $C_0=1$ and $C_i=exp(\int_{\alpha_i} \omega)$ for some $\alpha_i\in H_2(X^{(K)},\ZZ)$ and for $i=1,...,m-1$.
\end{prop}

\begin{proof}
Let $K\in\RR$ be fixed, $\Sigma(1)'=\{v_0,v_1,...,v_{m-1},v'_1,...,v'_{m-1}\}$ and $(L_G,\nabla)\in\MM'_0$.
% We have the commutative diagram of exact sequence as follows:

\begin{comment}
\eqt \label{beta and vi}
\begin{tikzcd}
G' \arrow{r}{\iota'} \arrow{d}{\simeq} & \ZZ^r\times\ZZ^{m-1} \arrow{d}{\simeq} \arrow{r}{\beta'} & N \arrow{d}{\simeq} \\
H_2(X')  \arrow{r}{\mathds{1}} & H_2(X',L_G) \arrow{r}{\partial} & H_2(L_G)
%H_2(X';\ZZ)  \arrow{r}{\mathds{1}} & H_2(X',L_G;\ZZ) \arrow{r}{\partial} & H_2(L_G;\ZZ)
%G' \arrow{r}{\iota'} \arrow{d}{f} &
\end{tikzcd}
\eqtn
%\]
\end{comment}

%\eqt \begin{CD} \label{beta and vi}    0 @>>> G'    @>{\iota'}>> \ZZ^r\times\ZZ^{m-1} @>{\beta'}>> N  @>>> 0\\   @.      @VV{\simeq}V           @VV{\simeq}V              @VV{\simeq}V\\  0 @>>> H_2(X') @>{\mathds{1}}>>   H_2(X',L_G) @>{\partial}>> H_2(L_G) @>>> 0    \end{CD} \eqtn

To suppress our notation, we assume that $\ZZ$ is the coefficient ring of all the homology groups.
We identify $H_2(X',L_G)$ with $\ZZ^m\times\ZZ^{m-1}$ by
%$$\ZZ\beta_0\oplus\ZZ\beta_1\oplus...\oplus\ZZ\beta_{m-1}\oplus\ZZ\beta'_1\oplus...\oplus\ZZ\beta'_{m-1} \cong \ZZ e_0\oplus\ZZ e_1\oplus...\oplus\ZZ e_{m-1}\oplus\ZZ e'_1\oplus...\oplus\ZZ e'_{m-1}$$
% $$\ZZ\beta_0\oplus\bigoplus_{i=1}^{m-1}\ZZ\beta_i\oplus\bigoplus_{i=1}^{m-1}\ZZ\beta'_i \cong \ZZ e_0\oplus\bigoplus_{i=1}^{m-1}\ZZ e_i\oplus\bigoplus_{i=1}^{m-1}\ZZ e'_i$$
identifying the basis $\{ \beta_0,\beta_1,...,\beta_{m-1},\beta'_1,...,\beta'_{m-1} \}$ with $\{ e_0,e_1,...,e_{m-1},e'_1,...,e'_{m-1} \}$. Similarly we can therefore identify $H_1(L_G)$ with $N=\ZZ^n$ by identifying the basis $\{ \partial\beta_0,...,\partial\beta_{n-1} \}$ with $\{ v_0,...,v_{n-1} \}$.
%$$\bigoplus_{i=0}^{n-1}\ZZ(\partial\beta_i) \cong \bigoplus_{i=0}^{n-1}\ZZ v_i.$$

Consider, for $i=1,...,m-1$, we compute that $\beta'(e'_i-e_i+e_0)=v'_i-v_i+v_0=0$.
%and by definition of $v'_i:=v_i-v_0$, we have $$\beta'(e'_i-e_i+e_0)=0.$$
Geometrically, it means that $\alpha_i:=\beta'_i-\beta_i+\beta_0\in Ker(\partial) = H_2(X')$ has no boundary since
$\partial\alpha_i:=\partial\beta'_i-\partial\beta_i+\partial\beta_0=0$ for each $i=1,...,m-1$.
%Thus we have $$\alpha_i:=\beta'_i-\beta_i+\beta_0\in Ker(\partial) = H_2(X').$$
Let us define $C_i:=exp(-\int_{\alpha_i}\omega)\neq 0$ for $i=1,...,m-1$, we compute that
\begin{equation}
\begin{split}
\tilde{z}'_i(L_G,\nabla)
& \triangleq Z_{\beta'_i}(L_G,\nabla)
 %= exp(-\int_{\beta'_i}\omega)Hol_\nabla(\partial {\beta'_i})
  = exp(-\int_{\alpha_i+\beta_i-\beta_0}\omega) Hol_\nabla(\partial\beta_i-\partial\beta_0)\\
 & = exp(-\int_{\alpha_i}\omega)\frac{exp(-\int_{\beta_i}\omega)Hol_\nabla(\partial\beta_i)}{exp(-\int_{\beta_0}\omega)Hol_\nabla(\partial\beta_0)}
 %& = exp(-\int_{\alpha_i}\omega) \frac{Z_{\beta_i}(L_G,\nabla)}{Z_{\beta_0}(L_G,\nabla)}\\
 %& = C_i \hat{z}'_i(L_G,\nabla)
  = C_i \hat{z}'_i(L_G,\nabla)
\end{split}
\end{equation}
Notice that $C_i:=exp(\int_{\alpha_i}\omega)$ for $i=1,...,m-1$ is independent of the choice of $(L_G,\nabla)\in\MM'_0$. Equivalently, $C_i:\MM'_0\longrightarrow \CC$ can be considered as a constant map.
%Hence $\hat{z}'_i(L_G,\nabla)$ is just a scaler multiple of $\tilde{z}'_i(L_G,\nabla)$ by the constant $C_i$.We then obtain$$C_i \tilde{z}'_i(L_G,\nabla) = \hat{z}'_i(L_G,\nabla) $$for $i=1,...,m-1$, as required.

For $i=0$ in particularly, we have
$\hat{z}'_0(L_G,\nabla)=\tilde{z}'_0(L_G,\nabla) =C_0 \tilde{z}'_0(L_G,\nabla)$ since $C_0:=1$ according to definition. Thus the proof is completed.

\begin{comment}
Thus $$(\hat{z}'_i)^{\pm 1} = C_i^{\mp 1}(\hat{z}'_i)^{\pm 1}\in\OO(\MM'_0)$$
for $i=1,...,m-1$. And we have
\begin{equation}
\begin{split}
\OO(\MM'_0)
& = < (\tilde{z}'_0)^{\pm 1},...,(\tilde{z}'_{m-1})^{\pm 1} > \\
& = < (\hat{z}'_0)^{\pm 1},(C_1\hat{z}'_{1})^{\pm 1},...,(C_{m-1}\hat{z}'_{m-1})^{\pm 1} > \\
& = < (\hat{z}'_0)^{\pm 1},(\hat{z}'_{1})^{\pm 1},...,(\hat{z}'_{m-1})^{\pm 1} >
\end{split}
\end{equation}
\end{comment}

\end{proof}

The aim of this section is to construct a subring of function in $\OO(\MM'_+\sqcup \MM'_-)$. Recall that the natural complex structure on $\MM'_0$ is given by the functions $Z_\beta(L,\nabla)=exp(-\int_\beta \omega)Hol_\nabla(\partial \beta)$ for $\beta\in H^{eff}_2(X',L;\ZZ)$. Before constructing the subring, we give a representation of the ring $\OO(\MM'_+\sqcup \MM'_-)$ as follows:

\begin{prop}\label{definition of g = g+ x g-}
Let $\MM'_0$ be the toric modified mirror moduli of $X_\Delta$ with the natural complex structure attached to it. Then we have the ring isomorphism:
\begin{align}
\CC[z_0^\pm,...,z_{n-1}^\pm]\times\CC[z_0^\pm,...,z_{n-1}^\pm] &\overset{g}{\longrightarrow} \OO(\MM'_+\sqcup \MM'_-)\text{  given by  }\\
(z_i,z_j)&\overset{g_+\times g_-}{\longmapsto} ( Z_{\beta_i}|_{\MM'_+} ,Z_{\beta_j}|_{\MM'_-} ).
\end{align}
\end{prop}

\begin{proof}
Notice that %the maps
%\begin{align}	g_+:\CC[z_0^\pm,...,z_{n-1}^\pm]\longrightarrow\MM'_+:z_i\longmapsto Z_{\beta_i}|_{\MM'_+} \\	g_-:\CC[z_0^\pm,...,z_{n-1}^\pm]\longrightarrow\MM'_-:z_i\longmapsto Z_{\beta_i}|_{\MM'_-}	\end{align}	are basically the same and
the ring
$ \OO(\MM'_+\sqcup \MM'_-)$
is nothing but a Cartesian product of rings $\OO(\MM'_+) \times \OO(\MM'_-)$.
To prove that the Cartesian product of maps $g=g_+\times g_-$ is an isomorphism is equivalent to prove that $g_+$ and $g_-$ are both isomorphisms. Without loss of generality, we just show that $g_+$ is both injective and surjective.

To show that the map is injective, we assume that there is a polynomial $P\in \CC[z_0^\pm,...,z_{n-1}^\pm]$ such the $g_+(P)=0$. We want to claim that $P$ vanishes. We now have $g_+(P)(L,\nabla)=0$ for any $(L,\nabla)\in \MM'_+$. In particularly, suppose we fix $L$, we have $g_+(P)(L,\nabla)=0$ for any flat $U(1)$ connection over $L$. In other words, $P$ vanishes on the distinguished boundary of some polydisc in $\CC^n$. By Cauchy's integral formula in several complex variables, $P$ vanishes.

We now show that the map $g_+$ is surjective. Firstly, we let $\{v_0^*,...,v_{n-1}^*\}\in M$ be the dual basis of $\{v_0,...,v_{n-1}\}\in N$. %According to the commutative diagram \eqref{beta and vi}, we have, for $j=n,...,m-1$, %and $$1=<u,v_j>=\sum_{i=0}^{n-1}<v_i^*,v_j><u,v_i>=\sum_{i=0}^{n-1}<v_i^*,v_j>.$$
We have
$ \partial\beta_j=\sum_{i=0}^{n-1}<v_i^*,v_j>\partial\beta_i $  and
$ \beta_j=\sum_{i=0}^{n-1}<v_i^*,v_j>\beta_i + \alpha_j, \exists\alpha_j\in H_2(X;\ZZ)$ since $ v_j=\sum_{i=0}^{n-1}<v_i^*,v_j>v_i $ for $j=n,...,m-n$.
%We substitute $<v_0^*,v_j>=1-\sum_{i=1}^{n-1}<v_i^*,v_j>$. Then
%$$\beta_j-\beta_0=\sum_{i=1}^{n-1}<v_i^*,v_j>\beta_i + (<v_0^*,v_j>-1)\beta_0+\alpha_j=\sum_{i=1}^{n-1}<v_i^*,v_j>(\beta_i-\beta_0)+\alpha_j$$
%$$<v_0^*,v_j>=1-\sum_{i=1}^{n-1}<v_i^*,v_j>$$
%Let us define $Q_j:=exp(-\int_{\alpha_j}\omega_X)$ to be the K$\ddot{a}$hler parameters for $j=n,...,m-1$.
Let $Q_1,...,Q_{m-n}$ be the K$\ddot{a}$hler parameters which are defined by $Q_j:=exp(-\int_{\alpha_{j+n-1}\omega_X)}$ for $j=1,...,m-n$. We compute that
\begin{align} \label{compute zr-1 in terms of z1 to zn-1}
Z_{\beta_j}
%&=  exp(-\int_{\beta_j}\omega)Hol_\nabla(\partial\beta_j)\\
&= exp(-\int_{\sum_{i=0}^{n-1}<v_i^*,v_j>\beta_i + \alpha_j}\omega)Hol_\nabla(\sum_{i=0}^{n-1}<v_i^*,v_j>\partial\beta_i)\\
&= exp(-\int_{\alpha_j}\omega)\prod_{i=0}^{n-1}[exp(-\int_{\beta_i}\omega)Hol_\nabla(\partial\beta_i)]^{<v_i^*,v_j>}
=Q_j \prod_{i=0}^{n-1}Z_{\beta_i}^{<v_i^*,v_j>}
\end{align}

\begin{comment}
\eqt \begin{split}
\hat{z}_j &= \frac{Z_{\beta_j}}{Z_{\beta_0}}\\
&= exp(-\int_{\beta_j-\beta_0}\omega)Hol_\nabla(\partial\beta_j-\partial\beta_0)\\
&= exp(-\int_{\sum_{i=1}^{n-1}<v_i^*,v_j>(\beta_i-\beta_0)+\alpha_j}\omega)Hol_\nabla(\sum_{i=1}^{n-1}<v_i^*,v_j>(\partial\beta_i-\partial\beta_0))\\
&= exp(-\int_{\alpha_j}\omega)\prod_{i=1}^{n-1}[exp(-\int_{(\beta_i-\beta_0)}\omega)Hol_\nabla(\partial\beta_i-\partial\beta_0)]^{<v_i^*,v_j>}\\
&=q_j \prod_{i=1}^{n-1}(\frac{Z_i}{Z_0})^{<v_i^*,v_j>}\\
&=q_j \prod_{i=1}^{n-1}\hat{z}_i^{<v_i^*,v_j>}
\end{split} \eqtn
\end{comment}

Immediately, we have that the map $g_+$ is surjective since we can check that
$g_+(q_j \prod_{i=0}^{n-1}z_i^{<v_i^*,v_j>})=%q_j \prod_{i=0}^{n-1}g_+(z_i)^{<v_i^*,v_j>}=
Q_j \prod_{i=1}^{n-1}Z_{\beta_i}^{<v_i^*,v_j>}=Z_{\beta_j}$ for any $j=n,...,m-1$.

Overall, we have showed that the map $g_+$ is both injective and surjective. In other words, we obtained the isomorphism $g_+$ between the rings $\CC[z_0^\pm,...,z_{n-1}^\pm]$ and $\OO(\MM'_+)$.
%$ g_+:\CC[z_0^\pm,...,z_{n-1}^\pm] \overset{\sim}{\longrightarrow} \OO(\MM'_+)$.
By the exact same argument, we can also show that $g_-$ is also an isomorphism. Hence we obtain the result that $g=g_+\times g_-$
%$$ g:\CC[z_0^\pm,...,z_{n-1}^\pm]\times \CC[z_0^\pm,...,z_{n-1}^\pm] \overset{g_+\times g_-}{\longrightarrow} \OO(\MM'_+\sqcup\MM'_-)$$
is an isomorphism. And thus the proof is completed.
\end{proof}

With the modified coordinates $\hat{z}'_i$ and ring representation of $\OO(\MM'_+\sqcup \MM'_-)$ at hand, we are ready to construct the subring of function as follows:

\begin{prop} \label{well defined in R} \label{R injection to O(M)}
Let $\hat{z}'_0,...,\hat{z}'_{m-1}:\MM'_0\longrightarrow\CC$ be functions as defined in Proposition \ref{z hat definition}. Then, when the functions are restricted to $\MM'_+\sqcup\MM'_-\subset \MM'_0$, we have
$$ \hat{z}'_0,...,\hat{z}'_{m-1} \in \OO(\MM'_+\sqcup\MM'_-).$$

Moreover, there is a representation of the subring $R$ in $ \OO(\MM'_+\sqcup\MM'_-)$ which is defined to be finitely generated by $ \{ \hat{z}'_0,...,\hat{z}'_{m-1} \}$, that is,

%Moreover, under the representation of the ring $\OO(\MM'_+\sqcup\MM'_-)$, we have
$$ R \simeq \CC[z_0^\pm,...,z_{n-1}^\pm]\times_f\CC[z_0^\pm,...,z_{n-1}^\pm]\simeq \frac{\CC[u^\pm,v^\pm,z_1^\pm,...,z_{n-1}^\pm]}{<uv=W>}$$

where the localization $f=(f_1,f_2)$ is given by
\begin{align*}
f_1:\CC[z_0^\pm,...,z_{n-1}^\pm]&\longrightarrow\CC[z_0^\pm,...,z_{n-1}^\pm]\\
z_0&\longmapsto z_0\\
%z_1&\longmapsto z_1\\
&\vdots \\
z_{n-1}&\longmapsto z_{n-1}
\end{align*}
and
\begin{align*}
f_2:\CC[z_0^\pm,...,z_{n-1}^\pm]&\longrightarrow\CC[z_0^\pm,...,z_{n-1}^\pm]\\
z_0&\longmapsto z_0W(\frac{z_1}{z_0},...,\frac{z_{n-1}}{z_0})\\
z_1&\longmapsto z_1W(\frac{z_1}{z_0},...,\frac{z_{n-1}}{z_0})\\
&\vdots \\
z_{n-1}&\longmapsto z_{n-1}W(\frac{z_1}{z_0},...,\frac{z_{n-1}}{z_0}).
\end{align*}
for some $W\in\CC[z_1^\pm,...,z_{n-1}^\pm]$, such that,
$$W\simeq(1+\delta_0)+\sum_{i=1}^{n-1}(1+\delta_i)z_i + \sum_{i=n}^{m-1}Q_{i-n+1}(1+\delta_i)\prod_{j=1}^{n-1}z_j^{<v^*_i,v_j>}$$
where the $\delta_i$'s are the disc generating function as defined in \eqref{definition of delta i}.
\end{prop}

\begin{rmk}
The previous theorem requires localization of rings. To recall the localization of rings, suppose $A,B,C$ are rings and $$f_1:A\longrightarrow C \ \ \ ;\ \ \  f_2:B\longrightarrow C $$ are ring homomorphisms. Denote $f=(f_1,f_2)$. We have
$$A\times_f B \trieq \{(a,b)\in A\times B\ \  | \ f_1(a)=f_2(b)\in C\}$$
\end{rmk}

\begin{proof}
According to Theorem \ref{thm z tilde with defn of delta} and Proposition \ref{z hat definition}, we have
%$$\hat{z}'_0(L,\nabla)= \tilde{z}'_0(L,\nabla)=Z_{\beta_0}(L,\nabla)$$$$ \hat{z}'_0(L,\nabla)=\tilde{z}'_0(L,\nabla)=\sum_{i=0}^{m-1}(1-\delta_i)Z_{\beta_i}(L,\nabla)$$$$\hat{z}'_i(L,\nabla)  = \frac{Z_{\beta_i}(L,\nabla)}{Z_{\beta_0}(L,\nabla)}$$
\begin{align}
	\hat{z}'_0(L,\nabla)&=
    	\begin{cases}
			Z_{\beta_0}(L,\nabla), \text{for $(L,\nabla)\in \MM'_-$},\\
			\sum_{i=0}^{m-1}(1-\delta_i)Z_{\beta_i}(L,\nabla) \text{for $(L,\nabla)\in\MM'_+$}
		\end{cases}	\\
	 \hat{z}'_j(L,\nabla)  &= \frac{Z_{\beta_j}(L,\nabla)}{Z_{\beta_0}(L,\nabla)}  \text{,for any $(L,\nabla)\in$ and $j=1,...,m-1$}
\end{align}
which are all expressed in terms of the semi-flat coordinates $Z_\beta(L,\nabla)$ when restricted to the subspace $\MM'_+\sqcup\MM'_-$. Hence the first part is proved, that is, $\hat{z}'_0,...,\hat{z}'_{m-1}\in\OO(\MM'_+\sqcup\MM'_-)$.

For the proof of the second part of the theorem, that is to show that
$$ R \simeq \CC[z_0^\pm,...,z_{n-1}^\pm]\times_f\CC[z_0^\pm,...,z_{n-1}^\pm]\simeq \frac{\CC[u^\pm,v^\pm,z_1^\pm,...,z_{n-1}^\pm]}{<uv=W>},$$ we refer to chapter 4.6.3 of Chan-Lau-Leung\cite{Chan Lau Leung}.

%\begin{enumerate}\item	for $b_2<|\epsilon|^2$ and $i=1,...,r-1$,\eqt \tilde{z}'_i(L,\nabla)=Z_{\beta_i}(L,\nabla), \eqtn\eqt \tilde{z}'_0(L,\nabla)=Z_{\beta_0}(L,\nabla);\eqtn\item	for $b_2>|\epsilon|^2$ and $i=1,...,r-1$,\eqt \tilde{z}'_i(L,\nabla)=Z_{\beta_i}(L,\nabla), \eqtn\eqt \tilde{z}'_0(L,\nabla)=\sum_{i=0}^{r-1}(1-\delta_i)Z_{\beta_i}(L,\nabla).\eqtn	where \eqt \label{definition of delta i}\delta_i:=\sum_{\alpha>0} n_{\beta_i+\alpha}exp(-\int_\alpha\omega) \eqtn	which sums over $\alpha\in H_2(X';\ZZ)$ that can be represented by non-trivial rational curves of Chern number zero.\end{enumerate}

\end{proof}

We now have the geometric realization of Spec$R$ as
$$ \text{Spec$R$} = \{(u,v,z_1,...,z_{n-1})\in (\CC^\times)^2 \times (\CC^\times)^{n-1} | uv=W(z_1,...,z_{n-1}) \} .$$

As motivated and supported by the physics literature and the SYZ construction, we give the following definition.

\begin{defn} \label{main def A}
Let $X_\Delta$ be a toric Calabi-Yau manifold. We define
the mirror variety $X^\vee$ of $X_\Delta$ be the partial compactification of Spec$R$ as
$$X^\vee = \{(u,v,z_1,...,z_{n-1})\in \CC^2 \times (\CC^\times)^{n-1} | uv=W(z_1,...,z_{n-1}) \} .$$
\end{defn}

According to the ring homomorphism defined in Proposition \ref{R injection to O(M)}, that is,
$$ R \overset{\varphi}{\hookrightarrow} \OO(\MM_+\sqcup\MM_-),$$
we have the pull-back map
$$\MM_+\sqcup\MM_- \overset{\varphi^*}{\longrightarrow} \text{Spec$R$} \subset X^\vee.$$
By carefully tracing the definition of the ring homomorphism
$$\varphi : \CC[u^\pm,v^\pm,z_1^\pm,...,z_{n-1}^\pm] \longrightarrow R \longrightarrow \OO(\MM_+\sqcup \MM_-).$$
we can easily and directly express its pull-back map $\varphi^*$ explicitly as follows:

\begin{thm} \label{varphi map moduli to X check} \label{main thm A}
Let $\MM_0$ be the mirror moduli of $X_\Delta$ and $\varphi:R\longrightarrow\OO(\MM_+\sqcup\MM_-)$ be the ring homomorphism as defined in Proposition \ref{R injection to O(M)}.
Then the pull-back map $$\MM_+\sqcup\MM_- \overset{\varphi^*}{\longrightarrow} \text{Spec$R$} \subset X^\vee,$$ is given by
$$(L,\nabla) \overset{\varphi^*}{\longrightarrow} (u,v,z_1,...,z_{n-1})$$
where
\begin{align}
u&=\hat{z}_0(L,\nabla)\\
v&=\frac{W(\hat{z}_1(L,\nabla),...,\hat{z}_{n-1}(L,\nabla))}{\hat{z}_0(L,\nabla)}\\
z_i&=\hat{z}_i(L,\nabla),\ \textrm{ for } i=1,...,n-1.
\end{align}
%for $i=1,...,n-1.$
for any $(L,\nabla)\in\MM_+\sqcup\MM_-$.
\end{thm}

\begin{rmk}
According to predictions in Hori-Iqbal-Vafa \cite{Hori Iqbal Vafa}, the mirror of a toric Calabi-Yau manifold $X_\Delta$ should be defined by the polynomial
%H-I-Vafa has constructed the mirror of a toric Calabi-Yau manifold as
$$uv=1+\sum_{i=1}^{n-1}z_i+\sum_{j=n}^{m-1}q_j\prod_{i=1}^{n-1}z_i^{<v_i^*,v_j>}$$
by physical considerations.
Although the prediction is made, we realize that the equation should be modified by instantons in term of the symplectic information of $X_\Delta$.
In the above SYZ construction, we have re-constructed the mirror.
Notice that the mirror we constructed via SYZ involves quantum correction using open Gromov-Witten invariants.
%Notice that the correct mirror should be corrected by the open Gromov-Witten invariants.
Thus the correct mirror variety modifying the mirror predicted by the physicists should be defined by the following polynomial:
$$ uv=(1+\delta_0)+\sum_{i=1}^{n-1}(1+\delta_i)z_i+\sum_{j=n}^{m-1}(1+\delta_j)Q_{j-n+1}\prod_{i=1}^{n-1}z_i^{<v_i^*,v_j>}.$$
\end{rmk}

Thus the construction of SYZ mirror variety for toric Calabi-Yau manifolds is completed. Next, we move on and study SYZ transformation for Dirichlet branes.

\chapter{Naive SYZ transform for D-branes}
In this chapter, we perform SYZ transform to certain class of special Lagrangians and turns out that the SYZ mirror brane recovers physicists' prediction for such class of A-branes. According to the philosophy of the SYZ programme that mirror symmetry is nothing but T-duality, there is a natural mirror construction for Lagrangians obtained via the conormal bundle construction as studied in \cite{Chan,Leung Yau Zaslow}. Thus we are ready to perform the SYZ mirror brane construction to certain class of special Lagrangians which is well studied among physicists as in \cite{Aganagic Vafa,Lerche Mayr Warner}. Moreover, the SYZ mirror subvariety we have constructed agrees with the prediction given by Aganagic and Vafa in \cite[Section 3.1]{Aganagic Vafa}

To begin with, we start by introducing the idea behind the SYZ mirror brane construction that is well described in \cite{Chan,Leung Yau Zaslow}.
%-------------------------------------------------------------------------------------------------

\section{T-duality for D-branes}
In this section, we define the SYZ transformation for A-branes. Our goal is to recover physicists prediction using SYZ transformation. Motivated by the SYZ transformation for Lagrangian fibers and Lagrangian chapters in Leung Yau Zaslow\cite{Leung Yau Zaslow}, a generalised version of such transform is studied in Chan\cite{Chan}. Given some A-brane $(L,\nabla)$, the SYZ transformation provide an explicit construction for the B-brane $(L^\vee,\nabla^\vee)$, that is,
$$ (L,\nabla) \overset{\mathcal{F}^{SYZ}}\longmapsto (L^\vee,\nabla^\vee). $$
The SYZ transformation $\mathcal{F}^{SYZ}$ that we are going to construct have important consequences in homological mirror symmetry. The SYZ transformation $\mathcal{F}^{SYZ}$ induces an isomorphism between the objects of the triangulated categories:
$$ \mathcal{F}^{SYZ}:\mathcal{D}^b\mathcal{F}uk(X) \overset{\simeq}\longrightarrow \mathcal{D}^b\mathcal{C}oh(X^\vee) $$
for some mirror pair $(X,X^\vee)$ which was verified in Chan\cite{Chan}.

To begin with, we consider an n-dimensional integral affine manifold $B$ which transition maps are given by elements in $Aff(\ZZ^n):=GL(\ZZ,n)\times\ZZ^n$. Let $\{x_1,...,x_n\}$ be local affine coordinates of $B$ and $\Lambda\subset TB$ be the lattice sub-bundle generated by $\{\frac{\partial}{\partial x_1},...,\frac{\partial}{\partial x_n}\}$. On the cotangent bundle $T^*B$, we consider the lattice sub-bundle $\Lambda^\vee \subset T^*B$ generated by $\{dx_1,...,dx_n\}$. Based on the SYZ program, we consider the following pair of manifolds together with their fibrations on $B$:
\begin{align}
X &:=\frac{T^*B}{\Lambda^\vee} \label{X=T*B/L*} \overset{\mu}\longrightarrow B\\
X^\vee &:=\frac{TB}{\Lambda}\label{X=TB/L} \overset{\mu^\vee}\longrightarrow B
\end{align}

Notice that $X$ has a natural symplectic structure attached to its affine structure on the base $B$.

\begin{defn}
Let $X$ be defined as \eqref{X=T*B/L*}. Let $(x_1,...,x_n,\xi_1,...,\xi_n)$ be local coordinates on $T^*B$  denotes the cotangent vector $\sum_j \xi_j dx_j $ at the point $(x_1,...,x_n)\in B$. We define the canonical symplectic form on $X$ by
$$\omega := \sum_j dx_j \wedge d\xi_j .$$
\end{defn}

On the other hand, $X^\vee$ has a natural complex structure attached to its affine structure on the base $B$.

\begin{defn}
Let $X^\vee$ be defined as \eqref{X=TB/L}. Let $(x_1,...,x_n,y_1,...,y_n)$ be local coordinates on $T^B$  denotes the tangent vector $\sum_j y_j \frac{\partial}{\partial x_j} $ at the point $(x_1,...,x_n)\in B$. We define the canonical complex coordinates on $X^\vee$ by
$$z_i := exp( 2\pi(x_j + iy_j )) \ \ \ for\ j=1,...,n$$
\end{defn}

According to the construction in Chan\cite{Chan} and for the sake of the geometric case which we will consider in later sections, we restrict ourselves to some specific class of Lagrangians $L\subset X$ of the following form:
\begin{align} \label{Lagrangian L in T*B/L*}
L=\{(x_1,...,x_n,\xi_1,...,\xi_n) \in X |
&(x_1,...,x_k)=\in \RR^k  ;\\
&(x_{k+1},...,x_n)=(c_{k+1},...,c_n);\\
&(\xi_1,...,\xi_k)=(\xi_1(x_1,...,x_k),...,\xi_k(x_1,...,x_k));\\
&(\xi_{k+1},...,\xi_n)\in \RR^{n-k}\}
\end{align}
where $0\leq k\leq n$, $c_{k+1},...,c_n\in \RR$ be constants and $\xi_1(x_1,...,x_k),...,\xi_k(x_1,...,x_k)$ be smooth functions. Let us denote the affine subspace by $I\subset B$, defined by setting $(x_{k+1},...,x_n)=(c_{k+1},...,c_n)$ locally.

Moreover, we consider flat $U(1)$ connections $\nabla$ on such Lagrangian $L$, that is,
\begin{align}
\nabla = d + 2\pi i (\sum_{j=1}^k a_jdx_j + \sum_{l=k+1}^n b_ld\xi_l )  \label{a=a(x,...,x) definition}
\end{align}
where $a_j = a_j(x_1,...,x_k),j=1,...,k$ are smooth functions and $b_l\in \RR , l= k+1,...,n$ are constants.

The following SYZ transform requires the identification of the dual mirror manifold $X^\vee=\frac{TB}{\Lambda}$ with the moduli space of connections $\MM$ as Definition \ref{SYZ definition of M_0} in the previous chapter.

\begin{defn}\label{SYZ transformation of branes}
Let the A-brane $(L,\nabla)$ as defined in \eqref{Lagrangian L in T*B/L*} and \eqref{a=a(x,...,x) definition}, we define the SYZ transformation $(L^\vee,\nabla^\vee) = \mathcal{F}^{SYZ}(L,\nabla)$ as follows:
\begin{align}
L^\vee := \{ &(F,\nabla_F)\in \MM \simeq X^\vee | \mu(F)\in I \subset B ;\\
& \text{$\nabla_F$ twisted by $\nabla|_F$ is trivial when restricted to $F\cap L$}\}.
\end{align}
And the connection $\nabla^\vee$ on $L^\vee$ is constructed by reversing the construction process. Namely, we let the connection to be of the form
\begin{align}
\nabla^\vee = d + 2\pi i (\sum_{j=1}^k a_jdx_j + \sum_{l=k+1}^n \beta_ldy_l ) ,
\end{align}
where $a_j = a_j(x_1,...,x_k),j=1,...,k$ as in \eqref{a=a(x,...,x) definition} and $\beta_l = \beta_l(x_1,...,x_k),j=1,...,k$ be some smooth functions such that
$\nabla_{F^\vee}$ twisted by $\nabla^\vee|_{F^\vee}$ is trivial when restricted to $F^\vee = \mu^\vee(b)$ intersecting with $L^\vee$, for $b\in I\subset B$.
\end{defn}

The SYZ transform has an important consequence in Homological Mirror Symmetry. Moreover, the following remarks reflect that our construction for $\mathcal{F}^{SYZ}$ makes sense.

\begin{rmk}
According to Chan\cite{Chan,Chan Ueda,Chan Ueda Pomerleano}, consider the case where
$$X=\{(u,v,z)\in\CC\times\CC\times\CC^\times | uv=P(z) \}$$ for some polynomial $P(z)$ together with the fibration $\mu$
$$\mu : (u,v,z) \longmapsto (log|z| , (|u|^2 - |v|^2 ) ) .$$
And $X^\vee$ be the corresponding SYZ mirror.
It has been shown that the SYZ transformation induces an equivalence between the triangulated categories, that is,
$$ \mathcal{F}^{SYZ}:\mathcal{D}^b\mathcal{F}uk(X) \overset{\simeq}\longrightarrow \mathcal{D}^b\mathcal{C}oh(X^\vee).$$
\end{rmk}

Back to the A-brane $(L,\nabla)$ that we are considering as defined in \eqref{Lagrangian L in T*B/L*} and \eqref{a=a(x,...,x) definition}. Under the SYZ transformation in Definition \ref{SYZ transformation of branes}, we can explicitly compute the mirror brane $(L^\vee,\nabla^\vee)$ as follows:

\begin{lem} \label{toy model mirror brane}
Let $X$ be a symplectic manifold as in \eqref{X=T*B/L*}, $L$ be a Lagrangian as in \eqref{Lagrangian L in T*B/L*} and $\nabla$ be a flat $U(1)$ connection on $L$ as in \eqref{a=a(x,...,x) definition}. Then the SYZ transformation of the A-brane $(L,\nabla)$ is given by $(L^\vee,\nabla^\vee)=\mathcal{F}^{SYZ}(L,\nabla)$ which is locally given by
\begin{align}
L^\vee &= \{(z_1,...,z_n)\in (\CC^\times)^n|z_j=exp(2\pi (c_j-ib_j),j=k+1,...n\},\\
\nabla^\vee&=d+2\pi i(\sum_{j=1}^k a_jdx_j + \sum_{j=1}^k \xi_j(x_1,...,x_k)dy_j)
\end{align}
\end{lem}

\begin{proof}
By the identification between $X^\vee$ and the moduli of connections as in Definition \ref{SYZ definition of M_0},we identify the point $(x_1,...,x_n,y_1,...,y_n)\in X^\vee $ with $(F,\nabla)$ where
\begin{align}
F&:=\mu^{-1}(x_1,...,x_n) \subset X,\\
\nabla &:= d+2\pi i (\sum_{j=1}^n y_j d\xi_j ).
\end{align}
According to the Definition \ref{SYZ transformation of branes}, we set $\mu(F)\in I$ and $\nabla_F$ twisted by $\nabla|_F$ is trivial when restricted to $F$.

Notice that $\mu(F)\in I$ is equivalent to set $(x_{k+1},...,x_n)=(c_{k+1},...,c_n)$ where there is no condition imposed on $(x_1,...,x_k)\in\RR^{k}$.

Moreover, by setting $\nabla_F$ twisted by $\nabla|_F$ is trivial when restricted to $F$ is equivalent to the condition that
\begin{align}
d &= \nabla_F - (\nabla|_F-d)\\
&= d+ 2\pi i ( \sum_{l=k+1}^n y_ld\xi_l) ) - (\sum_{l=k+1}^n b_ld\xi_l)\\
&=d+ 2\pi i ( \sum_{l=k+1}^n (y_l-b_l)d\xi_l).
\end{align}
This condition is equivalent to set $(y_{k+1},...,y_n)=(b_{k+1},...,b_n)$ with no condition imposed on $(y_1,...,y_k)\in\RR^{k}$.

Overall, we obtain that
$$L^\vee = \{(z_1,...,z_n)\in (\CC^\times)^n|z_j=exp(2\pi (c_j-ib_j),j=k+1,...n\}.$$

On the other hand, we apply similar argument to construct
$$\nabla^\vee=d+2\pi i(\sum_{j=1}^k a_jdx_j + \sum_{j=1}^k \xi_j(x_1,...,x_k)dy_j).$$ The condition that $\nabla^\vee_{F^\vee}$ twisted by $\nabla^\vee|_{F^\vee}$ is trivial when restricted to $F^\vee$ is equivalent to the condition that $\beta_l(x_1,...,x_k)=\xi_l(x_1,...,x_k)$ for $l=k+1,...,n$. Thus we have $$\nabla^\vee=d+2\pi i(\sum_{j=1}^k a_jdx_j + \sum_{j=1}^k \xi_j(x_1,...,x_k)dy_j)$$ as desired. Hence the proof is completed.

%\mu(F)\in I \subset B ;\\& \text{$\nabla_F$ twisted by $\nabla|_F$ is trivial when restricted to $F$}\}.

\end{proof}

With the explicit expression of $L,\nabla,L^\vee,\nabla^\vee$ at hand, we immediately obtain some exchange of information between symplectic geometric properties of the A-brane $(L,\nabla)$ and complex geometric properties of the B-brane $(L^\vee,\nabla^\vee)$.

\begin{prop}
Let the A-brane and B-brane be $(L,\nabla),(L^\vee,\nabla^\vee)$ as in the above Lemma \ref{toy model mirror brane}. %Then the followings are equivalent:
Then $L$ be a Lagrangian is equivalent to the condition that
$$\frac{\partial\xi_j}{\partial x_l}=\frac{\partial\xi_l}{\partial x_j}$$
for all $j,l=1,...,k$ and $\nabla$ is a flat $U(1)$ connection is equivalent to the condition that
$$\frac{\partial a_j}{\partial x_l}=\frac{\partial a_l}{\partial x_j}$$ for all $j,l=1,...,k$.
Moreover, the $(0,2)$ part $F^{(0,2)}$ of the curvature 2-form $F$ of $\nabla^\vee$ is trivial which defines a holomorphic line bundle over $L^\vee$.

%\begin{enumerate}
%\item $L$ is a Lagrangian and $\nabla$ is a flat $U(1)$ connection on $L$.
%\item $\frac{\partial\xi_j}{\partial x_l}=\frac{\partial\xi_l}{\partial x_j}$ and $\frac{\partial a_j}{\partial x_l}=\frac{\partial a_l}{\partial x_j}$ for all $j,l=1,...,k$.
%\item The $(0,2)$ part $F^{(0,2)}$ of the curvature 2-form $F$ of $\nabla^\vee$ is trivial.
%\end{enumerate}
\end{prop}

\begin{proof}
The proof of the statement can be found in Proposition 2.1, 2.2 and 2.3 of Chan\cite{Chan}.
\end{proof}

%\begin{proof}According to Proposition 2.1 in \cite{Chan}, $L$ be a Lagrangian is equivalent to the condition that$$\frac{\partial\xi_j}{\partial x_l}=\frac{\partial\xi_l}{\partial x_j}$$for all $j,l=1,...,k$. And by Proposition 2.2 in \cite{Chan}, $\nabla$ is a flat $U(1)$ connection is equivalent to the condition that$$\frac{\partial a_j}{\partial x_l}=\frac{\partial a_l}{\partial x_j}$$ for all $j,l=1,...,k$.We immediately obtain that (1) and (2) are equivalent.
%We also obtain that (3) is equivalent to (2) by the Proposition 2.3 in \cite{Chan}, since\begin{align} F^{(0,2)} &= \sum_{j,l=1}^k \frac{\partial a_j}{\partial x_l} + \sum_{j,l=k+1}^n \frac{\partial _j}{\partial x_l}\end{align}\end{proof}

\begin{rmk}
Notice that the SYZ transformation in Definition \ref{SYZ transformation of branes} concerns specific classes of A-brane $(L,\nabla)$. In fact, this SYZ transformation induces an SYZ transformation of an isomorphism class of $(L,\nabla)$, where $L$ is up to Hamiltonian isotopy and $\nabla$ is up to gauge equivalence, to a holomorphic line bundle $(L^\vee,\nabla^\vee)$.
\end{rmk}

With the SYZ transformation $\mathcal{F}^{SYZ}$ constructed as above, we are ready to apply this to explicit geometric examples which will be investigated in later sections.

%-------------------------------------------------------------------------------------------------

\section{Construction of Lagrangians}
In this section, we introduce the class of Lagrangians $L\subset X_\Delta$ which was studied extensively in the literatures, such as Aganagic-Vafa\cite{Aganagic Vafa}, Lerche-Mayr-Warner\cite{Lerche Mayr Warner} and Fang-Liu\cite{Fang Liu}. These are the geometric cases which we will apply our SYZ transformation $\mathcal{F}^{SYZ}$ on in later sections.

Let us first consider Lagrangians on the most basic K$\ddot{a}$hler manifold $\CC^m$. We recall that the K$\ddot{a}$hler form on $\CC^m$ is given by $\omega_{\CC^m} \trieq \sum_{j=0}^{m-1} dz_j\wedge d\bar{z}_j.$
%\begin{align}
%\omega_{\CC^n} \trieq \sum_{j=1}^n dz_j\wedge d\bar{z}_j.
%\end{align}
We also recall that $\CC^n$ has a natural fibration by
\begin{align}
\mu_{\CC^m}:\CC^m &\longrightarrow (\RR_{\geq 0})^m,\ \ by\\
(z_0,...,z_{m-1})&\longmapsto\frac{1}{2}(|z_0|^2,...,|z_{m-1}|^2)
\end{align}

Let $l^{(1)},...,l^{(k)}\in \ZZ^m$ be integer valued vectors. We construct the Lagrangian subspace $\tilde{L}\subset \CC^m$ using these vectors. These vectors are referred as "charges" and the coordinates $|z_j|^2$ are viewed as "fields" in the physics literatures \cite{Aganagic Vafa},\cite{Ooguri Vafa}.

\begin{defn} \label{define L tilde in C n}
Let $l^{(1)},...,l^{(k)},l_\perp^{(k+1)},...,l_\perp^{(m)}\in \ZZ^m$ be linearly independent and the set $\{l^{(1)},...,l^{(m)}\},\{l_\perp^{(k+1)},...,l_\perp^{(n)}\}$ are mutually orthogonal with respect to the standard inner product structure. We define the (Lagrangian) subspace $ \tilde{L}\subset \CC^m $ by
\begin{align}
\tilde{L}
= \left\{ (z_0,...,z_{m-1})\in\CC^m|
\begin{array}{ll}
\sum_{j=0}^{m-1} l^{(a)}_j\cdot r_j^2 = c^{(a)}, &\forall a=1,...,k \\
\sum_{j=0}^{m-1} (l^{(b)}_\perp)_j\cdot \theta_j = \phi^{(b)}, &\forall b=k+1,...,m
\end{array}
\right\}
\end{align}
for some constants $c^{(1)},...,c^{(k)},\phi^{(k+1)},...,\phi^{(m)}\in \RR$ where $z_j=r_je^{i\theta_j}$ for $j=0,...,m-1$.
\end{defn}

Indeed, the subspace $\tilde{L}$ is Lagrangian in $\CC^m$ with respect to its standard symplectic form $\omega$:

\begin{thm} \label{L is special lagrangian}
Let $\tilde{L}\subset \CC^m$ as in Definition \ref{define L tilde in C n}. Then $\tilde{L}$ is a Lagrangian. Moreover, $\tilde{L}$ is special if and only if $$l^{(a)}\cdot\mathds{1}=0 \4 a=1,...,k$$ where $\mathds{1}\trieq (1,...,1)\in \ZZ^m$
\end{thm}

\begin{proof}
Let $V\in T\tilde{L}\subset T\CC^m$ be a vector field given by
$$ V=\sum_{j=0}^{m-1} (f_j \frac{\partial}{\partial r_j^2} +g_j \frac{\partial}{\partial \theta_j}). $$
Notice that $\tilde{L}$ is defined by equations, the coefficients $f_j$'s and $g_j$'s are under the following constrains:
$$ d(\sum_{j=0}^{m-1} l^{(a)}_j\cdot r_j^2)\cdot V=0\ \ ;\ \ d(\sum_{j=0}^{m-1} (l^{(b)}_\perp)_j\cdot \theta_j)\cdot V=0$$
which is equivalent to the constrain that
$ \sum_{j=0}^{m-1} l^{(a)}_j\cdot f_j=0\ \ ;\ \ \sum_{j=0}^{m-1} (l^{(b)}_\perp)_j\cdot g_j)=0$ for any $a=1,...,k$ and $b=k+1,...,m$.
In other words, we obtain that
$$ (f_0,...,f_{m-1}) \in \RR<l^{(k+1)}_\perp,...,l^{(m)}_\perp>\ \ ;\ \
(g_0,...,g_{m-1}) \in \RR<l^{(1)},...,l^{(k)}>.$$
Therefore, for any $V_1,V_2\in T\tilde{L}$, we can compute that
\begin{align}
&\omega_{\CC^m}(V_1,V_2) \\
&=(\sum_{j=0}^{m-1} dr_j^2 \wedge d\theta_j)( (\sum_{j=0}^{m-1} (f_j^{(1)} \frac{\partial}{\partial r_j^2} +g_j^{(1)} \frac{\partial}{\partial \theta_j}),(\sum_{j=0}^{m-1} (f_j^{(2)} \frac{\partial}{\partial r_j^2} +g_j^{(2)} \frac{\partial}{\partial \theta_j})) \\
&=\sum_{j=0}^{m-1} f_j^{(1)}\cdot g_j^{(2)} + \sum_{j=0}^{m-1} f_j^{(2)}\cdot g_j^{(1)}
=0
\end{align}
since $\{(f^{(0)}_1,...,f^{(1)}_{m-1}),((f^{(2)}_0,...,f^{(2)}_{m-1})\}$ and $\{(g^{(0)}_1,...,g^{(1)}_{m-1}),((g^{(2)}_0,...,g^{(2)}_{m-1})\}$ are mutually orthogonal. Thus $\tilde{L}$ is a Lagrangian in $\CC^m$ as desired.

We now show that $\tilde{L}$ is special.
Let $p=(R^2_0,...,R_{m-1}^2,\Theta_0,...,\Theta_{m-1})\in \tilde{L},$ and $U$ be a neighborhood of the origin in $\RR^m$ being the coordinate system of $p\in\tilde{L}$ given by
\begin{align}
\varphi':U\longrightarrow \tilde{L}:(t_1,...,t_m)&\longmapsto (r^2_0,...,r^2_{m-1},\theta_0,...,\theta_{m-1})\\
\theta_j &= \Theta_j+\sum_{s=1}^k t_sl_j^{(s)}\\
r^2_j &= R^2_j+\sum_{s=k+1}^m t_s(l_\perp^{(s)})_j
%(t_1,...,t_r)&\longmapsto(R^2_1+\sum_{s=1}^kt_s(l_\perp^{(s)})_1,...,R^2_r+\sum_{s=1}^kt_s(l_\perp^{(s)})_r,\Theta_1+\sum_{s=1}^kt_sl_1^{(s)},...,\Theta_r+\sum_{s=1}^kt_sl_r^{(s)})
\end{align}
for $j=0,...,m-1$. Therefore we obtain the map
\begin{align}
\varphi:U\longrightarrow \tilde{L}&\longrightarrow \CC^m\\
(t_1,...,t_m)&\longmapsto (z_0(t_1,...,t_m),...,z_{m-1}(t_1,...,t_m)),\ by\\
z_j(t_1,...,t_m)&=\sqrt{R^2_j+\sum_{s=k+1}^m t_s(l_\perp^{(s)})_j} exp(i(\Theta_j+\sum_{s=1}^kt_sl_j^{(s)}))
\end{align}
for $j=0,...,m-1$. We are now ready to compute $\Omega|_{\tilde{L}}$ by $\varphi^*(\Omega)$:
\begin{align}
\varphi^*(\Omega)
&= \varphi^*(dz_0\wedge...\wedge dz_{m-1})\\
&= \bigwedge_{j=0}^{m-1} dz_j(t_1,...,t_m)\\
&= \bigwedge_{j=0}^{m-1} \sum_{s=1}^m  \frac{\partial z_j}{\partial t_s}dt_s\\
&= \sum_{\sigma\in S_m} Sgn(\sigma) \prod_{s=1}^m \frac{\partial z_{\sigma(s)}}{\partial t_s}dt_1\wedge\cdots\wedge dt_m\\
&= \sum_{\sigma\in S_m} Sgn(\sigma)
	(\prod_{s=1}^k (l_\perp^{(\sigma(s))} )_s \frac{z_{\sigma(s)}}{r^2_{\sigma(s)}}
    \prod_{s=k+1}^m il_s^{(\sigma(s))}z_{\sigma(s)})%\frac{\partial z_{\sigma(s)}}{\partial t_s}
dt_1\wedge\cdots\wedge dt_m\\
&= i^{m-k}z_0...z_{m-1} \sum_{\sigma\in S_m} Sgn(\sigma)
	\prod_{s=1}^k (l_\perp^{(\sigma(s))} )_s \frac{1}{r^2_{\sigma(s)}}
    \prod_{s=1}^m il_s^{(\sigma(s)})%z_{\sigma(s)}%\frac{\partial z_{\sigma(s)}}{\partial t_s}
dt_1\wedge\cdots\wedge dt_m\\
\end{align}
We observe that
  $$ \prod_{j=0}^{m-1} z_j= r_0...r_{m-1} e^{i\sum_{j=0}^{m-1}\Theta_j}exp(i(\sum_{s=k+1}^m t_s\sum_{j=0}^{m-1} l^{(s)}_j)).$$
Therefore we have that
$$ Im(i^{m-k}exp(-i(\sum_{s=k+1}^m t_s\sum_{j=0}^{m-1}l^{(s)}_j))e^{i\sum_{j=0}^{m-1} \Theta_j}\Omega|_{\tilde{L}})=0 $$
Thus we conclude that $\tilde{L}$ is special if and only if $ \sum_{j=0}^{m-1}l^{(s)}_j=0$ for all $s=1,..,k$.
\end{proof}

Construction of Lagrangian $L$ in toric Calabi-Yau manifolds $X_\Delta$ is quite similar to that of $\tilde{L}$ in $\CC^m$. In order to replicate the construction process, we briefly recall the geometric realization of $X_\Delta$ as a symplectic quotient $\CC^m\sslash_t G_\RR$.

Let $\Sigma$ be a fan with $m$ rays $v_0,...,v_{m-1}\in N:=\mathds{R}^n$. We obtain a pair of dual short exact sequences
\begin{equation}
0\longrightarrow G\mathop{\longrightarrow}^\iota \mathds{Z}^m \mathop{\longrightarrow}^ \beta N\longrightarrow 0;
\end{equation}
\begin{equation}
0\longrightarrow M\mathop{\longrightarrow}^{\beta^\vee} {\mathds{Z}^m}^\vee  \mathop{\longrightarrow}^{\iota^\vee}G^\vee\longrightarrow 0.
\end{equation}
And we let $\Delta=\{\nu\in M_\RR:(\nu,v_i)\geq -\lambda_i \forall i\}$ be a polytope in $M_\RR$. We then obtain a toric variety $X_\Delta$ as in the previous chapter (...). We recall that $X_\Delta$ can be viewed as a symplectic quotient, that is,
$$X_\Delta \simeq \CC^m\sslash_t G_\RR $$ for some $t\in G_\RR$ as in (...).
To be more explicit, we consider
\begin{align}
	X&\trieq X_\Delta \simeq \CC^m\sslash_t G_\RR\\
		&=\frac{\mu_{\CC^m}^{-1}(t)}{G_\RR}\\
		&=
			\left\{ (r_0e^{i\theta_0},...,r_{m-1}e^{i\theta_{m-1}})\in\CC^m|
				\begin{array}{ll}
					\iota^\vee (r_0^2,...,r_{m-1}^2) = t \in G_\RR \\
                    (\theta_0,...,\theta_{m-1})\in \RR^m
				\end{array}
			\right\} \Bigr/ G_\RR\\
		&=
			\left\{ (r_0e^{i\theta_0},...,r_{m-1}e^{i\theta_{m-1}})\in\CC^m|
				\begin{array}{ll}
					\iota^\vee (r_0^2,...,r_{m-1}^2) = t \in G_\RR \\
                    (\theta_0,...,\theta_{m-1})\in G_\RR^\perp \subset  \RR^m
				\end{array}
			\right\}
\end{align}

Let we consider the map $\iota: \ZZ^{m-n} \simeq G\longrightarrow \ZZ^m$ as a $(n\times (m-n))$ integer valued matrix, that is, $\iota=(\iota_{ij})_{i=0,...,m-1;j=1,...,m-n}$. We obtain a representation of $X_\Delta$ as follows:
\begin{align}
X_\Delta \simeq \left\{ (r_0e^{i\theta_0},...,r_{m-1}e^{i\theta_{m-1}})\in\CC^m|
				\begin{array}{ll}
					\sum_{i=0}^{m-1} \iota_{ij}r_i^2 = t_j  \\
                    \sum_{i=0}^{m-1} \iota_{ij}\theta_i = 0
				\end{array}
                \4 j=1,...,k
			\right\}
\end{align}

Notice that we construct $X_\Delta$ by symplectic quotient of $\CC^m$. Roughly speaking, we perform the similar symplectic quotient process on $\tilde{L}$ to obtain a Lagrangian $L$ in $X_\Delta$.

\begin{defn} \label{definition of L in X}
Let $\tilde{L}\subset \CC^m$ and $X\simeq \CC^m\sslash_t G_\RR$ as defined above. We define $L\subset X$ by
\begin{align}
L\trieq \frac{\mu^{-1}(t) \cap \tilde{L} }{ G_\RR } \subset \frac{\mu^{-1}(t)  }{ G_\RR } \simeq X.
\end{align}
\end{defn}

Observe that $L$ depends strongly on $\tilde{L}$. It is natural to expect that some properties of $\tilde{L}$ should descend to $L$ as follows:

\begin{thm}
Suppose $\tilde{L}$ is a special Lagrangian in $\CC^m$, then $L$ is a special Lagrangian in $X$.
\end{thm}

\begin{proof}
We recall that $\CC^m$ has standard symplectic form $\omega_{\CC^m}$ and standard volume form $\Omega_{\CC^m}$. We also recall that $X_\Delta$ together with the symplectic form $\omega_X=\omega_{Red}$ is viewed as symplectic reduction of $(\CC^m,\omega_{\CC^m})$ .
To show that $L\subset X_\Delta$ is Lagrangian, we let $u,v\in TL \subset TX_\Delta$ with lift $\tilde{u},\tilde{v}\in T\tilde{L} \subset T\CC^m$. we consider
\begin{align}
\omega_X|_L(u,v) = \omega_{Red}(u,v) \trieq \omega_{\CC^m}(\tilde{u},\tilde{v})
= \omega_{\CC^m} |_{\tilde{L}}  (\tilde{u},\tilde{v}) =0
\end{align}

We recall that $X_\Delta$ together with the volume form $\Omega_X=\Omega_{GIT}$ is viewed as GIT quotient of $(\CC^m,\Omega_{\CC^m})$:
$$\psi:\CC^m \simeq Spec\CC[e_0,...,e_{m-1}] \overset{\beta^*}\longrightarrow Spec\CC[N] \simeq (\CC^\times)^n \overset{dense}\subset X_\Delta $$
where $\psi: (z_0,...,z_{m-1})\longmapsto (\zeta_1,...,\zeta_n)$ and $\zeta_i=\prod_{j=0}^{m-1} z_j^{v_j^{(i)}}$. By definition of the volume forms, we have
$$\Omega_{\CC^m}\trieq dz_0\wedge\cdots\wedge dz_{m-1}\ \ ,\ \ \Omega_X\trieq \frac{d\zeta_1}{\zeta_1}\wedge\cdots\wedge \frac{d\zeta_n}{\zeta_n}.$$
%Notice that we can choose some holomorphic function $\zeta_{n+1},...,\zeta_r$ on $\CC^r$ such that
%$$\psi^*(\Omega_X)=\frac{\partial}{\partial \zeta_{n+1}}\lrcorner\cdots\lrcorner\frac{\partial}{\partial \zeta_{r}}\lrcorner\Omega_{\CC^r}$$ since $$(\frac{\ZZ^r}{G}\times N) \simeq \ZZ^r \overset{\beta}\longrightarrow N$$

To show that $L\subset X_\Delta$ is special, we suppose that $V_1,...,V_n \in TL \subset TX_\Delta$ with lift $\tilde{V}_1,...,\tilde{V}_n\in T\tilde{L} \subset T\CC^m$. We consider
\begin{align}
(\Omega_X|_L)(V_1,...,V_n)
	&= \Omega_X(V_1,...,V_n) \\
    &= (\psi^*\Omega_X)(\tilde{V}_1,...,\tilde{V}_n)\\
    &= \bigwedge_{i=1}^n (\frac{d\prod_{j=0}^{m-1} z_j^{v_j^{(i)}}}{\prod_{j=0}^{m-1} z_j^{v_j^{(i)}}})(\tilde{V}_1,...,\tilde{V}_n)\\
    &= \sum_{0\leq i_1<...<i_n\leq (m-1)} c_{i_1,...,i_n} \frac{dz_{i_1}\wedge\cdots\wedge dz_{i_n}}{z_{i_1}...z_{i_n}} (\tilde{V}_1,...,\tilde{V}_n)\\
    %&= \frac{\partial}{\partial \zeta_{n+1}}\lrcorner\cdots\lrcorner\frac{\partial}{\partial \zeta_{r}}\lrcorner\Omega_{\CC^r}(\tilde{v}_1,...,\tilde{v}_n)\\
   % &= (\frac{\partial}{\partial \zeta_{n+1}}\lrcorner\cdots\lrcorner\frac{\partial}{\partial \zeta_{r}}\lrcorner\Omega_{\CC^r})|_{\tilde{L}}(\tilde{v}_1,...,\tilde{v}_n).
\end{align}
for some constants $c_{i_1,...,i_n}\in \RR$ for each $0\leq i_1<...<i_n\leq (m-1)$. Hence by similar arguments in Theorem \ref{L is special lagrangian}, we obtain that $Im(i^{n-k}\Omega_X|_L) =0$ for some and thus $L\subset X$ is special.

\end{proof}

With this explicit geometric example of A-model and A-brane $(X_\Delta,L,\nabla)$, together with the SYZ transform $\mathcal{F}^{SYZ}$ from the previous section, we are ready to perform $\mathcal{F}^{SYZ}$ on to $(X_\Delta,L,\nabla)$. It is an interesting question to ask what B-model and B-brane $(X_\Delta^\vee,L^\vee,\nabla^\vee)$ will be produced. We expect that the outcome will be the same as what physicists has predicted \cite{Aganagic Vafa} .

%-------------------------------------------------------------------------------------------------

\section{Lagrangian image}
In this chapter, our goal is to compute the image of the Lagrangian $L\subset X_\Delta$ on the base $B$ under the Gross fibration $\mu_G$. At the end of this section, we will obtain the result that
$$ \mu_G (L) = \frac{\mu_\Delta(L)}{\RR\cdot u} \times \RR_{\geq 0}.$$
This piece of information of $L$ is an essential part of data as we construct its SYZ mirror $L^\vee$ in the next section.

We recall that we first start with a polytope $\Delta\subset M$ where there exist a vector $u\in M$ such that $<u,v_i>=1$ for each ray $v_i$. This is exactly the Calabi-Yau condition as in Lemma \ref{CY condition on polytope}. We also naturally have a moment map $\mu_\Delta:X_\Delta \longrightarrow \Delta \subset M$. Moreover, let $\epsilon\in\CC^\times$ be fixed, we have the Gross fibration
\begin{align}
\mu_G : X_\Delta &\longrightarrow B := \frac{M}{\RR\cdot u} \times \RR_{\geq 0} \ \ \ \ by\\
x &\longmapsto ( [\mu_\Delta (x)] , |\chi^u(x)-\epsilon|^2 )
\end{align}
and we denote the first and second projection of $\mu_G$ by  $\mu_G^{(1)}$ and $\mu_G^{(2)}$ respectively.

In order to compute the image of $L$ under the Gross fibration $\mu_G$, we first compute $\mu_G^{(1)}(L)$:

\begin{prop} \label{proposition 1st projection of lagrangian image}
Let $\mu_G$ be the Gross fibration, $\mu_\Delta$ be the moment map and $L\subset X_\Delta$ be the Lagrangian as defined before. Then we have $$\mu_G^{(1)}(L) =  \frac{\mu_\Delta(L)}{\RR\cdot u}.$$
\end{prop}

\begin{proof}
The proof basically follows exactly from the definition of $\mu_G^{(1)}$. By definition of $$\mu_G^{(1)}=Proj_1 \circ \mu_G:X_\Delta\longrightarrow \frac{\mu_\Delta(L)}{\RR\cdot u}, $$ we have $\mu_G^{(1)}(x)= [\mu_\Delta (x)]$ for any $x\in X_\Delta$. As a result, we obtain that
$$\mu_G^{(1)}(L) = [\mu_\Delta (L)] = \frac{\mu_\Delta(L)}{\RR\cdot u}$$ as desired.
\end{proof}

Immediately, we obtain a direct consequence of the above proposition as follows:

\begin{cor} \label{corollary 1st projection of lagrangian image}
With the notations as in Proposition \ref{proposition 1st projection of lagrangian image}, we have the intersection of $(\mu_G^{(1)})^{-1}([p])$ and $L$ is non-empty for each $p\in\mu_\Delta(L)$ .
\end{cor}

Notice that $\mu_G=\mu_G^{(1)} \times \mu_G^{(2)} $ is just a Cartesian product of maps. To compute $\mu_G(L)$, we also need some result of the image of $L$ under the second projection of the Gross fibration as follows:

\begin{prop} \label{proposition 2nd projection of lagrangian image}
Let $\chi^u:X\longrightarrow\CC$ be the holomorphic function as in Proposition \ref{Proposition compute CHI is product of zs }. Then we have
$$ \chi^u ( (\mu_G^{(1)})^{-1}([p]) \cap L ) = \RR_{\geq 0} \cdot e^{i\Phi}$$
for some unit direction $e^{i\Phi}\in U(1) \subset \CC $.
\end{prop}

\begin{proof}
We recall that the Lagrangian has a representation that
\begin{align}
L\simeq  \left\{
	\begin{array}{lll}
		(r_0,...,r_{m-1})\in (\RR_{\geq 0})^m  & \ \sum_{j=0}^{m-1} l^{(a)}_j\cdot r_j^2 = c^{(a)} & \sum_{i=0}^{m-1} \iota_{ij}r_i^2 = t_j\\
        (\theta_0,...,\theta_{m-1})\in \RR^m    & :\sum_{j=0}^{m-1} (l^{(b)}_\perp)_j\cdot \theta_j = \phi^{(b)} &  \sum_{i=0}^{m-1} \iota_{ij}\theta_i = 0 \\
       & \ for\ a=1,...,k;\ b=k+1,...,m; &j=1,...,m-n
 	\end{array}
     \right\}
\end{align}
according to the definition of $L= (\mu^{-1}(t) \cap \tilde{L}) / G_\RR$ as in Definition \ref{definition of L in X}. Moreover, we can compute that
\begin{align} \label{eqt r square condition}
(\mu_G^{(1)})^{-1}[p]\cap L \simeq \{\overrightarrow{(r_ie^(i\theta_i)}\in L |
\overrightarrow{(r_i^2)} = \vec{\lambda} + \beta(p+tu) \in (\RR_{\geq 0})^m \exists t\in\RR \}.
\end{align}
According to the Proposition \ref{Proposition compute CHI is product of zs }, we have $$\chi^u(z)=\prod_{j=0}^{m-1} z_j=r_0...r_{m-1} exp(i(\theta_0+...+\theta_{m-1})).$$
According to equation \eqref{eqt r square condition}, we have
\begin{align}
|\chi^u(L)|
	&=\{\prod_{j=0}^{m-1} r_j | (r_0e^{i\theta_0},...,r_{m-1} e^{i\theta_{m-1}})\in L \} \\
	&=\{ \sqrt{\prod_{j=0}^{m-1} ((\lambda+\beta(p))_j + t )} \in \RR_{\geq 0} | t\in \RR\ and\  \vec{\lambda} + \beta(p+tu) \in (\RR_{\geq 0})^m \}
\end{align}
Notice that there exists some $t_0\in \RR$ such that $$\vec{\lambda} + \beta(p+t_0u) \in \partial (\RR_{\geq 0})^m.$$ In other words, $(\lambda+\beta(p))_j + t_0$ vanishes for some $j=0,...,m-1$. Thus
the product $\prod_{j=1}^r ((\lambda+\beta(p))_j + t_0 ) =0$. On the other hand, the product$\prod_{j=0}^{m-1} ((\lambda+\beta(p))_j + t_0 ) \longrightarrow +\infty$ as $t\longrightarrow +\infty$. We obtain that $$|\chi^u(L)|=\RR_{\geq 0}.$$

And for any $(r_0e^{i\theta_0},...,r_{m-1}e^{i\theta_{m-1}})\in L$, the angle vector satisfies the equation
$$ \sum_{j=0}^{m-1} (l^{(b)}_\perp)_j\cdot \theta_j = \phi^{(b)} $$ for each $b=k+1,...,m$. In other words, we have
$$ \vec{\theta} = \vec{\phi} + \RR<l^{(1)},...,l^{(k)}> $$ for some constant $\vec{\phi} =(\phi_0,...,\phi_{m-1})\in \RR^m $. We then have $$\mathds{1}\cdot \vec{\theta} = \mathds{1}\cdot \vec{\phi}=\phi_0+...+\phi_{m-1}=:\Phi $$ which is a constant since $\mathds{1}\cdot l^{(1)}=...=\mathds{1}\cdot l^{(k)}=0$ for special Lagrangian $L\in X_\Delta$.

Hence, overall, we obtain that the image of $L$ under $\chi^u$ is a half-line in $\CC$, that is, $$\chi^u(L) = |\chi^u(L)| e^{i\Phi} = \RR_{\geq 0} \cdot e^{i\Phi}.$$
Thus the proof is completed.

\begin{comment}
\begin{align}
\begin{array}{ll}
\sum_{j=1}^n l^{(a)}_j\cdot r_j^2 = c^{(a)}, &\forall a=1,...,k \\
\sum_{j=1}^n (l^{(b)}_\perp)_j\cdot \theta_j = \phi^{(b)}, &\forall b=k+1,...,n
\end{array}
\end{align}

\begin{align}
X_\Delta \simeq \left\{ (r_1e^{i\theta_1},...,r_re^{i\theta_r})\in\CC^r|
				\begin{array}{ll}
					\sum_{i=1}^n \iota_{ij}r_i^2 = t_j  \\
                    \sum_{i=1}^n \iota_{ij}\theta_i = 0
				\end{array}
                \4 j=1,...,k
			\right\}
\end{align}
\end{comment}

\end{proof}

Since $\mu_G^{(2)}$ is nothing but $|\chi^u-\epsilon|^2$, we immediately have the following consequence:

\begin{cor}\label{corollary 2nd projection of lagrangian image}
With the notations above, we have
$$ \mu_G^{(2)} ( (\mu_G^{(1)})^{-1}([p]) \cap L ) = \RR_{\geq \xi} $$
for some constant $\xi\in \RR$.
\end{cor}

\begin{proof}
By the definition of $\mu_G^{(2)}:=Proj_2 \circ \mu_G$, we have $\mu_G^{(2)}(x)=|\chi^u(x)-\epsilon|^2$ for any $x\in X_\Delta$. In other words, we have
$$ \mu_G^{(2)} : X_\Delta \overset{\chi^u}\longrightarrow \CC \overset{f}\longrightarrow \RR $$
where $f(z):= |z-\epsilon|^2$, that is $\mu_G^{(2)} = f\circ \chi^u $. Therefore, according to the previous proposition, we have
\begin{align}
\mu_G^{(2)} ( (\mu_G^{(1)})^{-1}([p]) \cap L )
	&= f\circ \chi^u( (\mu_G^{(1)})^{-1}([p]) \cap L ) \\
	&= f (\RR_{\geq 0} \cdot e^{i\Phi}) \\
    &= ( \underset{z}{inf} |\epsilon-z|^2 , \underset{z}{sup} |\epsilon-z|^2 )
   % &= ( inf \{|\epsilon-x|^2\in\RR|x\in \RR_{\geq 0} \cdot e^{i\Phi} \}, sup\{|\epsilon-x|^2\in\RR|x\in \RR_{\geq 0} \cdot e^{i\Phi} \} )
\end{align}
where $z$ runs through $z\in \RR_{\geq 0} \cdot e^{i\Phi} \subset \CC$. Let $\xi=\underset{z}{inf} |\epsilon-z|^2$, the minimum distance square between the point $\epsilon$ and the half line $\RR_{\geq 0} \cdot e^{i\Phi}$. And we see that $\underset{z}{sup} |\epsilon-z|^2 = +\infty $. We then have $$\mu_G^{(2)} ( (\mu_G^{(1)})^{-1}([p]) \cap L ) = [\xi,+\infty) = \RR_{\geq \xi}$$ and thus the proof is completed.
\end{proof}

With the images of the Lagrangian $L$ under the fibrations $\mu_G^{(1)}$ and $\mu_G^{(2)}$ at hand, we are finally ready to compute the image $\mu_G(L)$ which we desired.

\begin{thm}
Let $\mu_G$ be the Gross fibration, $\mu_\Delta$ be the moment map and $L\subset X_\Delta$ be the Lagrangian as defined before. Then the image of the Lagrangian $L$ on the base $B$ under the Gross fibration $\mu_G$ is
$$ \mu_G (L) = \frac{\mu_\Delta(L)}{\RR\cdot u} \times \RR_{\geq \xi}.$$
\end{thm}

\begin{proof}
To proof the equality, we divide the proof into 2 parts, which is the "$\subset$" part and the "$\supset$" part.

To show the "$\subset$" part, we let $x\in L$. By Proposition \ref{proposition 1st projection of lagrangian image}, we have that $\mu_G^{(1)}(x) = [p] \in \frac{M}{\RR\cdot u} $ for some $p\in \mu_\Delta (L)$. Notice that the point $x$ is in the intersection of $(\mu_G^{(1)})^{-1}([p])$ and $L$, we then have $\mu_G^{(2)}(x) = q \in \RR_{\geq \xi} $ according to the above Corollary \ref{corollary 2nd projection of lagrangian image}. Overall, we obtain that
$$ \mu_G(x)%=(\mu_G^{(1)}(x),\mu_G^{(2)}(x))
=([p],q)\in \frac{\mu_\Delta (L)}{\RR\cdot u} \times \RR_{\geq \xi} .$$ Thus we have $$ \mu_G (L) \subset \frac{\mu_\Delta(L)}{\RR\cdot u} \times \RR_{\geq \xi}.$$

To show the "$\subset$" part, we let $([p],q)\in \frac{\mu_\Delta (L)}{\RR\cdot u} \times \RR_{\geq \xi}$. According to Corollary \ref{corollary 1st projection of lagrangian image}, we have
$$(\mu_G^{(1)})^{-1}([p]) \cap L  \neq \phi .$$ And according to Corollary \ref{corollary 2nd projection of lagrangian image}, we obtain that
$$ (\mu_G^{(2)})^{-1}(q)  \cap  (\mu_G^{(1)})^{-1}([p]) \cap L  \neq \phi .$$
Immediately, we obtained that there exist a point $x\in L$ such that $\mu_G^{(1)}(x)=[p]$ and $\mu_G^{(2)}(x)=q$. Thus we have $\mu_G(x) % =(\mu_G^{(1)}(x),\mu_G^{(2)}(x))
=([p],q)$ and $$ \mu_G (L) \supset \frac{\mu_\Delta(L)}{\RR\cdot u} \times \RR_{\geq \xi}.$$

Overall, we have the equality
%$$ \mu_G (L) = \frac{\mu_\Delta(L)}{\RR\cdot u} \times \RR_{\geq \xi}.$$
as desired.

\end{proof}

\begin{rmk}
Notice that the constant $\xi\in\RR$ that appears in the Lagrangian image depends on the constant $\epsilon\in\CC$ that appears in the Gross fibration $\mu_G$. To simplify our computation in later sections, we choose $\epsilon\in\CC$ such that
$$ Argument ( \epsilon ) \in [\Phi-\pi,\Phi+\pi].$$ With this choice of $\epsilon$, we see that  minimum distance between $\epsilon$ and $\RR_{\geq 0} \cdot e^{i\Phi}$ is $|\epsilon|$ and thus $\xi=0$. This means that under such Gross fibration $\mu_G$, the image of Lagrangian becomes
$$ \mu_G (L) = \frac{\mu_\Delta(L)}{\RR\cdot u} \times \RR_{\geq 0}\subset B.$$
\end{rmk}

With the image of Lagrangian computed, we are ready to compute the mirror brane $L^\vee$ under the SYZ program in the next section.

%-------------------------------------------------------------------------------------------------

\section{Semi-flat Mirror Brane}

In this section, we follow the SYZ transformation procedure introduced before to construct the semi-flat B-brane $\mathcal{B}$ as a subset in $\MM_0$. We will first transform the A-brane $(L,\nabla)$ in a fiberwise sense, then glue them together along the Lagrangian image $\mu_G(L)\subset B$ to obtain the semi-flat mirror brane $$ \BB = \{(L_G,\nabla)\in \MM_0 | \prod_{j=0}^{m-1} Z_{\beta_j}(L_G,\nabla)^{-l_j^{(a)}} = e^{c^{(a)}+i\phi^{(a)}}\4 a=1,...,k \} $$ which is the main result of this section.

Before we construct the fiberwise as mentioned, we first give an equivalence between properties of torus connection $\nabla$ and properties of its holonomy homomorphism $Hol_\nabla$ as follows:

\begin{lem}
Let $\TT$ be an $n_\RR$-torus and $S$ be a $k_\RR$ dimensional subtorus in it. Suppose that $\TT$ and $S$ are equipped with flat $U(1)$ connections $\nabla_\TT$ and $\nabla_S$ respectively, then $\nabla_\TT$ twisted by $\nabla_S$ is trivial when restricted to $S$ if and only if $Hol_{\nabla_\TT}$ coincide with $Hol_{\nabla_S}^{-1}$ in $Hom(\pi_1(S),U(1))$.
\end{lem}

\begin{proof}
Since $\nabla_\TT$ and $\nabla_S$ are flat $U(1)$ connections, we suppose that $\nabla_\TT=d+iA_\TT$ and $\nabla_S=d+iA_S$ for some 1-form $A_\TT \in \Omega^1(\TT,\RR)$ and $A_S \in \Omega^1(S,\RR)$ respectively. By definition, the condition that $\nabla_\TT$ twisted by $\nabla_S$ is trivial when restricted to $S$ is defined as $\nabla_\TT|_S + iA_S = d $ on $S$. Equivalently, we have $A_\TT|_S=-A_S$ as a 1-form in $\Omega^1(S,\RR)$.

Let $\gamma$ be a 1-cycle on $S$ represented by the closed path $\gamma(t):[0,1]\longrightarrow S$. We briefly recall the definition of the holonomy operator $Hol_\nabla(\gamma)$ as the evaluation of the flat chapter $s$ on the trivial bundle $U(1)\times S$ at base point $\gamma(1)$ with initial value $s(0)$ fixed to be $1\in U(1)$. Hence we can compute that $Hol_{\nabla_S}(\gamma)=exp(i\int_\gamma A_S)$, and similarly $Hol_{\nabla_\TT}(\gamma)=exp(i\int_\gamma A_\TT)$. This explicit relationships between the holonomy operators $Hol_{\nabla_S},Hol_{\nabla_\TT}$ and the 1 forms $A_S,A_\TT$ give us the equivalence that $A_\TT|_S=-A_S$ if and only if $Hol_{\nabla_\TT}=Hol_{\nabla_S}^{-1}$ since
$$Hol_{\nabla_S}(\gamma)=exp(i\int_\gamma A_S)=exp(-i\int_\gamma A_\TT)=Hol_{\nabla_\TT}(\gamma)^{-1}$$
\end{proof}

As inspired by the above lemma, we are required to compute the fundamental groups of the Lagrangian $L$ and the intersection of a Gross fiber $\mu^{-1}_G(b)$ with $L$. We express the intersection $\mu^{-1}_G(b) \cap L$ in terms of an orbit of the action $TB\curvearrowright X$ as follows:

%\begin{lem}Let $(L,\nabla)$ be an A-brane. Suppose that $x\in L$ and $\mu_G(x)=b\in B$, then$$ \mu^{-1}_G(b) \cap L = \RR<\beta(l^{(1)}),...,\beta(l^{(k)})>\cdot x $$\end{lem}

%\begin{lem}Let $(L,\nabla)$ be an A-brane and $b\in \mu_G(L)$ is in the Lagrangian image. Then either $L\cap \mu_G^{-1}(b)$ has exactly 1 connected component with  $$L\cap \mu_G^{-1}(b)=\RR<l^{(1)},...,l^{(k)}>\cdot x$$ for some $x\in L\cap \mu_G^{-1}(b)$, or $L\cap \mu_G^{-1}(b)$ has exactly 2 connected components with$$L\cap \mu_G^{-1}(b)=\RR<l^{(1)},...,l^{(k)}>\cdot x_1 \sqcup \RR<l^{(1)},...,l^{(k)}>\cdot x_2$$ for some $x_1,x_2\in L\cap \mu_G^{-1}(b)$ in different component.
%Suppose that $x\in L$ and $\mu_G(x)=b\in B$, then$$ \mu^{-1}_G(b) \cap L = \RR<\beta(l^{(1)}),...,\beta(l^{(k)})>\cdot x $$\end{lem}

\begin{lem}
Let $(L,\nabla)$ be an A-brane and $b\in \mu_G(L)$ is in the Lagrangian image. Then one of the followings is true.
\begin{enumerate}
\item$L\cap \mu_G^{-1}(b)$ has exactly 1 connected component with
$$L\cap \mu_G^{-1}(b)=\RR<l^{(1)},...,l^{(k)}>\cdot x_0$$ for some $x\in L\cap \mu_G^{-1}(b)$
\item$L\cap \mu_G^{-1}(b)$ has exactly 2 connected components with
$$L\cap \mu_G^{-1}(b)=\RR<l^{(1)},...,l^{(k)}>\cdot x_1 \sqcup \RR<l^{(1)},...,l^{(k)}>\cdot x_2$$ for some $x_1,x_2\in L\cap \mu_G^{-1}(b)$ in each components.
\end{enumerate}
\end{lem}

\begin{proof}
According to the definition of the Lagrangian $L$ in Definition \ref{definition of L in X}, we have
\begin{align*}
L
\simeq \left\{ (r_0,...,r_{m-1},\theta_0,...,\theta_{m-1})|
\begin{array}{ll}
\sum_{j=0}^{m-1} l^{(a)}_j\cdot r_j^2 = c^{(a)}, &\forall a=1,...,k \\
\sum_{j=0}^{m-1} (l^{(b)}_\perp)_j\cdot \theta_j = \phi^{(b)}, &\forall b=k+1,...,m
\end{array}
\right\}(mod\ G)
\end{align*}
Notice that the condition $\sum_{j=0}^{m-1} (l^{(b)}_\perp)_j\cdot \theta_j = \phi^{(b)}, \forall b=k+1,...,m$ is equivalent to $\vec{\theta}\in\vec{\Theta} + \bigoplus_{a=1}^k\RR\cdot l^{(a)}$ for some  $\vec{\Theta}=(\Theta_0,...,\Theta_{m-1})\in \RR^m$.

  On the other hand, we recall that the Gross fibration is given by $\mu_G(x)=([mu_\delta(x)],|\chi^u(x)-\epsilon|^2)$ for $x\in X$. Suppose that $b=([m],\eta)\in B= (M/\RR u)\times \RR_{\geq 0}$. By expressing $x\in X$ in term of $(r_0,...,r_{m-1},\theta_0,...,\theta_{m-1})$ coordinates, we have the condition $\mu_G^{(1)}(x)=[\mu_\delta(x)]=[m]$ is equivalent to the condition $\vec{r}=\vec{\lambda}+\beta(m+tu)$ for some $t\in \RR$. Also, by the computation in Proposition \ref{Proposition compute CHI is product of zs }, we have $\eta=|\chi^u(x)-\epsilon|^2=|r_0...r_{m-1}-\epsilon e^{-i\phi}|^2$ where $\phi:=\Theta_0+\cdots+\Theta_{m-1}\in \RR$ is a constant. Notice that when we substitute $\vec{r}=\vec{\lambda}+\beta(m+tu)$ into $\eta=|r_0...r_{m-1}-\epsilon e^{-i\phi}|^2$ we obtain a quadratic equation in 1 variable $t\in\RR$. Thus $t$ has at most 2 roots and
\begin{align*}
\mu_G^{-1}(b)
\simeq \left\{ (r_0,...,r_{m-1},\theta_0,...,\theta_{m-1})|
\begin{array}{ll}
\vec{r}=\vec{\lambda}+\beta(m+tu) \text{ where $t\in\RR$ is a root} \\
\vec{\theta}\in \RR^m
\end{array}
\right\}(mod\ G)
\end{align*}

We can check that if $\vec{r}=\vec{\lambda}+\beta(m+tu)$ for some $t\in\RR$ then $
\sum_{j=0}^{m-1} l^{(a)}_j\cdot r_j^2 = c^{(a)}, \forall a=1,...,k $ since $m\in \mu_\delta(L)$ and  $<l^{(a)},\vec{\lambda}+\beta(m+tu)>=<l^{(a)},\vec{\lambda}+\beta(m)>+0t= c^{(a)}$. And hence we are now ready to compute the intersection between $\mu_G^{-1}(b)$ and $L$ as follows:
For the case that $t$ has only 1 root $t_0$, then
\begin{align*}
\mu_G^{-1}(b) \cap L
&\simeq \left\{ %(r_1,...,r_r,\theta_1,...,\theta_r)|
\begin{array}{ll}
\vec{r}=\vec{\lambda}+\beta(m+t_0u) \\
\vec{\theta}\in\vec{\Theta} + \bigoplus_{a=1}^k\RR\cdot l^{(a)}
\end{array}
\right\}(mod\ G)\\&\simeq \RR<l^{(1)},...,l^{(k)}>\cdot x
\end{align*}
for some $x_0\in \mu_G^{-1}(b) \cap L$ where $\mu_G^{-1}(b) \cap L$ has only 1 connected component. For the case that $t$ has exactly 2 roots $t_1,t_2$, then
\begin{align*}
\mu_G^{-1}(b) \cap L
&\simeq \left\{ %(r_1,...,r_r,\theta_1,...,\theta_r)|
\begin{array}{ll}
\vec{r}=\vec{\lambda}+\beta(m+t_iu) \4 i=1,2. \\
\vec{\theta}\in\vec{\Theta} + \bigoplus_{a=1}^k\RR\cdot l^{(a)}
\end{array}
\right\}(mod\ G)\\&\simeq \RR<l^{(1)},...,l^{(k)}>\cdot x_1 \sqcup \RR<l^{(1)},...,l^{(k)}>\cdot x_2.
\end{align*}
for some $x_1,x_2\in \mu_G^{-1}(b) \cap L$ in different connected component of $\mu_G^{-1}(b) \cap L$ where $\mu_G^{-1}(b) \cap L$ has exactly 2 connected components. Thus the proof is completed.

\end{proof}

\begin{rmk}
Notice that the family of isomorphism class of flat $U(1)$ connection on $L$ is one-one correspondence to the group $Hom(\pi_1(L),U(1))$. By the above Lemma and the assumption that $\mu_\Delta(L)$ do not intersect dimension $k$ strata of the polytope $\Delta$, we obtain the Lagrangian $L$ that is diffeomorphic to a $k$-torus. Moreover, the generators of the fundamental group of $L$ are just $\beta(l^{(1)}),...,\beta(l^{(k)})$ according to the above Lemma.
\end{rmk}

With the 2 Lemmas above, we are ready to compute the fiberwise dual of $\mu^{-1}_G(b) \cap L$ in $\MM_0$:

\begin{prop}
Let $(L,\nabla_L)$ be an A-brane where $\nabla_L$ is an isomorphism class of flat $U(1)$ connection defined by $Hol_\nabla(\beta(l^{(a)}))=e^{i\phi^{(a)}}\in U(1)$ for each $a=1,...,k$. Then the fiberwise SYZ dual of $L\cap \mu_G^{(1)}(b)$ is given by
$$ \BB_b = \{(L_G,\nabla)\in \MM_0 |  Hol_\nabla(\beta(l^{(a)}))=e^{-i\phi^{(a)}} \4 a=1,...,k\} .$$
\end{prop}

\begin{proof}
The proof is done by direct computation using the definition of SYZ transform and the Lemmas above.
Let $L_G = \mu_G^{(1)}(b)$ be fixed. According to the SYZ construction, the fiberwise dual of $L\cap L_G$ is given by
$$\BB_b
	:= \{ (L_G,\nabla)\in \MM_0 | \text{$\nabla$ twisted by $\nabla|_L$ is trivial when restricted to $L \cap L_G$}.\} $$
By the result of the above lemma, we have
\begin{align}
\BB_b&=  \{ (L_G,\nabla)\in \MM_0 | Hol_\nabla = Hol_{\nabla_L}^{-1} \ on\ Hom(\pi_1(L\cap L_G),U(1)) \}\\
    &= \{ (L_G,\nabla)\in \MM_0 | Hol_\nabla(\beta(l^{(a)}) = Hol_{\nabla_L}(\beta(l^{(a)})^{-1} \4 a=1,...,k \}
  \end{align}  By the assumption on $Hol_{\nabla_L}(\beta(l^{(a)})=e^{i\phi^{(a)}} \4 a=1,...,k$, we obtain that
$$  \BB_b=  \{ (L_G,\nabla)\in \MM_0 | Hol_\nabla(\beta(l^{(a)}) = e^{-i\phi^{(a)}} \4 a=1,...,k \}
$$
Thus the proof is completed.
\end{proof}

%%Define $\beta_j(L_G)$ and $\beta_j(L_\Delta)$ !!!

On the other hand, we now consider a Gross fiber $L_G=\mu_G^{-1}(b)$ for some $b=([m],\eta)\in B$. For each of $L_G$ we can define a moment map fiber $L_\delta$ that is Lagrangian isotopy of $L_G$. We now outline the construction. Recall that the Gross fiber $L_G$ is defined by the 2 equations $\mu_\delta(x)=m+cu$ for some $c\in \RR$ and $|\chi^u(x)-\epsilon|^2=\eta$. We now define the Lagrangian isotopy by the following continuous family of Lagrangians parameterized by $t\in [0,1]$:
\begin{align}
L_t := \left\{x\in X_\Delta :
	\begin{array}{lll}
    	\mu_\Delta(x) \in m+\RR\cdot u \\
        |\chi^u(x)-t\epsilon|^2 =\eta
	\end{array}
\right\}.
\end{align}
We observe that $L_1$ is nothing but the Gross fiber $L_G$, and $L_0$ is a moment map fiber which we denote it by $L_\Delta:=L_0$. Moreover, we can compute their relative $H_2$ groups. That is, $H_2(X_\Delta,L_G;\ZZ)=<\beta_0(L_G),...,\beta_{m-1}(L_G)>$ where the $\beta_j(L_G)$ are defined \eqref{define beta on gross fibers} and $H_2(X_\Delta,L_\Delta;\ZZ)=<\beta_0(L_\Delta),...,\beta_{m-1}(L_\Delta)>$ where $\beta_j(L_\Delta)$ are just the basic discs under the identification
$$\begin{CD}
	0	@>>>	G			@>\iota >>	\ZZ^m		@>\beta >>	N			@>>>	0	\\
	@.			@VV\simeq V			@VV\simeq V			@VV\simeq V				\\
   	0	@>>>	H_2(X_\Delta)@>>>	H_2(X_\Delta,L_\Delta)@>>>	H_1(,L_\Delta)	@>>>	0	
\end{CD}$$

\begin{lem}
With the notations above, we let $L_G$ be a Gross fiber with its corresponding moment map fiber $L_\Delta$. Then
$$ \int_{\beta_j(L_G)-\beta_0(L_G)}\omega_X = \int_{\beta_j(L_\Delta)-\beta_0(L_\Delta)}\omega_X $$
for each $j=1,...,m-1$.
\end{lem}

\begin{proof}
According to the definitions of $\beta_\cdot(L_G)$ in \eqref{define beta on gross fibers}, we have to fix a $b_0\in B_+$ first. Suppose that $L_G=\mu_G^{-1}(b)$ for some $b=([m],\eta)\in B$. We choose a smooth path $b(s):[0,1]\longrightarrow B$ such that $b(0)=b_0$ and $b(1)=b$. According to the above construction, for each Gross fiber $\mu_G^{-1}(b(s))$ we have the continuous family of Lagrangian
\begin{align}
L_{t,s} := \left\{x\in X_\Delta :
	\begin{array}{lll}
    	\mu_\Delta(x) \in \pi_1(b(s))+\RR\cdot u \\
        |\chi^u(x)-t\epsilon|^2 =\pi_1(b(s))
	\end{array}
\right\}.
\end{align}
for any $t,s\in[0,1]$. In particularly, we have $L_{1,1}=L_G$ for $t=1,s=1$ and $L_{0,1}=L_\Delta$ for $t=0,s=1$. We observe that $\RR<v_j-v_0>$ acts on each Lagrangian $L_{t,s}$ for each $j=1,...,m-1$ where the action is given by
$$\phi(v_j-v_0)\cdot (z_0,...,z_{m-1}) := (e^{-i\phi}z_0,z_1,...,z_{j-1},e^{i\phi}z_j,z_{j+1},...,z_{m-1}) .$$ The action is well-defined since we can check that $\chi^u(\vec{z})=\chi^u(\phi(v_j-v_0)\cdot\vec{z})$. We now fix a point in $L_G=L_{1,1}$ and move the point along the continuous family $L_{t,s}$, then we obtain the point $z(t,s)\in L_{t,s}$ for each $t,s\in[0,1]$. We define the map
\begin{align}
H_j:[0,1]\times[0,1]\times \SSS &\longrightarrow X_\Delta  \\
by \hspace{90pt} (t,s,[\phi])&\longmapsto \phi(v_j-v_0)\cdot z(t,s)
\end{align}
Hence we can now compute that
\begin{align*}
\beta_j(L_G)-\beta_0(L_G)
	&= \beta_j(L_{0,0})-\beta_0(L_{0,0}) + H_j([0,1]\times \{0\}\times\SSS) + H_j(\{1\}\times [0,1]\times\SSS)
    \\
\beta_j(L_\Delta)-\beta_0(L_\Delta)
	&= \beta_j(L_{0,0})-\beta_0(L_{0,0}) + H_j(\{0\}\times [0,1]\times\SSS)
\end{align*}
as an equation in $H_2(X_\Delta,L_G)$ and $H_2(X_\Delta,L_\Delta)$ respectively for any $j=1,...,m-1$.

With the 2 equations, we consider
\begin{align}
& \int_{\beta_j(L_G)-\beta_0(L_G)}\omega - \int_{\beta_j(L_\Delta)-\beta_0(L_\Delta)}\omega \\
	%&= \int_{ H_j([0,1]\times \{0\}\times\SSS) + H_j(\{1\}\times [0,1]\times\SSS)-H_j(\{0\}\times [0,1]\times\SSS)}\omega \\
    &= \int_{ H_j([0,1]\times \{0\}\times\SSS)}\omega +  \int_{H_j(\{1\}\times [0,1]\times\SSS)}\omega-\int_{H_j(\{0\}\times [0,1]\times\SSS)}\omega \\
    &= \int_{ \partial H_j([0,1]\times [0,1] \times\SSS)}\omega + \int_{ H_j([0,1]\times \{1\}\times\SSS)}\omega \\
    &= \int_{ H_j([0,1]\times [0,1] \times\SSS)} d\omega
    	+ \int_{t=0}^1 \int_{\phi=0}^{2\pi} (H_j^{s=1})^*\omega\\
    &= \int_{t=0}^1 \int_{\phi=0}^{2\pi} \omega((H_{j}^{s=1})_*(\frac{\partial}{\partial t}),(H_{j}^{s=1})_*(\frac{\partial}{\partial \phi }) ) dt \wedge d\phi
\end{align}
We now remain to compute the term $\omega((H_{j}^{s=1})_*(\frac{\partial}{\partial t}),(H_{j}^{s=1})_*(\frac{\partial}{\partial \phi }) )$.

By definition of $H_j$, we have $H_j(t,1,[\phi])=\phi(v_j-v_0)\cdot \vec{z}(t,1)$ where
\begin{align}
Arg(\vec{z}(t,1)) &= ((\theta_0-\phi),\theta_1,...,\theta_{j-1},(\theta_j-\phi),\theta_{j+1},...,\theta_{m-1})	\\
\vec{z}(t,1)|^2 &= \vec{\lambda} + \beta^\vee(\pi_1(b)) + f(t)(1,...,1)
\end{align}
for some smooth functions $f(t):[0,1]\longrightarrow\RR$. Thus, by chain rule, we have
\begin{align}
(H_{j}^{s=1})_*(\frac{\partial}{\partial t})
	&= \sum_{j=0}^{m-1} \frac{ \partial r_j^2}{\partial t} \frac{\partial}{\partial r_j^2} +
		\sum_{j=0}^{m-1} \frac{ \partial \theta_j}{\partial t} \frac{\partial}{\partial \theta_j} =  f'(t)\sum_{j=0}^{m-1}  \frac{\partial}{\partial r_j^2} \\
(H_{j}^{s=1})_*(\frac{\partial}{\partial \phi})
	&= \sum_{j=0}^{m-1} \frac{ \partial r_j^2}{\partial \phi} \frac{\partial}{\partial r_j^2} +
		\sum_{j=0}^{m-1} \frac{ \partial \theta_j}{\partial \phi} \frac{\partial}{\partial \theta_j} =  \frac{\partial}{\partial \theta_j}-\frac{\partial}{\partial \theta_0}
\end{align}
And hence we can compute that
%\begin{align}	&\omega((H_{j}^{s=1})_*(\frac{\partial}{\partial t}),(H_{j}^{s=1})_*(\frac{\partial}{\partial \phi }) ) \\    &= \sum_{l=0}^{r-1} (dr_l^2 \wedge d\theta_l ) (f'(t)\sum_{j=1}^{r-1}  \frac{\partial}{\partial r_j^2},(\frac{\partial}{\partial \theta_j}-\frac{\partial}{\partial \theta_0}))\\    &= f'(t) \sum_{l,j=0}^{r-1}\delta_{lj}(\delta_{lj}-\delta_{l0}) \\    &=0	\end{align}
$$\omega((H_{j}^{s=1})_*(\frac{\partial}{\partial t}),(H_{j}^{s=1})_*(\frac{\partial}{\partial \phi }) )
    %&= \sum_{l=0}^{r-1} (dr_l^2 \wedge d\theta_l ) (f'(t)\sum_{j=1}^{r-1}  \frac{\partial}{\partial r_j^2},(\frac{\partial}{\partial \theta_j}-\frac{\partial}{\partial \theta_0}))\\
    = f'(t) \sum_{l,j=0}^{m-1}\delta_{lj}(\delta_{lj}-\delta_{l0})
    =0 $$
Overall, combining above computations, we obtain our desired result that
$$\int_{\beta_j(L_G)-\beta_0(L_G)}\omega - \int_{\beta_j(L_\Delta)-\beta_0(L_\Delta)}\omega=0$$
and thus completing the proof.
\end{proof}

In order to glue the fiberwise dual $\BB_b$ together along $\mu_G(L)$, we first give an equivalence statement of $L_G=\mu_G^{-1}(b)$ being a Gross fiber at the base point $b$ lying in $\mu_G(L)$, that is:

\begin{lem}
Let $L_G$ be a Gross fiber $\mu_G^{-1}(b)$ for some $b=([p],q)\in B$. Then $p\in \mu_\Delta(L)$ if and only if $$ \sum_{j=0}^{m-1} l^{(a)}_j \int_{\beta_j(L_G)}\omega = c^{(a)} \4 a=1,...,k$$
\end{lem}

\begin{proof}
Let $(L,\nabla)$ be the A-brane and $L_G$ be a Gross fiber $\mu_G^{-1}(b)$ for some $b=([p],q)\in B$. Using the above construction of the corresponding moment map fiber $L_\Delta=\mu_\Delta(\nu)$, we have
$ \mu_G^{(1)}(L_G) = [p]=[\nu]=[\mu_\Delta(L_\Delta)] .$ This gives us the equivalence that $\mu_G^{(1)}(L_G)\in \mu_\Delta(L)(mod\ \RR u)$ if and only if $\mu_\Delta(L_\Delta) \in  \mu_\Delta(L)$. And by definition and construction of $L$, we have
\begin{align}
\vec{\lambda}+\beta^\vee(\mu_\Delta(L_\Delta))=\{y\in\RR^m | <l^{(a)},y>=c^{(a)} \4 a=1,...,k \}
\end{align}
And by the theorem of Cho-Oh\cite{Cho Oh}, we have $\int_{\beta_j(L_\Delta)} \omega = \vec{\lambda}+\beta^\vee(\nu) $ for $j=0,...,m-1$. Hence we obtain the equivalence that $\mu_\Delta(L_\Delta) \in  \mu_\Delta(L)$ if and only if $$\sum_{j=0}^{m-1}l^{(a)}_j\int_{\beta_j(L_\Delta)} \omega  =c^{(a)} \4 a=1,...,k.$$
Using the above lemma and the condition that $l^{(a)}_0+\cdots l^{(a)}_{m-1}=0$, we compute that
\begin{align}
c^{(a)}
	&= \sum_{j=0}^{m-1}l^{(a)}_j\int_{\beta_j(L_\Delta)} \omega
  %  &= -(l^{(a)}_1+\cdots+l^{(a)}_{r-1})( \int_{\beta_0(L_\Delta)} \omega )  +\sum_{j=1}^{r-1}l^{(a)}_j\int_{\beta_j(L_\Delta)} \omega \\
    = \sum_{j=1}^{m-1}l^{(a)}_j\int_{\beta_j(L_\Delta)-\beta_0(L_\Delta)} \omega
   % &= \sum_{j=1}^{r-1}l^{(a)}_j\int_{\beta_j(L_G)-\beta_0(L_G)} \omega \\
    = \sum_{j=0}^{m-1}l^{(a)}_j\int_{\beta_j(L_G)} \omega
\end{align}
for each $a=1,...,k$. And thus the equivalence is verified.
\end{proof}

According to the SYZ transformation, the mirror brane is given by
$$ \BB := \bigcup_{b\in I} \BB_b$$
where $I$ denote the image of Lagrangian $L$ under the Gross fibration $\mu_G$, that is $I:= \mu_G(L)$. Finally, we compute the semi-flat mirror brane $\BB$ as follows:

\begin{thm} \label{def thm of BB0 semiflat B brane}
Let $(L,\nabla_L)$ be an A-brane. Under the SYZ transformation, we obtain the semi-flat mirror brane $\BB_0$ as
$$ \BB_0 = \{(L_G,\nabla)\in \MM_0 | \prod_{j=0}^{m-1} Z_{\beta_j}(L_G,\nabla)^{-l_j^{(a)}} = e^{c^{(a)}+i\phi^{(a)}}\4 a=1,...,k \} $$
\end{thm}

\begin{proof}
%By SYZ construction of the mirror brane $\BB_0$,\begin{align}\BB_0:=\left\{(L_G,\nabla)\in \MM_0 :	\begin{array}{l}	\mu^{(1)}_G(L_G) \in \mu_\Delta(L) (mod\ \RR u)\\ Hol_\nabla(\beta(l^{(a)}))=e^{-i\phi^{(a)}}	\end{array    \4 a=1,...,k \right\}\end{align}
Before we compute the mirror brane $\BB_0$, we first consider the condition $\mu^{(1)}_G(L_G) \in \mu_\Delta(L) (mod\ \RR u)$. By the previous lemma, we have that $\mu^{(1)}_G(L_G) \in \mu_\Delta(L) (mod\ \RR u)$ if and only if %$ \sum_{j=0}^{r-1} l^{(a)}_j \int_{\beta_j(L_G)}\omega = c^{(a)} \4 a=1,...,k$
\begin{align}
e^{-c^{(a)}}
	= exp (-\sum_{j=0}^{m-1} l^{(a)}_j \int_{\beta_j(L_G)}\omega)
    = \prod_{j=0}^{m-1} exp ( -\int_{\beta_j(L_G)}\omega)^{l^{(a)}_j }
\end{align}
Moreover, we observe that the 1-cycle (orbit) $\beta(l^{(a)})$ is just $\sum_{j=0}^{m-1}l^{(a)}_j\partial\beta_j(L_G)$ in $H_1(L_G;\ZZ)$. We then compute that $$Hol_\nabla(\beta(l^{(a)}))= Hol_\nabla(\sum_{j=0}^{m-1}l^{(a)}_j\partial\beta_j(L_G))
=\prod_{j=0}^{m-1}Hol_\nabla(\partial\beta_j(L_G))^{l^{(a)}_j}.$$
We are now ready to compute $\BB_0$ according to the SYZ transformation as follows:
\begin{align}
\BB_0
&:=\left\{(L_G,\nabla)\in \MM_0 :
	\begin{array}{l}
		\mu^{(1)}_G(L_G) \in \mu_\Delta(L) (mod\ \RR u)\\
        Hol_\nabla(\beta(l^{(a)}))=e^{-i\phi^{(a)}}
	\end{array}
    \4 a=1,...,k \right\} \\
&=	\left\{(L_G,\nabla)\in \MM_0 :
	\begin{array}{l}
		\prod_{j=0}^{m-1} exp ( -\int_{\beta_j(L_G)}\omega)^{l^{(a)}_j } =e^{-c^{(a)}} \\
        \prod_{j=0}^{m-1}Hol_\nabla(\partial\beta_j(L_G))^{l^{(a)}_j} =e^{-i\phi^{(a)}}
	\end{array}
    \4 a=1,...,k \right\}\\
%&=	\left\{(L_G,\nabla)\in \MM_0 :
%	\begin{array}{l}
%		\prod_{j=0}^{r-1} (exp ( -\int_{\beta_j(L_G)}\omega)Hol_\nabla(\partial\beta_j(L_G)))^{l^{(a)}_j } =e^{-c^{(a)}+i\phi^{(a)}} \\
%        for \  a=1,...,k .
%	\end{array} \right\}\\
&=	\left\{(L_G,\nabla)\in \MM_0 : \prod_{j=0}^{m-1} Z_{\beta_j}(L_G,\nabla)^{-l_j^{(a)}} = e^{c^{(a)}+i\phi^{(a)}}\4 a=1,...,k \right\}\label{eqt Z to q is cnst}
\end{align}
which completes the proof.

%We first recall that $\beta:\ZZ^r\longrightarrow N:e_j\longmapsto v_j$ which is geometrically equivalent to $\partial:H_2(X,L;\ZZ) \longrightarrow H_1(L;\ZZ):\beta_j\longmapsto\partial\beta_j$. We observe that $\beta(l^{(a)})=\beta(\sum_{j=0}^{r-1}l^{(a)}_je_j)=\sum_{j=0}^{r-1}l^{(a)}_jv_j\in N$  geometrically corresponds to  $\sum_{j=0}^{r-1}l^{(a)}_j\partial\beta_j(L_G)$ as a relative 2-cycle in $H_1(L_G;\ZZ)$.
\end{proof}

We recall that the mirror moduli $\MM_0$ has a geometric realization as a complex variety $X^\vee$, that is $\varphi : \MM_+\sqcup \MM_- \longrightarrow X^\vee $ as in the previous chapter. For an A-brane $(L,\nabla)$, we now have the semi-flat mirror brane $\BB_0$ at hand. It is natural to look for a subvariety $L^\vee$ in $X^\vee$ that represent the geometric realization of $\BB_0$ via the map $\varphi$. This will be investigated in the next section.
%-------------------------------------------------------------------------------------------------

\section{Mirror Brane Subvariety}
Let $(L,\nabla)$ be an A-brane as defined in Definition \ref{definition of L in X} which depends on the physical charges $l^{(1)},...,l^{(k)}\in \ZZ^m $. Based on physical arguments, physicist predicts that its mirror $L^\vee\subset X^\vee$ is given by the equations
\begin{align} \label{eqt z to power -q equals cnst}
\prod_{j=1}^{m-1} z_j^{-l_j^{(a)}} = e^{c^{(a)}+i\phi^{(a)}}
%\ \  where \\ z_m&=z_m(z_1,...,z_{n-1})=q_m\prod_{j=1}^{n-1} z_j^{<v_j^*,v_m>}.
\end{align}
where $z_i=z_i(z_1,...,z_{n-1})=Q_{i-n+1}\prod_{j=1}^{n-1} z_j^{<v_j^*,v_i>}$
for $a=1,...,k$ and $i=n,...,m-1$. With the above preparation, we are now ready to verify this prediction via SYZ mirror symmetry.

We first give a lemma which relates the equations given in \eqref{eqt z to power -q equals cnst} and \eqref{eqt Z to q is cnst} via the map $\varphi$ as in Theorem \ref{def thm of BB0 semiflat B brane}:

\begin{lem}
Let $\varphi:\MM_0 \longrightarrow X^\vee$ be given by $\varphi(L,\nabla) = (u,v,z_1,...,z_{n-1})$ as defined in Theorem \ref{varphi map moduli to X check}. Then $$\varphi^*(\prod_{j=1}^{m-1} z_j^{-l_j^{(a)}})=\prod_{j=0}^{m-1} Z_{\beta_j}(L,\nabla)^{-l_j^{(a)}}$$ where $z_i=z_i(z_1,...,z_{n-1})=Q_{i-n+1}\prod_{j=1}^{n-1} z_j^{<v_j^*,v_i>}$ for $i=n,...,m-1$.
\end{lem}

\begin{proof}
According to Theorem \ref{varphi map moduli to X check} we have $\varphi^*(z_j)=\hat{z}_j(L,\nabla)=Z_{\beta_j}(L,\nabla)/Z_{\beta_0}(L,\nabla)$.%=\frac{Z_{\beta_j}(L,\nabla)}{Z_{\beta_0}(L,\nabla)}$
And for $i=n,...,r-1$, we also have $$\varphi^*(z_i)=Q_{i-n+1}\prod_{j=1}^{n-1} \varphi^*(z_j)^{<v_j^*,v_i>}.$$ By \eqref{compute zr-1 in terms of z1 to zn-1}, we can express $Z_{\beta_n},..., Z_{\beta_{m-1}}$ in terms of the K$\ddot{a}$hler parameters $Q_1,...,Q_{m-n}$ and $Z_{\beta_0},..., Z_{\beta_{n-1}}$. Thus we have $$\varphi^*(z_i)=Z_{\beta_i}(L,\nabla)/Z_{\beta_0}(L,\nabla)$$ for $i=n,...,m-1$.

We now compute that
\begin{align*}
	\varphi^*(\prod_{j=1}^{m-1} z_j^{-l_j^{(a)}})
    	&= \prod_{j=1}^{m-1} \varphi^*(z_j)^{-l_j^{(a)}} \\
        &= \prod_{j=1}^{m-1} (\frac{Z_{\beta_j}(L,\nabla)}{Z_{\beta_0}(L,\nabla)})^{-l_j^{(a)}}) \\
        &= Z_{\beta_0}(L,\nabla)^{-(-l^{(a)}_1-\cdots-l^{(a)}_{m-1})}\prod_{j=1}^{m-1} Z_{\beta_j}(L,\nabla)^{-l_j^{(a)}}) \\
\end{align*}
By assumption that $L$ is a special Lagrangian. Equivalently, the charges that defines $L$ satisfies $l^{(a)}_0+\cdots+l^{(a)}_{m-1}=0$ for $a=1,...,k$. Immediately we obtained that
$$ \varphi^*(\prod_{j=1}^{m-1} z_j^{-l_j^{(a)}}) = \prod_{j=0}^{m-1} Z_{\beta_j}(L,\nabla)^{-l_j^{(a)}} $$ for $a=1,...,k$. Thus the proof is completed.
\end{proof}

With this lemma, we are ready to compute the image of $\BB_0$ via the map $\varphi:\MM_0 \longrightarrow X^\vee$ which represents the mirror B-brane in the mirror manifold (variety) $X^\vee$.

\begin{thm} \label{main thm B}
Let $\varphi:\MM_0 \longrightarrow X^\vee$ as defined in Theorem \ref{varphi map moduli to X check} and $\BB_0\subset \MM_0$ be the semi-flat mirror brane as in Theorem \ref{def thm of BB0 semiflat B brane}. Then
$$ \varphi(\BB_0) = \{(u,v,z_1,...,z_{n-1})\in \varphi(\MM_0) :  \prod_{j=1}^{m-1} z_j^{-l_j^{(a)}} = e^{c^{(a)}+i\phi^{(a)}} \4 a=1,...,k \}$$ where $z_i=z_i(z_1,...,z_{n-1})=Q_{i-n+1}\prod_{j=1}^{n-1} z_j^{<v_j^*,v_i>}$ for $i=n,...,m-1$.
\end{thm}

\begin{proof}
According to Theorem \ref{def thm of BB0 semiflat B brane}, we have $\BB_0=\{(L,\nabla)\in \MM_0 : \prod_{j=0}^{m-1}Z_{\beta_j}(L,\nabla)^{l_j^{(a)}}=exp(c^{(a)}+i\phi^{(a)}) \4 a=1,...,k \}$.

We first show the "$\subset$" part. Let $(L,\nabla) \in \BB_0$, we have $(u,v,z_1,...,z_{n-1})=\varphi(L,\nabla)$, that is $z_j=Z_{\beta_j}(L,\nabla)/Z_{\beta_0}(L,\nabla)$ for $j=1,...,n-1$. We consider, for $a=1,...,k$,
\begin{align}
	\prod_{j=1}^{m-1} z_j^{-l_j^{(a)}}
    	&= \prod_{j=1}^{m-1} (\frac{Z_{\beta_j}(L,\nabla)}{Z_{\beta_0}(L,\nabla)})^{l_j^{(a)}} \\
        &= \prod_{j=0}^{m-1} Z_{\beta_j}(L,\nabla)^{l_j^{(a)}} \\
        &= exp(c^{(a)}+i\phi^{(a)}).
\end{align}
Thus we have
$$ \varphi(\BB_0) \subset \{(u,v,z_1,...,z_{n-1})\in \varphi(\MM_0) :  \prod_{j=1}^{m-1} z_j^{-l_j^{(a)}} = e^{c^{(a)}+i\phi^{(a)}} \4 a=1,...,k \}$$

To show the "$\supset$" part, we let $(u,v,z_1,...,z_{n-1})= \varphi(L,\nabla)\in \varphi(\MM_0)$ for some $(L,\nabla)\in \MM_0$ such that $\prod_{j=1}^{m-1} z_j^{-l_j^{(a)}} = e^{c^{(a)}+i\phi^{(a)}}$ for each $a=1,...,k$. We consider
\begin{align}
\prod_{j=0}^{m-1} Z_{\beta_j}(L,\nabla)^{l_j^{(a)}}
	&= \prod_{j=1}^{m-1} (\frac{Z_{\beta_j}(L,\nabla)}{Z_{\beta_0}(L,\nabla)})^{l_j^{(a)}} \\
    &= \prod_{j=1}^{m-1} z_j^{-l_j^{(a)}} \\
    &= exp(c^{(a)}+i\phi^{(a)}).
\end{align}
for $a=1,...,k$. Thus we have $(L,\nabla)\in \BB_0$ and
$$ \varphi(\BB_0) \supset \{(u,v,z_1,...,z_{n-1})\in \varphi(\MM_0) :  \prod_{j=1}^{m-1} z_j^{-l_j^{(a)}} = e^{c^{(a)}+i\phi^{(a)}} \4 a=1,...,k \}.$$

Overall, we have completed the proof.

\end{proof}

We have now completed the entire SYZ transformation procedure for such A-branes as defined in Definition \ref{definition of L in X}. The above theorem shows that the mirror brane is a subvariety in $X^\vee$ which is given by the equations in \eqref{eqt z to power -q equals cnst} as we desire. Given any A-brane $(L,\nabla)$ that is defined as in Definition \ref{definition of L in X}, we are now ready to formulate a definition for its corresponding (naive) SYZ mirror that is exactly defined by the defining equation of $\varphi(\BB_0)\subset\varphi(\MM_0)$ as in Theorem \ref{main thm B}.

\begin{defn} \label{main def B}
Let $(L,\nabla)$ be an A-brane as defined in Definition \ref{definition of L in X} which depends on the physical charges $l^{(1)},...,l^{(k)}\in \ZZ^m $. Then we define
$$L^\vee_{naive} \trieq  \{(u,v,z_1,...,z_{n-1})\in X^\vee :  \prod_{j=1}^{m-1} z_j^{-l_j^{(a)}} = e^{c^{(a)}+i\phi^{(a)}} \4 a=1,...,k \} $$ where $z_i=z_i(z_1,...,z_{n-1})=Q_{i-n+1}\prod_{j=1}^{n-1} z_j^{<v_j^*,v_i>}$ for $i=n,...,m-1$. Moreover, we denote it as the SYZ mirror brane  $L^\vee_{naive}$ of the A-brane $(L,\nabla)$.
\end{defn}

\begin{rmk}
The definition of $L^\vee_{naive}$ is motivated by the computation of $\varphi(\BB_0)$ as in Theorem \ref{main thm B} which is the most natural definition for the SYZ transform. In the sense that $L^\vee_{naive}$ and $\varphi(\BB_0)$, obviously, share the same defining equations in $X^\vee$ and, more importantly, $L^\vee_{naive}$ is the minimal subvaiety in $X^\vee$ containing $\varphi(\BB_0)$ as an open subset in the classical topology. In other words, $L^\vee_{naive}$ can be viewed as the completion of $\varphi(\BB_0)$ in the B-model variety $X^\vee$.
\end{rmk}

\begin{rmk}
Notice that the mirror brane that is constructed via the SYZ transformation is denoted by $L^\vee_{naive}$ instead of $L^\vee_{SYZ}$. We observe that the SYZ mirror brane we have constructed do not necessarily give the correct mirror which we will see when we study A-branes of Aganagic-Vafa type in the next chapter. For this reason, we will call such B-brane the naive SYZ mirror brane, and therefore we can leave this name, the SYZ mirror brane, for the B-brane after we have correctly modified our SYZ transformation. We will denote such B-brane subvariety by $L^\vee_{naive}$ throughout the rest of the thesis. The main reason for the failure of the SYZ transformation is due to the fact that discs counting data of the A-brane itself has not yet been captured in the construction process which will be investigated in more detail in later chapters.
\end{rmk}

We have completed the naive SYZ transformation of any A-brane given by physical charges as considered in \cite{Aganagic Vafa,Lerche Mayr Warner,Fang Liu}. As desired, we has explicitly constructed the naive SYZ mirror brane mathematically and rigorously that explains the principle behind physicists' predictions as in \cite{Aganagic Vafa}.

Next, we are going to study an even more specific type of A-branes than what we have studied in this chapter, namely, the Aganagic-Vafa A-branes.

%%%%%%%%%%%%%%%%%%%%%%%%%%%%%%%%%%%%%%%%%%%%%%%%%%%%%%%%%%%%%%%%%%%%%%%%%%%%%%%%%%%%%%%%%%%%%%%%%%

%\chapter{Enumerative meaning of Aganagic-Vafa mirror branes}
\chapter{SYZ transform for Aganagic-Vafa A-branes}

In this chapter, we narrow our focus to the investigation a special type of A-branes for toric Calabi-Yau A-model of complex dimension three which was first studied by Aganagic and Vafa. These Aganagic-Vafa A-branes has a special B-brane as a subvariety of $X^\vee=\{uv=W(x,y)\}$. The Aganagic-Vafa mirror B-brane is the subvariety defined by the equations $uv=0,\ x=x_0$ and $y=y_0$ where the pair $(x_0,y_0)$ is a root of the polynomial $W(x,y)$.
\par
According to recent study of these A-branes in Fang-Liu\cite{Fang Liu}, the coordinate $y_0$ of the B-brane encodes open Gromov-Witten invariants of the Aganagic-Vafa A-brane. And according to another recent development in SYZ, Chan-Cho-Lau-Tseng\cite{Cho Chan Lau Tseng} gave enumerative meaning of the inverse mirror map. As a result, we can relate the coordinate $x_0$ of the B-brane to the open Gromov-Witten invariants of the Lagrangian fibers. %With the recent results at hand, we are able to reveal the enumerative meaning of the Aganagic-Vafa B-brane. More specifically, we are going to represent the B-brane in terms of open Gromov-Witten invariant and open/close K$\ddot{a}$hler parameter alone.
%Throughout the chapter, we let $L$ be an Aganagic-Vafa brane in a toric Calabi-Yau threefold $X_\Delta$ as defined in \eqref{def AV brane}. According to Proposition \eqref{prop AV brane vs AV line}, we can assume that $L$ is defined by the 2 charges $l^{(1)}=(-1,1,0,...,0)$ and $l^{(1)}=(-1,0,1,0,...,0)$ in $\ZZ^r$ without any loss of generality.
%In the previous two chapters, we have computed the naive SYZ mirror $L_{naive}^\vee$ and predicted mirror $L^\vee$ of $L$ respectively.
With the explicit defining equations of $L^\vee$ and $L^\vee_{naive}$ at hand, we would naturally question whether the two subvarieties coincide or how far do they differ from each other if they do not.
%Naturally, we would like to know whether $L_{naive}^\vee$ and $L^\vee$ coincide.
We will see that the naive SYZ transform $L_{naive}^\vee$ does not coincide with the predicted mirror brane $L^\vee$.
The main reason lies in the fact that the naive SYZ transform $L_{naive}^\vee$ fails to capture the information of the counting discs of the Aganagic-Vafa A-brane $L$. However, we observe that $L_{naive}^\vee$ and $L^\vee$ are similar in some sense which we will discuss in detail later. More precisely, we observe that $L^\vee$ can be viewed as a "further quantum corrected" version of $L_{naive}^\vee$. This indicates that the SYZ transformation described in \cite{Leung Yau Zaslow,Chan} is not perfect yet in the sense that it cannot give the correct mirror in the special case of Aganagic-Vafa branes.
%With the explicit expression of the brane as in \eqref{main theorem Lv=z1,z2}, we have some clues to modify the naive SYZ transformation.
By comparing the defining equations of $L^\vee$ and $L^\vee_{naive}$, we are suggested to modify the naive SYZ transform using the disc counting data of the Aganagic-Vafa A-brane.
In the last part of this chapter, we will give a natural definition for the modified SYZ mirror construction that involves further quantum correction which, by construction, transforms every Aganagic-Vafa A-branes to its desired mirror partner.

%-------------------------------------------------------------------------------------------------
\section{Aganagic-Vafa A-branes}
Aganagic-Vafa A-brane is a (degenerate) type of Lagrangian in a toric Calabi-Yau three-fold which is determined by a line which meets the dimension 1 edges of the polytope
\begin{equation}
\Delta=\{\nu\in M| <\nu,v_i> \geq -\lambda_i, \forall i=0,...,m-1\}.
\end{equation}
where $M\simeq \RR^3$.
\begin{defn} \label{def AV line}
$\mathcal{L}\in \Delta$ is called an Aganagic-Vafa line if
\begin{equation}
\lambda+\beta^\vee(\LL)=\{y\in \mathds{R}^\vee | <y,l^{(1)}>=c,<y,l^{(2)}>=0\}\cap \Delta,
\end{equation}
for some
$$l^{(1)}=-1e_{i_0}+1e_{i_1}+0e_{i_2}\in \mathds{R}^m$$
$$l^{(2)}=-1e_{i_0}+0e_{i_1}+1e_{i_2}\in \mathds{R}^m$$
where $i_0,i_1,i_2=0,...,m-1$ and $c \neq 0$.
\end{defn}

Using the same notation and definition of Lagrangian $L$ as in Definition \ref{definition of L in X}, we now focus on a smaller class of such L which we denote as Aganagic-Vafa A-brane as follows:

\begin{defn} \label{def AV brane}
$L\subset X_\Delta$ is called an Aganagic-Vafa A-brane if
$L$ is a special Lagrangian and $ \mu_\Delta(L)\in \Delta$ is a line that  intersects a unique dimension 1 edge of the polytope $\Delta(1)$. % but not the dimension 0 vertexes of the polytope $\Delta(0)$.
\end{defn}

With the definition of an Aganagic-Vafa line and Aganagic-Vafa brane, we formulate a relationship between the 2 definitions via the moment map $\mu_\Delta$.

\begin{prop} \label{prop AV brane vs AV line}
Let $L\subset X_\Delta$ be a Lagrangian as defined in Definition \ref{definition of L in X}.
Then its moment map image  $\mu_\Delta(L)$ is an Aganagic-Vafa line if $L$ is an Aganagic-Vafa A-brane.
%$L\subset X_\Delta$ is an Aganagic-Vafa A-brane if and only if the moment map image of $L$ is an Aganagic-Vafa line.
\end{prop}

\begin{proof}
Let $L$ be an Aganagic-Vafa brane. By definition, we have $L$ is special and $\mu_\Delta\cup F_{i_0,i_1}=\{m_0\}$ for some unique $m_0\in M$ and $F_{i_0,i_1}$ is a 1 dimensional ( codimension 2 ) edge of the polytope as defined in Definition \ref{definition of F I polytope faces} for some unique $i_0,i_1=0,...,r-1$. Then $\mu_\Delta(L)$ can be represented by $m_0+\RR_{\geq o}v$ for some unique vector $v\in T_\cdot M \simeq M$. Notice that $L$ is special that forces $v=u$. In other words, we have
$$ \mu_\Delta = m_0 + \RR_{\geq 0}u \subset M .$$

Our goal is to find 2 charges $l^{(1)},l^{(2)}\in \ZZ^m$ that defines our $L$ which satisfies Definition \ref{def AV line}. We now set
$$l^{(1)}:=e_{i_1}-e_{i_0}\ \ ; \ \ l^{(2)}:=e_{i_2}-e_{i_0}$$ for some fixed $i_2=0,...,m-1$ different from  $i_0,i_1$. We also set
$c:= + \lambda_{i_2}-\lambda_{i_0}+<m_0,v_{i_2}>-<m_0,v_{i_0}>$. Notice that $c\in\RR$ has to be non-zero, otherwise $\mu_\Delta(L)$ intersects $F_{i_0,i_2}$ non-trivially at $m_0$ that contradicts $L$ being Agnatic-Vafa.

It suffices to prove the claim that
$$\{y\in (\RR^m)^\vee: <y,l^{(1)}>=0,<y,l^{(1)}>=c,\iota^\vee(y)=t\}=\vec{\lambda} +\beta^\vee(\mu_\Delta(L)).$$
We first observe that the vector subspaces $\RR<l^{(1)},l^{(2)}>$ and $\iota(G)$ has trivial intersection in $\RR^m$. By dimension counting, we have $\RR<l^{(1)},l^{(2)}>\oplus\iota(G)\subset \RR^m$ is codimension $(m-3)+2$ which is a line of the form $p_0+\RR v \subset (\RR^m)^\vee$ for some unique $p_0 \in \partial (\RR_{\geq 0}^m)^\vee $ and $v\in T_\cdot (\RR^m)^\vee \simeq (\RR^m)^\vee $ such that $p_0+\RR_{\geq 0}v \subset (\RR^m_{\geq 0})^\vee$. We can directly compute that $<l^{(1)},\vec{\lambda}+\beta^\vee(m_0)>=(<m_0,v_{i_1}>+ \lambda_{i_1})-(<m_0,v_{i_0}>+\lambda_{i_0})=0$ since $m_0$ is a point in $F_{i_0,i_1}=F_{i_0}\cap F_{i_1}$. And $<l^{(2)},\vec{\lambda}+\beta^\vee(m_0)>=c$ since this is how we set the value $c\in\RR^\times$. Thus $p_0=\vec{\lambda}+\beta^\vee(m_0)$ by uniqueness of the choice $p_0$. Moreover, we can compute directly that
$<l^{(1)},\beta^\vee(u)>=<v_{i_1},u>-<v_{i_0},u>=0$ and so as for $l^{(2)}$. Then $v=\beta^\vee(u)$ by uniqueness of the choice of $v$. Overall, we have $v=\beta^\vee(u)$ and $p_0=\vec{\lambda}+\beta^\vee(m_0)$. Thus we have proved the claim and also the proposition.
\end{proof}

\begin{rmk} \label{remarks WLOG define l1 and l2}
With out loss of generality, we assume that $i_0=0,i_1=1$ and $i_2=2$ which can be achieved by simply re-ordering the indexes of the rays $v_0,v_1,...,v_{m-1}\in\Sigma_\Delta(1)$. As a consequence, all Aganagic-Vafa A-branes can be viewed as special Lagrangian in $X_\Delta$ that is determined by the charges
$$ l^{(1)}=-e_0+e_1=(-1,1,0,...,0)\in \ZZ^m $$
$$l^{(2)}=-e_0+e_2=(-1,0,1,0...,0) \in \ZZ^m $$
after the re-ordering the indexes as described Definition \ref{def AV line}, Definition \ref{def AV brane} and Proposition \ref{prop AV brane vs AV line}.
\end{rmk}

\begin{rmk}
From the topological point of perspective, all Aganagic-Vafa branes are diffeomorphic to $\SSS^1 \times \CC$. By tracing the long exact sequence of homology, we obtain that
%And hence we immediately obtain
\begin{align}
\label{compute dim of H2 s}
\begin{CD}
 H_{2}(L)    @>>>     H_{2}(X)          @>>> H_2(X,L) @>>> H_1(L) @>>> H_1(X)   \\
 @VV\simeq V             @VV\simeq V    @VV\simeq V   @VV\simeq V       @VV\simeq V       \\
 0   @>>>   \ZZ^{m-3}    @>>> \ZZ^{m-2}  @>>> \ZZ @>>>0     \\
\end{CD}
%H^*(L,\mathds{Z})=H^0(L,\mathds{Z}) \oplus H^1(L,\mathds{Z}) \oplus H^2(L,\mathds{Z}) \oplus H^3(L,\mathds{Z})
%H^0(L,\mathds{Z})=\mathds{Z},\\H^1(L,\mathds{Z})=\mathds{Z},\\H^2(L,\mathds{Z})=0,\\H^3(L,\mathds{Z})=0.
\end{align} since the manifold $X_\Delta$ we are considering is a threefold.
\end{rmk}

%\begin{rmk}According to the definition of $L_l\in \frac{\mu^{-1}(t)}{G_\mathds{R}}  \cong X_\Delta$, we can explicitly represent the Aganagic-Vafa A-brane $L=L_l$ by systems of equations.That is\begin{equation}L=\{(z_1,...,z_r)|<r,q^{(1)}>=c ,<r,q^{(2)}>=0,\theta\equiv\theta_0 + \mathds{R}^r\ (mod\ G_\mathds{R})\}\end{equation}\end{rmk}

According to physics literature and based on physical considerations, physicists predict that a pair of mirror branes $L$ and $L^\vee$ exhibits surprising relationships between the open Gromov-Witten generating function on the A-side and the Abel-Jacobi map on the B-side. We re-formulate such surprising mirror prediction as the conjecture below:

\begin{conj} \label{mirror conjecture}
Let $L=L_{AV}$ in a Calabi-Yau 3-fold $X=X_\Delta^3$ be an Aganagic-Vafa A-brane as defined in Definition \ref{def AV brane}. Then its corresponding mirror B-brane is given by
\begin{align}
L^\vee = \left\{ (u,v,z_1,z_2)\in X^\vee \subset  (\CC^\times)^2\times\CC^2 :
\begin{array}{rl}
uv&=0\\
z_1&=q_0\\
z_2&=z_2(q_0)
\end{array}
\right\}
\end{align}
where $z_2(q_0)$ is the implicit function of $W(q_0,-)$ such that $W(q_0,z_2(q_0))=0$ and $q_0\in\CC^\times$ is the parameter of the deformation space of the subvariety $L^\vee$ in $X^\vee$. Moreover, $L$ is mirror to $L^\vee$ is the sense that the mirror theorem is satisfied, that is, $$F(\vec{Q})=\WW(\vec{q}) \text{  up to the (open/close) mirror map  } $$ where $\vec{Q}=(Q_0,Q_1,...,Q_{m-3})$ are the (open/close) K$\ddot{a}$hler parameter and $\vec{q}=(q_0,q_1,...,q_{m-3})$ are the (open/close) complex structure parameter. And we define $F(\vec{Q})$ and $\WW(\vec{q})$ by
\begin{align}
F(Q_0,Q_1,...,Q_{m-3})&:= \sum_{p=1}^\infty \sum_{b_0,...,b_{m-3}\geq 0} \frac{N_{b_0,...,b_{m-3}}}{p^2}Q_0^{pb_0}\cdots Q_{m-3}^{pb_{m-3}} \label{eqt F define}\\
\WW(q_0,q_1,...,q_{m-3}) &:= \int_{\Gamma_{q_0,q_1,...,q_{m-3}}} \Omega_{q_1,...,q_{m-3}}
\end{align}
with $\Omega_{q_1,...,q_{m-3}}$ be the volume form of $X^\vee$, $\Gamma_{q_0,q_1,...,q_{m-3}}$ be some open 3-chain in $X^\vee$ and $N_{b_0,...,b_{m-3}}$ be instanton number of the effective relative class $(b_0,...,b_{m-3})\in \ZZ^{m-2}\simeq H_2(X,L;\ZZ)$ as in \eqref{compute dim of H2 s}.
\end{conj}

This conjecture is proven mathematically in Fang-Liu\cite{Fang Liu}. With this fact, we now turn our focus to the explicit expression of the B-brane $L^\vee$. In other word, we want to express the coordinates $z_1,z_2$ of $L^\vee$ solely in terms of the symplectic information of its mirror A-side.

\begin{thm} \label{main theorem Lv=z1,z2}
Let $L$ be a Lagrangian in a Calabi-Yau 3-fold $X=X_\Delta^3$ as defined in Definition \ref{def AV brane}. Then its corresponding mirror B-brane is given by the coordinates:
\begin{enumerate}
	\item
\begin{align}
z_2=exp \Big( -\sum_{\beta\in H^{eff}_2(X,L;\ZZ)} n_\beta [\partial\beta] Q^\beta \Big)
\end{align}
	\item
\begin{align}
z_1&=Q_0exp \Big( \sum_{\alpha>0}E_1(\alpha)Q^\alpha\prod_{j=1}^{m-1}(\sum_{\alpha'>0}n_{\beta_j+\alpha'}Q^{\alpha'})^{\sum_{a=1}^{m-1}\alpha_a\iota_j^{(a)}} \Big)
\end{align}
\end{enumerate}
where $\alpha,\alpha'$ runs over $H^{eff}_2(X;\ZZ)$ and
	$Q^\beta:=Q_0^{b_0}\cdots Q_{m-3}^{b_{m-3}}$ for $(b_0,...,b_{m-3})\simeq \ZZ^{m-2}$ represents $\beta\in  H_2(X,L;\ZZ) \simeq \ZZ \oplus H_2(X;\ZZ)\simeq \ZZ^{m-2}$, 		
    $[\partial\beta]\in \ZZ$ represents $\partial \beta \in H_1(L;\ZZ)\simeq \ZZ$, 	
    $n_\beta\in \QQ$ is the open Gromov-Witten invariant with respect to the class $\beta\in  H_2(X,L;\ZZ)$ as defined in chapter (3.4), (3.5) of \cite{Fang Liu} and chapter (2.5) of \cite{Cho Chan Lau Tseng}. And $E_i(\alpha):=((-1)^{-<D_i,\alpha>-1}(-<D_i,\alpha>-1)!/\prod_{j\neq i}<D_j,\alpha>!)$  is a rational number as defined in Section 4.3 of \cite{Fang Liu} and Section 6.3 of \cite{Cho Chan Lau Tseng}.

In other words, we obtain the explicit defining equations for the mirror B-brane in terms of the A-side data as follows:
\begin{align*}
L^\vee = \left\{ %(u,v,z_1,z_2)\in X^\vee \subset  (\CC^\times)^2\times\CC^2 :\\
\begin{array}{rl}
&(u,v,z_1,z_2)\in X^\vee \subset  (\CC^\times)^2\times\CC^2 :\\
&uv=0\\
&z_1= Q_0 \cdot exp\Big( \sum_{\alpha} E_1(\alpha)
    \vec{Q}^\alpha
    \prod_{j=1}^{m-1} (1+\sum_{\alpha'} n_{\beta_j+\alpha'}\vec{Q}^{\alpha'})^{\sum_{a=1}^{m-3} \iota_j^{(a)} \alpha_a }  \Big) \\
&z_2=exp \Big( -\sum_{\beta>0} n_\beta [\partial\beta] \vec{Q}^\beta \Big)
\end{array}
\right\}
\end{align*}
\end{thm}

The proof of this theorem will be separated into 2 parts which will be presented in the next 2 sections immediately.

%-------------------------------------------------------------------------------------------------

\section{Mirror prediction}
As in the Mirror Conjecture \ref{mirror conjecture}, it is predicted that the open Gromov-Witten invariants of the Aganagic-Vafa A-brane $L$ in a Calabi-Yau 3-fold $X_\Delta^3$ can be computed by considering the superpotential $\WW$ of its corresponding mirror B-brane $L^\vee$ in the mirror manifold $X^\vee$. %More specifically, the instanton generating function of $L_{AV}$,
%\begin{equation}F(Q_0,...,Q_k):=\Sigma_{\beta}\frac{d_\beta}{k^2}Q^\beta\end{equation}where  $\beta$ runs over $H_2^{eff}(L_{AV},X_\Delta;\mathds{Z})$, and the Abel-Jacobi map of $L_{AV}^\vee$,\begin{equation}W(q_0,...,q_k)=\int_{\gamma(q_0)}\Omega_{q_1,...,q_k}\end{equation}should coincide up to the open/close mirror map.\par In this section, we are going to investigate the relationship between the instanton $d_\beta$ and the Aganagic-Vafa B-brane $L_{AV}^\vee$.

Recently, this mirror conjecture has been proven by Fang and Liu in \cite{Fang Liu}. We first state their result, then move on to express the equations that governs $L^\vee$. In other words, we make use of the mirror theorem to compute the coordinates $z_1,z_2$ of $L^\vee$ as in Theorem \ref{main theorem Lv=z1,z2}.

\begin{thm} \label{mirror theorm F=W by Fang}
Let $X$ and $X^\vee$ be a pair of mirror Calabi-Yau 3-folds. Suppose $L$ be an Aganagic-Vafa A-brane as in Definition \ref{def AV brane} with mirror brane $L^\vee$ as Conjecture \ref{mirror conjecture}. Then
$$F_0(Q_0,Q_1,...,Q_{m-3})=\WW_0(q_0,q_1,...,q_{m-3}) \text{  up to the (open/close) mirror map}$$ %\vec{Q}\longleftrightarrow\vec{q}
where
\begin{align}
F_0(\vec{Q})&:=\sum_{\beta\in H^{eff}_2(X,L;\ZZ)} n_\beta\vec{Q}^\beta \label{eqt define F0}\\
\WW_0(\vec{q})&:= \WW(\vec{q}) - \sum_{i,j=0}^{m-3} a_{ij} log(q_i)log(q_j)
\end{align}
such that $F_0\in \CC[[\vec{Q}]],\WW_0\in \CC[[\vec{q}]]$ are both power series with $\WW_0$ be the power series part of the Laurent series $\WW$, and $\vec{Q}^\beta=Q_0^{b_0}\cdots Q_{m-3}^{b_{m-3}}$ where $\beta$ corresponds to $(b_0,...,b_{m-3})$ under the identification $H_2(X,L;\ZZ)\simeq \ZZ^{m-2}$ as in \eqref{compute dim of H2 s}.
\end{thm}

\begin{proof}
The proof can be found in \cite{Fang Liu} which is the main result of the paper.
\end{proof}

Notice that the A-model generating functions appear in the mirror conjecture in Conjecture \ref{mirror conjecture} and the mirror theorem above Theorem \ref{mirror theorm F=W by Fang} has different definitions.However, we can show that $F_0$ and $F$ are the same via the multiple cover formula:

\begin{lem}
The series $F_0$ and $F$, as defined in \eqref{eqt define F0} and \eqref{eqt F define} respectively, are the same in $\CC[[Q_0,Q_1,...,Q_{m-3}]]$.
\end{lem}

\begin{proof}
By the multiple cover formula and analogues to the genus 0 closed Gromov-Witten invariants multiple cover formula, we assume the conjecture as in \cite{Lin} for genus 0 open Gromov-Witten invariants multiple cover formula, that is,
\begin{align}
n_\beta = \sum_{0< p|\beta} \frac{N_{\beta/p}}{p^2}
\end{align}
summing over integer $p\in\ZZ$ such that $(\beta/p)\in H_2^{eff}(X,L;\ZZ)$  where $\beta$ is an effective relative 2-cycle in $H_2^{eff}(X,L;\ZZ)$,
$n_\beta\in \QQ$ is the open Gromov-Witten invariants and $N_\beta\in \ZZ$ is the integral open instanton number. With this relation between $n_\beta$ and $N_\beta$. We now show that $F=F_0$. We first define $\tilde{N}:H_2^{eff}(X,L;\QQ) \longrightarrow \ZZ$ define by
\begin{align}
\tilde{N}_\beta =
\begin{cases}
	N_\beta, & \text{if $\beta\in H_2^{eff}(X,L;\ZZ)$} \\
    0  , & \text{if otherwise}
\end{cases}
\end{align}
According to definitions of the series $F_0$ and $F$ in $\CC[[Q_0,Q_1,...,Q_{m-3}]]$, we compute that
%\begin{align}F_0 &:=	\sum_{m>0} \sum_{\beta>0} \frac{N_\beta}{m^2}\vec{Q}^{m\beta} \\ &=	\sum_{m>0} \sum_{\beta>0, m|\beta} \frac{N_{\beta/m}}{m^2}\vec{Q}^{\beta} \\  &=	\sum_{m>0} (\sum_{\beta>0} \frac{\tilde{N}_{\beta/m}}{m^2}\vec{Q}^{\beta}) \\    &=	\sum_{\beta>0} (\sum_{m>0}  \frac{\tilde{N}_{\beta/m}}{m^2}) \vec{Q}^{\beta} \\    &=	\sum_{\beta>0} (\sum_{m>0,m|\beta}  \frac{N_{\beta/m}}{m^2}) \vec{Q}^{\beta} \\    &=	\sum_{\beta>0} n_\beta \vec{Q}^{\beta} =: F	\end{align}
$$F_0:=	\sum_{p>0} \sum_{\beta>0} \frac{N_\beta}{p^2}\vec{Q}^{p\beta}=	\sum_{p>0} \sum_{\beta>0, m|\beta} \frac{N_{\beta/p}}{p^2}\vec{Q}^{\beta}=\sum_{p>0} \sum_{\beta>0} \frac{\tilde{N}_{\beta/p}}{p^2}\vec{Q}^{\beta}$$
and on the other hand
$$F:=\sum_{\beta>0} n_\beta \vec{Q}^{\beta}=\sum_{\beta>0} (\sum_{p>0,p|\beta}  \frac{N_{\beta/p}}{p^2}) \vec{Q}^{\beta}=\sum_{\beta>0} \sum_{p>0}  \frac{\tilde{N}_{\beta/p}}{p^2} \vec{Q}^{\beta}$$
Thus we have $F=F_0\in\CC[[Q_0,Q_1,...,Q_{m-3}]]$ and hence the proof is completed.
\end{proof}

\begin{rmk}
Assume that the (open) multiple cover formula is true for Aganagic-Vafa branes in toric Calabi-Yau manifolds, then the above lemma implies that the mirror prediction as in Conjecture \ref{mirror conjecture} is true.
\end{rmk}

We observe that the Mirror Theorem \ref{mirror theorm F=W by Fang} involves the mirror map. In order to express $L^\vee$ as in Theorem \ref{main theorem Lv=z1,z2}, we need the explicit expression of the mirror map. In terms of the (open/closed) K$\ddot{a}$hler parameters $(Q_0,Q_1,...,Q_{m-3})$ and the (open/closed) complex structure parameters $(q_0,q_1,...,q_{m-3})$, we define the(open/closed) mirror map as follows:

\begin{defn} % Reference: Prop (5.2) of Fang-Liu. %
\label{define mirror map Mir}
With notations above, we define the (open/closed) mirror map
\begin{align*}
\MM ir:
\left\{
\begin{array}{c}
\text{Complex}\\
\text{Deformation}\\
\text{of $(X^\vee,L^\vee)$}
\end{array}\right\}
\longrightarrow
\left\{
\begin{array}{c}
\text{Symplectic}\\
\text{Deformation}\\
\text{of $(X,L)$}
\end{array}\right\}
%\text{\{Complex Deformation of $(X^\vee,L^\vee)$\}}\longrightarrow\text{\{Symplectic Deformation of $(X,L)$\}}
\end{align*}
by $\MM ir:(q_0,q_1,...,q_{m-3})\longmapsto(Q_0,Q_1,...,Q_{m-3})$ such that
\begin{align} \label{definition QA in terms of qB}
Q_a=q_aexp\Big( \sum_{j=0}^{m-1}l_j^{(a)} A_j(q_1,...,q_{m-3}) \Big)
\end{align}
for each $a=0,1,2$ where $A_j\in \QQ[[q_1,...,q_{m-3}]]$ is a power series defined by $A_j(\vec{q}):= \sum_{\beta>0} E_j(\beta) \vec{q}^\beta$, $\vec{l}_0$ is an extra charge defined by $\vec{e}_1-\vec{e}_0\in \ZZ^m$ and $\vec{l}_1,\vec{l}_2\in\ZZ^m$ are defined as in Remarks \ref{remarks WLOG define l1 and l2}.
\end{defn}

With this particular form of the (open/closed) mirror map $\MM ir$, we observe that there is an equality between $q_0\partial/\partial q_0$ and $Q_0\partial/\partial Q_0$ via the push-forward map $\MM ir_*$. This equality will eventually help us the part (1) of our main Theorem \ref{main theorem Lv=z1,z2}.

\begin{lem} \label{lemma Qd/dQ = qd/dq}
Let $\MM ir$ be the mirror map as above Definition \ref{define mirror map Mir}. Then
$$\MM ir_*(q_0\frac{\partial}{\partial q_0})=Q_0\frac{\partial}{\partial Q_0}$$
where $\MM ir_*$ is the push-forward map of $\MM ir$ between the tangent bundles.
\end{lem}

\begin{proof}
By definition of the mirror map $\MM ir:(q_0,q_1,...,q_{m-3})\longmapsto(Q_0,Q_1,...,Q_{m-3})$, we compute the push forward of $(\partial/\partial q_0)$ by chain rule, that is,
$$ \MM ir_*(\frac{\partial}{\partial q_0}) = \sum_{j=0}^{m-3} \frac{\partial Q_j}{\partial q_0}\cdot \frac{\partial}{\partial Q_j} $$
According to \eqref{definition QA in terms of qB}, we have
$$\frac{\partial Q_1}{\partial q_0} = \cdots = \frac{\partial Q_{m-3}}{\partial q_0} = 0 $$
since $A_1,...,A_{m-3}$ are functions of $(q_1,...,q_{m-3})$ and also
$$\frac{\partial Q_0}{\partial q_0} =
\frac{\partial}{\partial q_0}\Big(q_0exp \sum_{j=0}^{m-1}l_j^{(0)} A_j(q_1,...,q_{m-3}) \Big)
=exp\Big( \sum_{j=0}^{m-1}l_j^{(0)} A_j(q_1,...,q_{m-3}) \Big)=\frac{Q_0}{q_0}.$$
Then we have
$$\MM ir_*(q_0\frac{\partial}{\partial q_0}) = Q_0 \frac{\partial}{\partial Q_0}.$$
Thus the proof is completed.
\end{proof}

With above preparations, we are now ready to prove the first part of Theorem \ref{main theorem Lv=z1,z2}, that is, the mirror $L^\vee$ can be expressed as
\begin{align*}
L^\vee =
\left\{
(u,v,z_1,z_2)\in X^\vee \subset  (\CC^\times)^2\times\CC^2 :
\begin{array}{rllll}
%&(u,v,z_1,z_2)\in X^\vee \subset  (\CC^\times)^2\times\CC^2 :
&uv=0 \ \ ;\ \
z_1=q_0\ \ ;\\%Q_0exp \Big( \sum_{i=1}^{r-1}\sum_{\alpha}l^{(0)}_iE_i(\alpha)Q^\alpha\prod_{j=1}^{r-1}(\sum_{\alpha'}n_{\beta_j+\alpha'}Q^{\alpha'})^{\sum_{a=1}^{r-1}\alpha_a\iota_j^{(a)}} \Big)\\
&z_2=exp \Big( -\sum_{\beta\in H^{eff}_2(X,L;\ZZ)} n_\beta [\partial\beta] Q^\beta \Big)
\end{array}
\right\}
\end{align*}

\begin{proof}[Proof of Theorem \ref{main theorem Lv=z1,z2} part(1)]

Let $L$ and $L^\vee$ be mirror pair as defined in Conjecture \ref{mirror conjecture}. According to the Mirror Theorem \ref{mirror theorm F=W by Fang}, we have $$F(\vec{Q}(\vec{q}))=\WW_0(\vec{q})$$ where $\vec{Q}(\vec{q})$ is the (open/close) mirror map form the complex structure moduli to the K$\ddot{a}$hler moduli.
And by the previous Lemma \ref{lemma Qd/dQ = qd/dq}, we then have
$$Q_0\frac{\partial}{\partial Q_0}F(Q_0,Q_1,...,Q_{m-3})=q_0\frac{\partial}{\partial q_0}\WW_0(q_0,q_1,...,q_{m-3}).$$
By expressing $F(Q)$ and $W(q)$, we have
$$Q_0\frac{\partial}{\partial Q_0}\sum_{\beta>0}\sum_{p>0} \frac{N_\beta}{p^2}\vec{Q}(\vec{q})^{p\beta}=-log\ z_2(q_0,q_1,...,q_{m-3}).$$
By the multiple cover formula for open Gromov-Witten invariants, we compute that
\begin{align*}
z_2(q_0,q_1,...,q_{m-3})
	&= exp(-Q_0\frac{\partial}{\partial Q_0}\sum_{\beta>0} n_\beta \vec{Q}^\beta) \\
    &= exp(-Q_0\frac{\partial}{\partial Q_0}\sum_{\beta>0} n_\beta Q_0^{b_0}\cdots Q_{m-3}^{b_{m-3}}) \\
    &= exp(-\sum_{\beta>0} n_\beta b_0 Q_0^{b_0}\cdots Q_{m-3}^{b_{m-3}}) \\
    &= exp(-\sum_{\beta>0} n_\beta [\partial \beta] Q_0^{b_0}\cdots Q_{m-3}^{b_{m-3}})
\end{align*}
Thus the proof is completed.

%$$Q_0\frac{\partial}{\partial Q_0}\Sigma_\beta n_\beta Q(q)^\beta=-log\ y(q_0,q_1,...,q_k).$$and$$\Sigma_\beta n_\beta [\partial\beta]Q(q)^\beta=-log\ y(q_0,q_1,...,q_k).$$Hence we have the desired result$$y(q_0,q_1,...,q_k)=exp(-\Sigma_\beta n_\beta [\partial \beta] Q(q)^\beta)$$and by substituting $q=q(Q)$, we have$$y(q(Q))=exp(-\Sigma_\beta n_\beta [\partial \beta] Q^\beta).$$

\end{proof}

Instead of expressing the coordinate $z_2$ of $L^\vee$ by the implicit function $z_2(z_1)$ of the mirror curve $W(z_1,z_2)=0$ as in Proposition \ref{well defined in R}. By Theorem \ref{main theorem Lv=z1,z2} part (1), we now have more explicit formula for $z_2$ in terms of the (open/closed) K$\ddot{a}$hler parameters $\vec{Q}$. A natural question that follows is to look for similar expression of the coordinate $z_1$ in terms of $\vec{Q}$ which will be investigated in the next section.

\section{SYZ mirror map}

The goal of this section is to present a proof of Theorem \ref{main theorem Lv=z1,z2} part (2) and thus complete the proof of the theorem. The main reasoning involves the use of the SYZ mirror map introduced recently  in \cite{Chan Lau Leung}. The main result of that paper \cite{Cho Chan Lau Tseng} is that SYZ mirror map is nothing but the inverse of the mirror map $\MM ir$.

\begin{defn} \label{defn of SYZ mirror map}
Using the notations as we defined the mirror map $\MM ir$ in Definition \ref{define mirror map Mir}, we define the SYZ mirror map $\MM ir^{SYZ}$:
\begin{align*}
\MM ir^{SYZ}:
\left\{
\begin{array}{c}
\text{Symplectic}\\
\text{Deformation}\\
\text{of $(X,\omega)$}
\end{array}\right\}
\longrightarrow
\left\{
\begin{array}{c}
\text{Complex}\\
\text{Deformation}\\
\text{of $(X^\vee,J^\vee)$}
\end{array}\right\}
%\text{\{Complex Deformation of $(X^\vee,L^\vee)$\}}\longrightarrow\text{\{Symplectic Deformation of $(X,L)$\}}
\end{align*}
by $\MM ir^{SYZ}:(Q_1,...,Q_{m-3})\longmapsto(q_1,...,q_{m-3})$ such that
$$q_a=Q_a \prod_{j=0}^{m-1}(1+\delta_j)^{\iota_{j}^{(a)} }$$
for each $a=1,...,m-3$ where $(1+\delta_j)$ is the generation functions of genus 0 open Gromov-Witten invariants of $(X,L_+)$ as defined in \eqref{definition of delta i} and $(\iota^{(a)}_j)^{1\leq a \leq m-n}_{0\leq j \leq m-1}$ represents the matrix of the linear map $\iota:\ZZ^{m-n}\simeq G \Rightarrow \ZZ^m$ as in \eqref{toric exact sequence}.
%where $A_j\in \QQ[[q_1,...,q_k]]$ is a power series defined by $A_j(\vec{q}):= \sum_{\beta>0} E_j(\beta) \vec{q}^\beta$.
\end{defn}

Notice that the definition of the SYZ mirror map $\MM ir^{SYZ}$ involves the open Gromov-Witten invariants for SYZ fibers. The below theorem which relates $\MM ir^{SYZ}$ and $\MM ir$ implies that the mirror map $\MM ir$, that is the series $$Q_a=q_aexp\Big( \sum_{j=0}^{m-1}l_j^{(a)} A_j(q_1,...,q_{m-3}) \Big), $$ surprisingly encodes the information of the open Gromov-Witten invariants $n_\beta$ for SYZ fibers.

\begin{thm} \label{thm SYZ is inv of Mir map}
Let $X$ be a toric Calabi-Yau 3-fold. Then the SYZ mirror map $\MM ir^{SYZ}$ is inverse to the mirror map $\MM ir$ locally around $(Q_1,...,Q_{m-3})=\vec{0}$, that is $\MM ir^{SYZ}=(\MM ir)^{-1}$ on some neighborhood of $(Q_1,...,Q_{m-3})=\vec{0}$.
\end{thm}

\begin{proof}
The proof can be found in Cho-Chan-Lau-Tseng\cite{Cho Chan Lau Tseng}, Section 7.2 which is one of the main results of the paper.
\end{proof}

With the definitions of the 2 mirror maps in Definition \ref{define mirror map Mir} and \ref{defn of SYZ mirror map}, we immediately obtain some relationships between the closed complex parameters $q_1,...,q_k$ and the closed K$\ddot{a}$hler parameters $Q_1,...,Q_k$ and we also obtain an equality between the open K$\ddot{a}$hler parameter $Q_0$ and the open and closed complex parameters $q_0,...,q_k$.

\begin{lem}
With the notations as in the definition of the mirror map and the SYZ mirror map in Definition \ref{define mirror map Mir} and \ref{defn of SYZ mirror map}, we have
\begin{enumerate}
	\item
    	$q_a=Q_a \prod_{j=1}^{m-1}\Big(1+\sum_{\alpha >0} n_{\beta_j+\alpha}\vec{Q}^\alpha)\Big)^{\iota_j^{(i)}}$
        for $a=1,...,m-3$, and
	\item
    	$Q_0=q_0 \cdot exp \Big( -\sum_{\alpha>0}  E_1(\alpha)   q_1^{a_1}\cdots q_{m-3}^{a_{m-3}} \Big)$.
\end{enumerate}
\end{lem}

\begin{proof}
To prove part (1), we let $a=1,...,{m-3}$. Under the SYZ mirror map, we consider
$$q_a=(\MM ir^{SYZ}(Q_1,...,Q_{m-3}))_a=Q_a \prod_{j=0}^{m-1}(1+\delta_j)^{\iota_{j}^{(a)} }.$$
According to definition of the $\delta_j$'s in \eqref{definition of delta i}, we have $\delta_j=\sum_{\alpha>0} n_{\beta_j+\alpha}exp(-\int_\alpha \omega)$. We recall the definition of the K$\ddot{a}$hler parameters $Q_j=exp(-\int_{\alpha_j} \omega)$ for $j=1,...,{m-3}$ where we choose the 2-cycles $\alpha_j$ by the identification $\bigoplus_{j=1}^{m-3}\ZZ\cdot \alpha_j = H_2(X;\ZZ) \simeq \ZZ^{m-3}$. Hence, for $\alpha = a_1\alpha_1+\cdots+a_{m-3}\alpha_{m-3} \in H_2(X;\ZZ)$, we have $exp(-\int_\alpha \omega)=Q_1^{a_1}\cdots Q_{m-3}^{a_{m-3}}$. Then we obtain
$$q_a=Q_a\prod_{j=0}^{m-1}\Big( 1+\sum_{\alpha>0}n_{\beta_j+\alpha}Q_1^{a_1}\cdots Q_{m-3}^{a_{m-3}} \Big) ^ {\iota_j^{(a)}}$$ which completes part (1).

To prove part (2), we directly follow the definition of $\MM ir$ and we have
$$Q_0=\big(\MM ir(q_0,q_1,...,q_{m-3})\big)_0=q_0 exp\big( \sum_{j=0}^{m-1} l_j^{(a)} A_j(q_1,...,q_{m-3}) \big).$$
We recall that $A_j(q_1,...,{m-3})=\sum_{\alpha>0}E_j(\alpha)q_1^{a_1}\cdots q_{m-3}^{a_{m-3}}$ and $E_j(\alpha)\in\QQ$ according to their definitions in Definition \ref{define mirror map Mir}. And by definition of the extra charge $l^{(0)}\in \ZZ^m$, we have $l^{(0)}=\vec{e_2}-\vec{e_1}$. Hence we obtain
\begin{align*}
Q_0
	&= q_0\cdot exp \Big( \sum_{j=0}^{m-1} l_j^{(0)} \big( \sum_{\alpha>0} E_j(\alpha) q_1^{a_1}\cdots q_{m-3}^{a_{m-3}} \big) \Big)\\
    &=q_0 \cdot exp \Big( \sum_{\alpha>0} \big( \sum_{j=0}^{m-1}   l_j^{(0)} E_j(\alpha)  \big) q_1^{a_1}\cdots q_{m-3}^{a_{m-3}} \Big) \\
    &=q_0 \cdot exp \Big(\sum_{\alpha>0}  E_1(\alpha) q_1^{a_1}\cdots q_{m-3}^{a_{m-3}} \Big)
\end{align*}

\end{proof}
This lemma is important in proving part (2) of Theorem \ref{main theorem Lv=z1,z2} since that involves expressing the open complex parameter $q_0$ in terms of the open and closed K$\ddot{a}$hler parameters $Q_0,Q_1,...,Q_{m-3}$. With above preparations, we are now ready to prove it, that is
$$z_1=Q_0exp \Big( \sum_{i=1}^{m-1}\sum_{\alpha}l^{(0)}_iE_i(\alpha)Q^\alpha\prod_{j=1}^{m-1}(\sum_{\alpha'}n_{\beta_j+\alpha'}Q^{\alpha'})^{\sum_{a=1}^{m-1}\alpha_a\iota_j^{(a)}} \Big)$$

\begin{proof}[Proof of Theorem \ref{main theorem Lv=z1,z2} part(2)]
We first define a new quantity $E_0(\alpha):=E_1(\alpha)-E_2(\alpha)$ for any $\alpha\in H^{eff}_2(X;\ZZ)$ which is analogous to the definition of the extra charge $l^{(0)}=e_{i_1}-e_{i_2}$.
Using the above lemma, we substitute (1) into (2) to get
\begin{align}
	q_0
	&= Q_0 \cdot exp\Big( \sum_{\alpha>0} E_1(\alpha) q_1^{\alpha_1}\cdots q_{m-3}^{\alpha_{m-3}} \Big)\\
	&= Q_0 \cdot exp\Big( \sum_{\alpha>0} E_1(\alpha) \prod_{a=1}^{m-3}\big(Q_a\prod_{j=1}^{m-1}(1+\delta_j)^{\iota_j^{(a)}} \big) ^{\alpha_a } \Big)\\
    &= Q_0 \cdot exp\Big( \sum_{\alpha>0} E_1(\alpha)
    \big(\prod_{a=1}^{m-3} Q_a \big) ^{\alpha_a }
    \big(\prod_{a=1}^{m-3}\prod_{j=1}^{m-1}(1+\delta_j)^{\iota_j^{(a)}} \big) ^{\alpha_a }
    \Big)\\
    &= Q_0 \cdot exp\Big( \sum_{\alpha>0} E_1(\alpha)
    \vec{Q}^\alpha
    \prod_{j=1}^{m-1} (1+\delta_j) ^{\sum_{a=1}^{m-3} \iota_j^{(a)} \alpha_a }  \Big)\\
    &= Q_0 \cdot exp\Big( \sum_{\alpha>0} E_1(\alpha)
    \vec{Q}^\alpha
    \prod_{j=1}^{m-1} (1+\sum_{\alpha >0} n_{\beta_j+\alpha}\vec{Q}^\alpha) ^{\sum_{a=1}^{m-3} \iota_j^{(a)} \alpha_a }  \Big)
\end{align}
This completes the proof of part (2) Theorem \ref{main theorem Lv=z1,z2}.
\end{proof}

Together with the proof of part (1) in the previous chapter, we have completed the proof of the theorem. Thus we obtain the expression of $L^\vee$ in term of the symplectic geometric information of $(X,L)$ to be
\begin{align*}
L^\vee = \left\{ %(u,v,z_1,z_2)\in X^\vee \subset  (\CC^\times)^2\times\CC^2 :\\
\begin{array}{rl}
&(u,v,z_1,z_2)\in X^\vee \subset  (\CC^\times)^2\times\CC^2 :\\
&uv=0\\
&z_1= Q_0 \cdot exp\Big( \sum_{\alpha>0} E_1(\alpha)
    \vec{Q}^\alpha
    \prod_{j=1}^{m-1} (1+\sum_{\alpha >0} n_{\beta_j+\alpha}\vec{Q}^\alpha) ^{\sum_{a=1}^{m-3} \iota_j^{(a)} \alpha_a }  \Big) \\
&z_2=exp \Big( -\sum_{\beta>0} n_\beta [\partial\beta] \vec{Q}^\beta \Big)
\end{array}
\right\}
\end{align*}

Throughout this section and the previous chapter, we have examined $L^\vee$ and $L^\vee_{naive}$ of a given Aganagic-Vafa A-brane $L\subset X_\Delta$ respectively. Ideally, our best scenario is to construct $L^\vee_{naive}$ under naive SYZ transformation which gives the physicists' expected mirror $L^\vee$. However, $L^\vee_{naive}$ seems to be a bit different from $L^\vee$ by examining their defining equations. This suggests that the naive SYZ transformation needs to be modified which will be discussed in the next two sections.

\section{Mirrors comparison}
In this section, we apply our naive SYZ mirror brane construction as in Definition \ref{main def B} of Chapter 3 to the special case of Aganagic-Vafa A-branes $L\subset X_\Delta$. Then, we will be ready to compare the defining equations between $L_{naive}^\vee$ and $L^\vee$.

\begin{thm}
Let $L\in X_\Delta$ be an Aganagic-Vafa brane defined by the 2 charges $l^{(1)}=(-1,1,0,...,0)$ and $l^{(1)}=(-1,0,1,0,...,0)$ in $\ZZ^m$ and $X^\vee$ be its mirror variety defined by the equation $uv=W(z_1,z_2)$.
Then we have an explicit expression of the mirror brane $L_{naive}^\vee$ as follows:
\begin{align}
L_{naive}^\vee = \left\{ (u,v,z_1,z_2)\in X^\vee:
    \begin{array}{lll}
        z_1=Q_0\\
        z_2=1
    \end{array} \right\}
\end{align}
where $Q_0$ is the complexified open K$\ddot{a}$hler parameter.
\end{thm}

\begin{proof}
Using above notations of the Aganagic-Vafa brane $(L,\nabla)$, we recall from the previous chapter that we can explicitly express the naive SYZ mirror brane $L^\vee_{naive}$ as follows:
\begin{align}
	L^\vee_{naive}:=\left\{
    (u,v,z_1,z_2) \in X^\vee :
		\begin{array}{lll}
        uv=W(z_1,z_2)\\
        z_1=exp(-c^{(1)}-i\phi^{(1)})\\
        z_2=exp(-c^{(2)}-i\phi^{(2)})
		\end{array}
        \right\}
\end{align}
where $l^{(1)},l^{(2)}\in \ZZ^m$ are charges as given in the statement, $c^{(1)}\neq 0,c^{(2)}=0\in \RR$ are constants that defines the Aganagic-Vafa brane $L$ as in Definition \ref{definition of L in X} and $e^{i\phi^{(1)}},e^{i\phi^{(2)}} \in U(1)$ are also constants that defines the flat $U(1)$-connection $\nabla$ via the identification between $U(1)=Hom(\pi_1(L),U(1))$ and the moduli of isomorphism class of flat $U(1)$-connections.

To show that $z_2=1$, we consider the constants $c^{(2)}$ and $\phi^{(2)}$. By assumption of Aganagic-Vafa branes, we can assume that $c^{(2)}=0$ using above notations. On the other hand, we let $\mu_\Delta(L)=m_0 + \RR\cdot u \in \Delta$ for some unique $m_0\in \partial \Delta$ as in the proof of Proposition \ref{prop AV brane vs AV line}. We observe that
$$ L \cap \mu_\Delta (m_0) \subset
\{(r_0e^{i\theta_0},...,r_{m-1}e^{i\theta_{m-1}})\in \CC^m: r_2=r_0=0 \} $$ since $m_0$ is the intersection of the moment map Lagrangian image$\mu_\Delta(L)$, faces $F_2$ and $F_0$. We also notice that $\RR\cdot l^{(2)}$ acts on the subset $\{(r_0e^{i\theta_0},...,r_{m-1}e^{i\theta_{m-1}})\in \CC^m: r_2=r_0=0 \} $ trivially since $l^{(2)}=(-1,0,1,0,...,0)=e_2-e_0$ which acts on the angle coordinates $r_0e^{i\theta_0}$ and $r_2e^{i\theta_2}$ only. Let $\gamma$ be a $\RR\cdot l^{(2)}$ orbit, that is $\gamma=(\SSS\cdot l^{(2)})\cdot x$ for some $x\in L$. Immediately, we obtain that $\gamma$ is contractible. According to the definition of $\phi^{(2)}$, we have $$e^{i\phi^{(2)}}=Hol_\Delta(\gamma)=Hol_\Delta(point)=1\in U(1)$$ since the connection $\Delta$ is flat by assumption. Thus we have that $$z_2=exp(0+i0)=1.$$

To show that $z_1=Q_0$, we consider definition of the complexified open K$\ddot{a}$hler parameter $Q_0$. By definition $$Q_0=exp(-\int_{\beta_0}\omega)\cdot Hol_\Delta(-\partial\beta_0)$$ where $\beta_0$ is a disc class in $H_2(X,L;\ZZ)$ which is isomorphic to $H_1(L;\ZZ)\oplus H_2(X;\ZZ)=\ZZ<\partial\beta_0>\oplus \ZZ<
\alpha_1,...,\alpha_{m-3}>$. By the result of \cite{Cho Oh}, we can compute that  $\int_{\beta_0}\omega =r_1^2-r_0^2=c^{(1)}$ which is a constant according the definition of the Aganagic-Vafa brane $L$. On the other hand, we see from definition that $exp(i\phi^{(1)})=Hol_\Delta(\partial\beta_0).$ Hence we obtain that
%$$Q_0=exp(-\int_{\beta_0}\omega)\cdot Hol_\Delta(-\partial\beta_0)=exp(-c^{(1)}-i\phi^{(1)})=z_1$$
$$z_1=exp(-c^{(1)}-i\phi^{(1)})=exp(-\int_{\beta_0}\omega)\cdot Hol_\Delta(-\partial\beta_0)=Q_0.$$

Combining above computations, we obtain the defining equation of $L^\vee$ as
\begin{align}
	L^\vee_{naive}:=\left\{
    (u,v,z_1,z_2) \in X^\vee :
		\begin{array}{lll}
        uv=W(z_1,z_2)\\
        z_1=exp(-c^{(1)}-i\phi^{(1)})=Q_0\\
        z_2=exp(-c^{(2)}-i\phi^{(2)})=1
		\end{array}
        \right\}.
\end{align}
Thus the proof is completed.
\end{proof}

%\begin{rmk}	$L_{naive}^\vee \neq L^\vee$ but $L_{naive}^\vee \approx L^\vee$. That is,	\begin{align}	L_{naive}^\vee = \left\{ (u,v,z_1,z_2)\in X^\vee:    \begin{array}{lll}        z_1=Q_0+\cdots\\        z_2=1+\cdots    \end{array} \right\}\end{align}	\end{rmk}

\begin{rmk}
Ideally, we would expect that the naive SYZ mirror brane $L_{naive}^\vee$ should coincide with the mirror B-brane $L^\vee$ as predicted by physicists in the Conjecture \ref{mirror conjecture} . However, the two mirror are not necessarily the same. Consider the example $X_\Delta=\OO_{\PP^1}(-1)\oplus\OO_{\PP^1}(-1)$ as in \cite{Aganagic Vafa}, its mirror $X^\vee$ is given by the equation $uv=W(z_1,z_2)=1+z_1+z_2+q_1z_2z^{-1}_1$. We consider $L\subset X_\Delta$ to be an Aganagic-Vafa brane as considered in \cite{Aganagic Vafa}. Its predicted mirror $L^\vee$ is given by
\begin{align}
L^\vee = \left\{ (u,v,z_1,z_2)\in X^\vee:
    \begin{array}{lll}
    	uv=1+z_1+z_2+q_1z_2z^{-1}_1\\
        z_1=q_0\\
        z_2=-(1+q_0)/(1+q_1q_0^{-1})
    \end{array} \right\}
\end{align}
After direct computation, the SYZ mirror $L^\vee_{naive}$ is given by
\begin{align}
L^\vee_{naive} = \left\{ (u,v,z_1,z_2)\in X^\vee:
    \begin{array}{lll}
    	uv=1+z_1+z_2+q_1z_2z^{-1}_1\\
        z_1=Q_0\\
        z_2= 1
    \end{array} \right\}
\end{align}
In this particular example $Q_0=q_0$ is the open mirror map, so $z_1$ coordinates of both mirror coincide. However $z_2$ coordinates of both mirrors are obviously different as the open and closed complex parameter $q_0,q_1$ varies. Moreover, $z_1,z_2$ coordinates of $L^\vee$ solves the mirror curve equation $W(z_1,z_2)=0$, but that of $L^\vee_{naive}$ do not.
\end{rmk}
\begin{rmk}
Although $L^\vee_{naive}$ does not coincide with $L^\vee$ as we have expected, we observe that $L^\vee_{naive}$ and $L^\vee$ are similar in the sense that they coincide up to higher order terms of the closed K$\ddot{a}$hler parameter $Q_1,...,Q_{m-3}$. More precisely, we compare that

\begin{align*}
L^\vee = &\left\{ %(u,v,z_1,z_2)\in X^\vee \subset  (\CC^\times)^2\times\CC^2 :\\
\begin{array}{rl}
&(u,v,z_1,z_2)\in X^\vee \subset  (\CC^\times)^2\times\CC^2 :\\
&uv=0\\
&z_1=Q_0exp \Big( \sum_{\alpha}E_1(\alpha)Q^\alpha\prod_{j=1}^{m-1}(\sum_{\alpha'}n_{\beta_j+\alpha'}Q^{\alpha'})^{\sum_{a=1}^{m-1}\alpha_a\iota_j^{(a)}} \Big)\\
&z_2=exp \Big( -\sum_{\beta\in H^{eff}_2(X,L;\ZZ)} n_\beta [\partial\beta] Q^\beta \Big)
\end{array}
\right\} \\
&= \left\{ (u,v,z_1,z_2)\in X^\vee:
    \begin{array}{lll}
    	uv=W(z_1,z_2)=0\\
        z_1=Q_0(1+O(Q_1,...,Q_{m-3}))\\
        z_2= 1+O(Q_1,...,Q_{m-3})
    \end{array} \right\}
\end{align*}
and
\begin{align}
L^\vee_{naive} = \left\{ (u,v,z_1,z_2)\in X^\vee:
    \begin{array}{lll}
    	uv=W(z_1,z_2)\\
        z_1=Q_0\\
        z_2= 1
    \end{array} \right\}
\end{align}
\end{rmk}

We now explicitly see the differences between the defining equations for the naive SYZ mirror brane $L^\vee_{naive}$ and the predicted mirror brane $L^\vee$. We observe that  $L^\vee_{naive}$ and $L^\vee$ coincide up to higher order term of the closed K$\ddot{a}$hler parameters $Q_1,...,Q_{m-3}$ with coefficient involving the open Gromov-Witten invariants of the A-brane itself. Therefore, we are now ready to introduce the modified SYZ mirror brane construction with further quantum correction in the next section.

%%%%%%%%%%%%%%%%%%%%%%%%%%%%%%%%%%%%%%%%%%%%%%%%%%%%%%%%%%
\section{Further quantum correction}
In this section, we modify the naive SYZ transform exactly according to the differences between the predicted Aganagic-Vafa mirror brane $L^\vee$ and the naive SYZ mirror brane $L^\vee_{naive}$ which has been investigated in the previous section. After comparing the two mirror branes, we now understand how the naive SYZ mirror brane $L^\vee_{naive}$ fails to capture the information contributed by the discs counting of the A-brane. To modify the naive SYZ transformation, we exactly introduce the missing higher order terms that involves discs counting data of the A-brane to our naive SYZ transformation which we will refer it as 'further quantum correction'. We are now ready to define the SYZ transformation with further quantum correction as follows:

\begin{defn}\label{def SYZ with further quantum correction}
Using notations as before, suppose that $(L,\nabla)$ is an Aganagic-Vafa A-brane defined by the charges $\vec{l}_1=(-1,1,0,...,0)$ and $\vec{l}_2=(-1,0,1,0...,0)$ in $\ZZ^m$ as defined in Definition \ref{def AV brane}. And we also suppose that $X_\Delta$ and $X^\vee$ be mirror partners with $X_\Delta$ being a toric Calabi-Yau threefold A-model and $X^\vee$ being its mirror variety B-model $X^\vee$ as described in Definition \ref{main def A}. Then we define the SYZ transformation with further quantum correction as follows:
\begin{align}
L^\vee_{SYZ} = &\left\{ %(u,v,z_1,z_2)\in X^\vee \subset  (\CC^\times)^2\times\CC^2 :\\
\begin{array}{rl}
&(u,v,z_1,z_2)\in X^\vee \subset  (\CC^\times)^2\times\CC^2 :\\
%&uv=0\\
&z_1=Q_0exp \Big( \sum_{i=1}^{m-1}\sum_{\alpha}l^{(0)}_iE_i(\alpha)Q^\alpha\prod_{j=1}^{m-1}(\sum_{\alpha'}n_{\beta_j+\alpha'}Q^{\alpha'})^{\sum_{a=1}^{m-1}\alpha_a\iota_j^{(a)}} \Big)\\
&z_2=exp \Big( -\sum_{\beta\in H^{eff}_2(X,L;\ZZ)} n_\beta [\partial\beta] Q^\beta \Big)
\end{array}
\right\}
\end{align}
where $Q_0,Q_1,...,Q_{m-3}\in \CC^\times$ are the open and closed K$\ddot{a}$hler parameters, $n_\beta$ is the open Gromov-Witten invariant with respect to some relative homology class $\beta$ and $\vec{l}_0:=\vec{l}_2-\vec{l}_1$ is an auxiliary charge vector in $\ZZ^m$.
\end{defn}
By construction, we simply see that the SYZ transformation with further quantum correction give correct prediction to A-branes of Aganagic-Vafa type. Automatically, the coordinates $(z_1,z_2)$ will be a root of the mirror curve equation $W$. Moreover, we should expect our definition of SYZ transformation with further quantum correction continues to work well in other examples. The following remark illustrates that the SYZ transformation with further quantum correction correctly explains mirror phenomenon in the case of $A_n$-resolutions.
\begin{rmk}
In another special case, namely the case of $A_n$-resolutions, Chan has proved that the naive SYZ transformation induces an equivalence between the derived Fukaya category of the A-model and the derived category of the B-model as formulated in \cite{Chan,Chan Ueda,Chan Ueda Pomerleano}. One can check that the open Gromov-Witten invariants $n_\beta$ for every A-brane $(L,\nabla)$ are trivial. In a rough sense, having trivial open Gromov-Witten invariants of every A-branes implies that the SYZ transformation for A-branes has trivial further quantum correction. In other words, the naive SYZ transformation is suffice to explain mirror phenomenon, for example the HMS conjecture.
\end{rmk}
Instead of verifying our SYZ transformation in few explicit examples, we should also try to explore its behavior in general cases. However, not much is known yet.
\begin{rmk}
In general situations, it is still unknown that our definition of the SYZ transformation with further quantum correction can give correct prediction for mirror B-branes. In this very moment, not much is known and we only know that the SYZ transformation with further quantum correction works well in the case of A-branes of Aganagic-Vafa type and A-branes in A-models of $A_n$-resolution type. How well the further quantum correction is formulated as in Definition \ref{def SYZ with further quantum correction} is still unknown in general cases other than the special cases we have just mentioned.
\end{rmk}
Next, we are going to discuss possible direction for future development concerning the SYZ transform.

%%%%%--------------Next Chapter---------------------------%%%%%%
\chapter{Future development}
Ideally, we would expect that the modified SYZ transformation is a general procedure that transforms Dirichlet branes correctly. To be more aggressive, we also expect that the modified SYZ transformation with further quantum correction should induce a mirror functor which is an equivalence between the derived Fukaya category $\mathcal{DF}uk(X)$ and the derived category $\mathcal{D}^b\mathcal{C}oh(X^\vee)$. In other words, we expect that the geometry of the \textit{Homological Mirror Symmetry} can be revealed by the \textit{SYZ Mirror Symmetry}.

To begin with, we first give an evidence that the SYZ transformation explains the phenomenon of the homological mirror conjecture in certain situations.

\begin{thm}[Chan\cite{Chan}, Chan-Ueda\cite{Chan Ueda}, Chan-Ueda-Pomerleano\cite{Chan Ueda Pomerleano}] \label{chan's thm on HMS}
%SYZ $\Rightarrow$ HMS.
Let $$Y=\{(u,v,z)\in \CC^2 \times \CC^\times : uv=W(z)\}$$ with the symplectic form $\omega=(-i/2)(du\wedge d\bar{u} + dv\wedge d\bar{v} + dlog(z)\wedge dlog\bar{z}) $ for some nontrivial polynomial $W(z)\in \CC[z]$ and $Y^\vee$ be the SYZ mirror of $Y$ which can be computed to be $$\check{Y} = X_\Delta - D $$ for some toric Calabi-Yau $X_\Delta$ and some divisor $D$ in it.
Then the (naive) SYZ transformation $\mathcal{F}^{SYZ}$ induces an equivalence of triangulated categories
$$ \mathcal{DF}uk_0(Y) \simeq \mathcal{D}^b\mathcal{C}oh_0 (\check{Y}). $$
\end{thm}

\begin{proof}
The proof can be found in Theorem 4.1 and Corollary 4.1 of Chan \cite{Chan} which is the main result of the paper.
\end{proof}

%Guess...\\

Our goal is to modify the naive SYZ transformation for A-branes $(L,\Delta)$ in general case other than just the case of Aganagic-Vafa. From the comparison of $L_{naive}^\vee$ and $L^\vee$, we observe that $L^\vee$ is different to $L_{naive}^\vee$ by 'further quantum correction' by open Gromov-Witten invariants. Analogues to the construction of the B-model $X^\vee$ as in Chapter 2 that we need to modify the semi-flat mirror $(\MM_0,J_0)$ to obtain $(X^\vee,J)$ by 'quantum correction', we have to perform 'further quantum correction' procedure on $L_{naive}^\vee$ to obtain the desired correct mirror $L^\vee$ in general.
At least, at this moment, we see that the modified SYZ transformation works well in the Aganagic-Vafa's case and the $A_n$-resolution case.
%At least we would want such modified SYZ transformation give the true mirror for some particular cases such as Aganagic-Vafa branes in toric Calabi-Yau 3-folds and the example given in Chan's theorem \eqref{chan's thm on HMS}.
By comparing the explicit computation of $L_{naive}^\vee$ and $L^\vee$ in the Aganagic-Vafa's case, we expect that the explicit formula of the defining equations of $L^\vee$ should give some hints on how the SYZ transformation should be modified in general.

Ideally, if such modified SYZ transformation do exist, we would expect that this procedure will give the mirror partner $L^\vee$ for any given A-brane $(L,\nabla)$ in some mirror pair $(X,X^\vee)$. We would also expect the mirror pair $L$ and $L^\vee$ exhibits certain dual properties as the Mirror Conjecture \ref{mirror conjecture} which equates the open Gromov-Witten generating function on the A-side to the B-model superpotential up to the mirror map.

More aggressively, we would expect the modified SYZ transformation with further quantum correction would induce an equivalence mirror functor from $\mathcal{DF}uk(X)$ to $\mathcal{D}^b\mathcal{C}oh(X^\vee)$ for any mirror pair $X$ and $X^\vee$. In other words, we would expect that the homological mirror conjecture can be verified using the principle behind the SYZ Conjecture.
This expectation makes sense since certain cases has been verified as the above Chan's Theorem \ref{chan's thm on HMS}. It is natural to request for such statements that works in more general testing grounds which we would believe to be true.

\end{document}